\documentclass[a4paper,10pt]{amsart}
\usepackage[utf8]{inputenc}
 \usepackage{setspace}
\usepackage[cmtip,arrow]{xy}
\usepackage{lscape}
\usepackage{latexsym}
\usepackage[activeacute,english]{babel}
\usepackage{graphicx}
\usepackage{amsmath}
\usepackage{amsfonts}
\usepackage{amsthm}
\usepackage{amssymb}
\usepackage{amscd}
\usepackage{mathrsfs,bbm}
\usepackage{tikz}
\usepackage{pb-diagram,pb-xy,yfonts}
\xyoption{all}
\usepackage{graphics}

\usepackage{relsize}

\newtheorem{theorem}{Theorem}[section]
\newtheorem{lemma}[theorem]{Lemma}
\newtheorem{proposition}[theorem]{Proposition}
\newtheorem{corollary}[theorem]{Corollary}
\newtheorem{definition}[theorem]{Definition}

\newtheorem{remark}[theorem]{Remark}

\newcommand{\Hom}{\mathrm{Hom}}

\newcommand{\C}{\mathcal{C}}

\newcommand{\T}{\mathcal{T}}
\newcommand{\Mod}{\mathrm{Mod}}

\begin{document}
\title{Triangular Matrix Categories II: Recollements and functorially finite subcategories}

\author{Alicia Le\'on-Galeana, Mart\'in Ortiz-Morales, Valente Santiago Vargas}

\thanks{The author thanks project PAPIIT-Universidad Nacional Aut\'onoma de M\'exico IA105317}
\date{March 9, 2019}
\dedicatory{}
\maketitle

\begin{abstract}
In this paper we continue the study of triangular matrix categories $\mathbf{\Lambda}=\left[ \begin{smallmatrix}
\mathcal{T} & 0 \\ 
M & \mathcal{U}
\end{smallmatrix}\right]$ initiated in \cite{LeOS}. First, given an additive category $\mathcal{C}$ and an ideal $\mathcal{I}_{\mathcal{B}}$ in $\mathcal{C}$, we prove a well known result that there is a canonical recollement $\xymatrix{\mathrm{Mod}(\mathcal{C}/\mathcal{I}_{\mathcal{B}})\ar[r]_{} &  \mathrm{Mod}(\mathcal{C})\ar[r]_{}\ar@<-1ex>[l]_{}\ar@<1ex>[l]_{}  &   \mathrm{Mod}(\mathcal{B})\ar@<-1ex>[l]_{}\ar@<1ex>[l]_{}}$. We show that given a recollement between functor categories we can induce a new recollement between triangular matrix categories, this is a generalization of a result given by  Chen and Zheng in \cite[theorem 4.4]{Chen}. In the case of dualizing $K$-varieties we can restrict the recollement we obtained to the categories of finitely presented functors. Given a dualizing variety $\mathcal{C}$, we describe the maps category of $\mathrm{mod}(\mathcal{C})$ as  modules over a triangular matrix category and we study its Auslander-Reiten sequences and contravariantly finite subcategories, in particular we generalize several results from \cite{MVOM}. Finally, we prove a generalization of a result due to {Smal\o} (\cite[Theorem 2.1]{Smalo}), which give us a way of construct functorially finite subcategories in the category $\mathrm{Mod}\Big(\left[ \begin{smallmatrix}
\mathcal{T} & 0 \\ 
M & \mathcal{U}
\end{smallmatrix}\right]\Big)$ from those of $\mathrm{Mod}(\mathcal{T})$ and $\mathrm{Mod}(\mathcal{U})$.

\end{abstract}
\section{Introduction}
A recollement of abelian categories is an exact sequence of abelian categories where both the inclusion and the quotient functors admit left and right adjoints.  Recollements first appeared in the context of triangulated categories in the construction of the category of perverse sheaves on a singular space by Beilinson, Bernstein and Deligne (see \cite{Belinson}); they were trying to axiomatize the Grothendieck's six functors for derived categories of sheaves.  In representation theory, recollements were used by Cline, Parshall and Scott to study module categories of finite dimensional algebras over a field (see \cite{Parshall}). They appear in connection with quasi-hereditary algebras and highest weight categories.
Recollements of triangulated categories also have appeared in the work of Angeleri-H\"ugel, Koenig and Liu in connection with tilting theory and homological conjectures of derived categories of rings (see \cite{Lidia1}, \cite{Lidia2} and \cite{Lidia3}).\\ 
In the context of abelian categories, recollements were studied by Franjou and Pirashvili in \cite{Franjou}, motivated by the work of MacPherson-Vilonen in derived category of perverse sheaves (see  \cite{Macpherson}).  Several homological properties of recollements of abelian categories have been amply studied (see \cite{Psaro1},  \cite{Psaro2}, \cite{Psaro3}, \cite{Yasuaki}).\\
It should be noted that recollements of abelian categories appear naturally in various settings in representation theory. For example any idempotent element  $e$ in a ring $R$ with unit induces a canonical recollement between the module categories over the rings $R$, $R/ReR$ and $eRe$. In fact, in \cite{Psaro2},  Psaroudakis and  Vitoria, studied recollements of module categories and they showed that a recollement whose terms are module categories is equivalent to one induced by an idempoten element.\\
In the context of comma categories we point out that Chen and Zhen, studied conditions under which a recollement relative to abelian categories induces a new recollement relative to abelian categories and comma categories and they applied their results to deduce results about recollements in categories of modules over triangular matrices rings (see \cite{Chen}).\\
On the other hand, rings of the form $\left[\begin{smallmatrix}
T & 0 \\ 
M & U
\end{smallmatrix}\right]$ where $T$ and $U$ are rings and $M$ is a $T$-$U$-bimodule have appeared often in the study of the representation theory of artin rings and algebras (see for example \cite{AusPlatRei},\cite{DlabRingel}, \cite{Gordon}, \cite{Green1}, \cite{Green2}). Such a rings appear naturally in the study of homomorphic images of hereditary  artin algebras.\\ 
These types of algebras have been amply studied. For example, Zhu
considered the triangular matrix algebra  $\Lambda:=\left[\begin{smallmatrix}
T & 0 \\ 
M & U
\end{smallmatrix}\right]$ where $T$ and $U$ are quasi-hereditary algebras and he proved that under suitable conditions on $M$, $\Lambda$ is quasi-hereditary algebra (see \cite{Bin}).  In the paper \cite{Zhang},  the triangular matrix algebra of rank two was extended to the one of rank $n$ and obtained that there is a relation between the morphism category and the module category of the corresponding matrix algebra.\\
Also, in this direction let us recall the following result due to {Smal\o}: Let $T$ and $U$ be artin algebras, consider the matrix algebra $\Lambda:=\left[\begin{smallmatrix}
T & 0 \\ 
M & U
\end{smallmatrix}\right]$ and denote by $\mathrm{mod}(\Lambda)^{\mathcal{X}}_{\mathcal{Y}}$ the full subcategory of $\mathrm{mod}(\Lambda)$ whose objects are $U$-morphisms $f:M\otimes_{T}A\longrightarrow B$ with $A\in \mathcal{X}$ and $B\in \mathcal{Y}$ where $\mathcal{X}\subseteq \mathrm{mod}(T)$ and $\mathcal{Y}\subseteq \mathrm{mod}(U)$ are subcategories. Then, $\mathrm{mod}(\Lambda)^{\mathcal{X}}_{\mathcal{Y}}$ is functorially finite in $\mathrm{mod}(\Lambda)$ if and only if $\mathcal{X}$ and $\mathcal{Y}$  are functorially finite in $\mathrm{mod}(T)$ and $\mathrm{mod}(U)$ respectively (see \cite[Theorem 2.1]{Smalo}).\\
In the paper \cite{LeOS}, given two additive categories $\mathcal{U}$ and $\mathcal{T}$ and $M\in \mathrm{Mod}(\mathcal{U}\otimes \mathcal{T}^{op})$ we constructed  the matrix category $\mathbf{\Lambda}:=\left[\begin{smallmatrix}
\mathcal{T} & 0 \\ 
M & \mathcal{U}
\end{smallmatrix}\right]$ and we studied several of its properties. In particular we proved that there exists and equivalence of categories $\Big( \mathrm{Mod}(\mathcal{T}), \mathbb{G}\mathrm{Mod}(\mathcal{U})\Big) \simeq
\mathrm{Mod}(\mathbf{\Lambda})$. In this paper, one of the main results is a generalization of the result in \cite[theorem 4.4]{Chen}, that given a recollement between functor categories we can induce a recollement between modules over certain triangular matrix categories $\mathrm{Mod}\Big(\left[ \begin{smallmatrix}
\mathcal{T} & 0 \\ 
M & \mathcal{U}
\end{smallmatrix}\right]\Big)$. We also show that in the case of dualizing $K$-varieties we can restric that recollement to the category of finitely presented modules  $\mathrm{mod}(\mathbf{\Lambda})$ (see \ref{Recoll1} and \ref{Recollfinitos}). Finally, we prove an analogous of the  {Smal\o}'s result mentioned above, but for the context of dualizing varieties (see Theorem \ref{Smalogeneral}).\\
We now give a brief description of the contents on this paper.
\begin{itemize}
\item In section 2, we recall basic concepts and properties of the category $\mathrm{Mod}(\mathcal{C})$,  and some properties of dualizing varieties and comma categories that will be use throughout the paper.

\item In section 3, we recall the notion of recollement and we show that there is a recollement coming from a triple adjoint defined by M. Auslander in  \cite{AuslanderRep1}. This result is well known (see for example \cite[Example 3.12]{Psaro3}), but we give a proof by the convenience of the reader.

\item In section 4, we show how construct recollements in the category of modules over triangular matrix categories $\mathrm{Mod}\Big(\left[ \begin{smallmatrix}
\mathcal{T} & 0 \\ 
M & \mathcal{U}
\end{smallmatrix}\right]\Big)$. In this section we prove a generalization of a theorem due to Chen and Zhen  (\cite[Theorem 4.4]{Chen}), that given a recollement in functor categories we can induce a recollement between modules over triangular matrix categories.

\item In section 5, we study the category $\mathrm{maps}(\mathrm{Mod}(\mathcal{C})):=\Big(\mathrm{Mod}(\mathcal{C}),\mathrm{Mod}(\mathcal{C})\Big)$ of maps of the category $\mathrm{Mod}(\mathcal{C})$ and we give in this setting a description of the functor $\widehat{\Theta}:
\Big(\mathrm{Mod}(\mathcal{T}),\mathbb{G}\mathrm{Mod}(\mathcal{U})\Big)\longrightarrow \Big(\mathrm{Mod}(\mathcal{U}^{op}),\overline{\mathbb{G}}\mathrm{Mod}(\mathcal{T}^{op})\Big)$
constructed in \cite[Proposition 4.9]{LeOS} (see \ref{descrimaps}). We also give a description of the projective and injective objects of the category $\mathrm{maps}(\mathrm{mod}(\mathcal{C}))$  when $\mathcal{C}$ is a dualizing variety and we also describe its radical (see  \ref{projinjec}).

\item In section 6, we study the Auslander-Reiten translate in comma categories. So, we construct
 $(-)^{\ast}:\big(\mathrm{Mod}(\mathcal{T}),\mathbb{G}\mathrm{Mod}(\mathcal{U})\big)\rightarrow \big(\mathrm{Mod}(\mathcal{U}^{op}),\overline{\mathbb{G}}\mathrm{Mod}(\mathcal{T}^{op})\big)$ and we describe how it acts on the category $\mathrm{proj}\big(\mathrm{Mod}(\mathcal{T}),\mathbb{G}\mathrm{Mod}(\mathcal{U})\big)$, of finitely generated projective objects (see \ref{dualproyec}, \ref{descripcionmor} and \ref{funtorestrellita}). We also describe the action of the Auslander-Reiten translate in the category of maps $\mathrm{maps}(\mathrm{mod}(\mathcal{C}))=\Big(\mathrm{mod}(\mathcal{C}),\mathrm{mod}(\mathcal{C})\Big)$ when $\mathcal{C}$ is a dualizing variety (see \ref{ARsequncedesc}).

\item In section 7, we generalize some results from \cite{MVOM}. We consider a dualizing $K$-variety $\mathcal{C}$ and we study $\mathbf{\Lambda}=\left[\begin{smallmatrix}
 \mathcal C& 0 \\ 
 \widehat{\mathbbm{Hom}}& \mathcal C
\end{smallmatrix}\right]$.  We construct almost split sequences  in $\mathrm{mod}(\mathbf{\Lambda})$ that arise from almost split sequences in $\mathrm{mod}(\mathcal C)$ (see \ref{ARseq1} and \ref{ARseq2}). We also study almost split sequences in $ \mathrm{mod}(\mathrm{mod}(\mathcal{C})^{op})$ coming from almost split sequences with certain conditions in $\mathrm{maps}(\mathrm{mod}(\mathcal{C}))$ (see \ref{ARseq3}).

\item In section 8, we study contravariantly and functorially finite subcategories in $\mathrm{Mod}\Big(\left[ \begin{smallmatrix}
\mathcal{T} & 0 \\ 
M & \mathcal{U}
\end{smallmatrix}\right]\Big)$. In particular, we prove a generalization of the result of {Smal\o}  in \cite[Theorem 2.1]{Smalo}, which give us a way of construct functorally finite subcategories in the category $\mathrm{Mod}\Big(\left[ \begin{smallmatrix}
\mathcal{T} & 0 \\ 
M & \mathcal{U}
\end{smallmatrix}\right]\Big)$ of modules over a triangular matrix category (see \ref{Smalogeneral}). Finally, we see that the categories of monomorphisms and epimorphisms in $\mathrm{maps}(\mathrm{mod}(\mathcal{C}))=\Big(\mathrm{mod}(\mathcal{C}),\mathrm{mod}(\mathcal{C})\Big)$  are funtorially finite (see \ref{monoepifun}).

\end{itemize}
\section{Preliminaries}
\subsection{Categorical Foundations and Notations}
We recall that a category $\C$ together with an abelian group structure on each of the sets of morphisms $\C(C_{1},C_{2})$ is called  \textbf{preadditive category}  provided all the composition maps
$\C(C,C')\times \C(C',C'')\longrightarrow \C(C,C'')$
in $ \C $ are bilinear maps of abelian groups. A covariant functor $ F:\C_{1}\longrightarrow \C_{2} $ between  preadditive categories $ \C_{1} $ and $ \C_{2} $ is said to be \textbf{additive} if for each pair of objects $ C $ and $ C' $ in $ \C_{1}$, the map $ F:\C_{1}(C,C')\longrightarrow \C_{2}(F(C),F(C')) $ is a morphism of abelian groups. Let $\mathcal C$ and $\mathcal D$  be preadditive categories and $\mathbf{Ab}$ the category of abelian groups. A functor $F: \mathcal C\times \mathcal D\rightarrow \mathbf{Ab}$ is called \textbf{biadditive} if
$F:\mathcal C(C,C')\times \mathcal D(D,D')\rightarrow \mathbf{Ab} (F(C,D),F(C',D'))$ is bi additive, that is,  $F(f+f',g)=F(f,g)+F(f',g)$ and $F(f,g+g')=F(f,g)+F(f,g')$.\\
If $ \C $ is a  preadditive category we always considerer its opposite  category $ \C^{op}$ as a preadditive category by letting $\C^{op} (C',C)= \C(C,C') $. We follow the usual convention of identifying each contravariant  functor $F$ from a category $ \C $ to $ \mathcal{D} $ with the covariant functor $F$ from  $ \C^{op} $ to $ \mathcal{D}$.

\subsection{The category $\mathrm{Mod}(\mathcal{C})$}
Throughout this section $\mathcal{C}$ will be an arbitrary skeletally small preadditive category, and $\mathrm{Mod}(\mathcal{C})$ will denote the \textit{category of covariant functors} from $\mathcal{C}$ to  the category of abelian groups $ \mathbf{Ab}$, called the category of $\mathcal{C}$-modules. This category has as objects  the functors from $\mathcal C$ to $\mathbf{Ab}$, and  and a morphism $ f:M_{1}\longrightarrow M_{2} $ of $ \C $-modules is a natural transformation,  that is, the set of morphisms $\mathrm{Hom}_\mathcal C(M_1,M_2)$ from $M_1$ to $M_2$  is given by $\mathrm{Nat} (M_{1}, M_{2} )$.  We sometimes we will write for short, $\mathcal{C}(-,?)$
instead of $\mathrm{Hom}_{\mathcal{C}}(-,?)$ and when it is clear from the context we will use just $(-,?).$\\
We now recall some of properties of the category $ \Mod(\C) $, for more details consult  \cite{AuslanderRep1}. The category $\Mod(\C) $ is an abelian with the following properties:
\begin{enumerate}
\item A sequence
\[
\begin{diagram}
\node{M_{1}}\arrow{e,t}{f}
 \node{M_{2}}\arrow{e,t}{g}
  \node{M_{3}}
\end{diagram}
\]
is exact in $ \Mod(\C) $ if and only if
\[
\begin{diagram}
\node{M_{1}(C)}\arrow{e,t}{f_{C}}
 \node{M_{2}(C)}\arrow{e,t}{g_{C}}
  \node{M_{3}(C)}
\end{diagram}
\]
is an exact sequence of abelian groups for each $ C$ in $\C $.

\item Let $ \lbrace M_{i}\rbrace_{i\in I} $ be a family of $ \C $-modules indexed by the set $ I $.
The $ \C $-module $ \underset{\i\in I}\amalg M_{i}$ defined by $  (\underset{i\in I}\amalg M_{i})\ (C)=\underset{i\in I}\amalg \ M_{i}(C)$ for all $ C $ in $ \C $, is a direct sum for the family $ \lbrace M_{i}\rbrace_{i\in I} $ in $ \Mod(\C) $, where $\underset{i\in I} \amalg M_{i}(C)  $ is the direct sum in $ \mathbf{Ab} $ of the family of abelian groups $ \lbrace M_{i}(C)\rbrace_{i\in I} $.  The $ \C $-module $ \underset{\i\in I}\prod M_{i}$ defined by $(\underset{i\in I}\prod M_{i})\ (C)=\underset{i\in I}\prod M_{i}(C)   $ for all $C$ in $\C$, is a product for the family $ \lbrace M_{i}\rbrace_{i\in I} $ in $\Mod(\C)$, where $ \underset{i\in I}\prod M_{i}(C)  $ is the product in $\mathbf{Ab}$.

\item  For each $C$ in $\C $, the $\C$-module $(C,-)$ given by $(C,-)(X)=\C(C,X)$ for each $X$ in $\C$, has the property that for each $\C$-module $M$, the map $\left( (C,-),M\right)\longrightarrow M(C)$ given by $f\mapsto f_{C}(1_{C})$ for each $\C$-morphism $f:(C,-)\longrightarrow M$ is an isomorphism of abelian groups. We will often consider this isomorphism an identification.
Hence
\begin{enumerate}
\item The functor $ P:\C\longrightarrow \Mod(\C) $ given by $ P(C)=(C,-) $ is fully faithful.
\item For each family $\lbrace  C_{i}\rbrace _{i\in I}$ of objects in $ \C $, the $ \C $-module $ \underset{i\in I}\amalg P(C_{i}) $ is a projective $ \C $-module.
\item Given a $ \C $-module $ M $, there is a family $ \lbrace C_{i}\rbrace_{i\in I} $ of objects in $ \C $ such that there is an epimorphism $ \underset{i\in I}\amalg P(C_{i})\longrightarrow M\longrightarrow 0 $.
\end{enumerate}
\end{enumerate}

\subsection{Change of Categories}
The results that appears in this subsection  are directly taken from \cite{AuslanderRep1}. Let $ \mathcal{C} $ be a skeletally small category. There is a unique (up to isomorphism)  functor $ \otimes_{\mathcal{C}}:\mathrm{Mod}(\mathcal{C}^{op})\times \mathrm{Mod}(\mathcal{C})\longrightarrow \mathbf{Ab} $ called the \textbf{tensor product}. The abelian group $\otimes_{\mathcal{C}}(A,B) $  is denoted by $A\otimes_{\mathcal{C}} B $ for all $\mathcal{C}^{op}$-modules $A$ and all $\mathcal{C}$-modules $B$.
\begin{proposition}\label{AProposition0} The tensor product has the following properties:
\begin{enumerate}
\item
\begin{enumerate}
\item For each $ \mathcal{C}$-module $B$, the functor $\otimes_{\mathcal{C}}B:\mathrm{Mod}(\mathcal{C}^{op})\longrightarrow \mathbf{Ab}$ given by $(\otimes_{\mathcal{C}}B)(A)=A\otimes_{\mathcal{C}}B$ for all $\mathcal{C}^{op}$-modules $A$ is right exact.
\item For each $ \mathcal{C}^{op}$-module $A$, the functor $ A\otimes_{\mathcal{C}}:\mathrm{Mod}(\mathcal{C})\longrightarrow \mathbf{Ab}$ given by $(A\otimes_{\mathcal{C}})(B)=A\otimes_{\mathcal{C}}B $ for all $ \mathcal{C}$-modules $B$ is right exact.
\end{enumerate}

\item For each $ \mathcal{C}^{op}$-module $A$ and each $ \mathcal{C}$-module $B$, the functors $A\otimes_{\mathcal{C}}$ and $\otimes_{\mathcal{C}}B $ preserve arbitrary sums.

\item For each object $C$ in $ \mathcal{C} $ we have $ A\otimes_{\mathcal{C}}(C,-) =A(C)$ and $(-,C)\otimes_{\mathcal{C}}B=B(C)$ for all $\mathcal{C}^{op}$-modules $A$ and all $\mathcal{C}$-modules $B$.
\end{enumerate}
\end{proposition}
Suppose now that $ \mathcal{C}' $ is a subcategory of the skeletally small category $ \mathcal{C} $.
 We use the tensor product of $ \mathcal{C}'$-modules, to describe the left adjoint $ \mathcal{C}\otimes_{\mathcal{C}'} $ of the restriction functor $ \mathrm{res}_{\mathcal{C}'}:\mathrm{Mod}(\mathcal{C})\longrightarrow \mathrm{Mod}(\mathcal{C}')$.\\
Define the functor $ \mathcal{C}\otimes_{\mathcal{C}'}:\mathrm{Mod}\left( \mathcal{C}'\right) \longrightarrow \mathrm{Mod}\left( \mathcal{C}\right)  $ by $ (\mathcal{C}\otimes_{\mathcal{C}'}M)\left( C\right) =\left(-,C\right) \mid_{\mathcal{C}'}\otimes_{\mathcal{C}'} M $ for all $ M\in \mathrm{Mod}\left(\mathcal{C}'\right)  $ and $ C\in\mathcal{C}.$ Using the properties of the tensor product it is not difficult to establish the following proposition.

\begin{proposition}
$\textnormal{\cite[Proposition 3.1]{AuslanderRep1}}$ \label{Rproposition1}
Let $ \mathcal{C}' $ a subcategory of the skeletally small category $\mathcal{C}$. Then the functor $ \mathcal{C}\otimes_{\mathcal{C}'}:\mathrm{Mod}\left( \mathcal{C}'\right) \longrightarrow \mathrm{Mod}\left( \mathcal{C}\right)  $ satisfies:

\begin{enumerate}
\item $\mathcal{C}\otimes_{\mathcal{C}'}$ is right exact and preserves sums;

\item The composition $ \mathrm{Mod}\left( \mathcal{C}'\right) \overset{\mathcal{C}\otimes_{\mathcal{C}'}}\longrightarrow \mathrm{Mod}\left( \mathcal{C}\right)\overset{\mathrm{res}_{\mathcal{C}'}}\longrightarrow \mathrm{Mod}\left( \mathcal{C}'\right)    $ is the identity on $ \mathrm{Mod}\left( \mathcal{C}'\right); $

\item For each object $ C'\in \mathcal{C}' $, we have $ \mathcal{C}\otimes_{\mathcal{C}'}\mathcal{C}'\left(C',-\right) =\mathcal{C}\left(C', -\right); $

\item For each $\mathcal{C'}$-module $M$ and each $ \mathcal{C} $-module $ N $, the restriction map \[ \mathcal{C}\left( \mathcal{C}\otimes_{\mathcal{C}'}M,N\right)\longrightarrow \mathcal{C}'\left( M,N\mid_{\mathcal{C}'}\right)\] is an isomorphism;
\item $ \mathcal{C}\otimes_{\mathcal{C}'} $ is a fully faithful functor;

\item $\mathcal{C}\otimes_{\mathcal{C}'}$ preserves projective objects.
\end{enumerate}
\end{proposition}
Having described the left adjoint $ \mathcal{C}\otimes_{\mathcal{C}'} $ of the restriction functor $ \mathrm{res}_{\mathcal{C}'}:\mathrm{Mod}\left(\mathcal{C}\right) \longrightarrow \mathrm{Mod}\left( \mathcal{C}'\right),$ we now describe its right adjoint.\\
Let $ \mathcal{C}' $ be a full subcategory of the category $ \mathcal{C} $. Define the functor $ \mathcal{C}'\left( \mathcal{C},-\right):\mathrm{Mod}\left( \mathcal{C}'\right) \longrightarrow \mathrm{Mod}\left( \mathcal{C}\right)$ by $\mathcal{C}'\left( \mathcal{C},M\right)\left( X\right) =\mathcal{C}'\left( \left(X,-\right) \mid_{\mathcal{C}'},M\right)    $ for all $ \mathcal{C}' $-modules $ M $ and all objects $ X $ in $ \mathcal{C}.$ We have  the following proposition.

\begin{proposition} 
$\textnormal{\cite[Proposition 3.4]{AuslanderRep1}}$ 
\label{Rproposition2}
Let $ \mathcal{C}' $ a subcategory of the skeletally small category $\mathcal{C}$. Then the functor $ \mathcal{C}'\left( \mathcal{C},-\right) :\mathrm{Mod}\left( \mathcal{C}'\right) \longrightarrow \mathrm{Mod}\left( \mathcal{C}\right)$ has the following properties:
\begin{enumerate}

\item $ \mathcal{C}'\left( \mathcal{C},-\right)  $ is left exact and preserves inverse limits;

\item The composition $\mathrm{Mod}\left( \mathcal{C}'\right) \overset{\mathcal{C}'\left( \mathcal{C},-\right) }\longrightarrow \mathrm{Mod}\left( \mathcal{C}\right)\overset{\mathrm{res}_{\mathcal{C}'}}\longrightarrow \mathrm{Mod}\left( \mathcal{C}'\right)    $ is the identity on $\mathrm{ Mod}\left( \mathcal{C}'\right);$

\item For each $\mathcal{C}' $-module $M $ and $\mathcal{C}$-module $N$, the restriction map 
$$\mathcal{C}\left( N,\mathcal{C}'\left( \mathcal{C},M\right) \right) \longrightarrow \mathcal{C}'\left( N\mid_{\mathcal{C}'},M\right)$$ is an isomorphism;

\item $\mathcal{C}'\left( \mathcal{C},-\right)  $ is a fully faithful functor;
\item $\mathcal{C}'\left( \mathcal{C},-\right) $ preserves injective objects.
\end{enumerate}
\end{proposition}

\subsection{Dualizing varietes and Krull-Schmidt Categories}

The subcategory of $\mathrm{Mod}(\mathcal{C})$ consisting of
all finitely generated projective objects, $\mathfrak{p}(\mathcal{C})$, is a skeletally small additive category in which idempotents split, the functor $P:\mathcal{C}\rightarrow \mathfrak{p}(\mathcal{C})$, $P(C)=\mathcal{C}(C,-)$, is fully faithful and induces by restriction $\mathrm{res}:\mathrm{Mod}(\mathfrak{p}(\mathcal{C}))\rightarrow \mathrm{Mod}(\mathcal{C})$, an equivalence of categories. For this reason, we may assume that our categories are skeletally small, additive categories, such that idempotents split. Such categories were called \textbf{annuli varieties} in \cite{Auslander2}, for
short, varieties.\\
To fix the notation, we recall known results on functors and categories that we use through the paper, referring for the proofs to the papers by Auslander and Reiten \cite{Aus, AuslanderRep1, Auslander2}.

\begin{definition}
Let  $\mathcal{C}$ be a variety. We say $\mathcal{C}$ has \textbf{pseudokernels}; if given a map $f:C_1\rightarrow C_0$, there exists a map $g:C_2 \rightarrow C_1$ such that the sequence of morphisms  $\mathcal{C}(-, C_2 )\xrightarrow{(-,g)}\mathcal{C}( -,C_1 )\xrightarrow{(-,f)}\mathcal{C}(-, C_0 )$ is exact in $\mathrm{Mod}(\mathcal{C}^{op})$.
\end{definition}


Now, we recall some results from \cite{Auslander2}.

\begin{definition}
Let $R$ be a commutative artin ring. An $R$-variety $\mathcal{C}$, is a variety such that $\mathcal{C}(C_{1},C_{2})$ is an $R$-module, and the composition is $R$-bilinear. An $R$-variety $\mathcal{C}$ is $\mathbf{Hom}$-\textbf{finite}, if for each pair of objects $C_{1},C_{2}$ in $\mathcal{C},$ the $R$-module $\mathcal{C}(C_{1},C_{2})$ is finitely generated. We denote by $(\mathcal{C},\mathrm{mod}(R))$, the full subcategory of $(\mathcal{C},\mathrm{
\mathrm{Mod}}(R))$ consisting of the $\mathcal{C}$-modules such that; for
every $C$ in $\mathcal{C}$ the $R$-module $M(C)$ is finitely generated. 
\end{definition}

Suppose $\mathcal{C}$ is a Hom-finite $R$-variety. If $M:\mathcal{C}\longrightarrow \mathbf{Ab}$ is a $\mathcal{C}$-module, then for each $C\in \mathcal{C}$ the abelian group $M(C)$ has a structure of $\mathrm{End}_{\mathcal{C}}(C)^{op}$-module and hence as an $R$-module since $\mathrm{End}_{\mathcal{C}}(C)$ is an $R$-algebra. Further if $f:M\longrightarrow M'$ is a morphism of $\mathcal{C}$-modules it is easy to show that $f_{C}:M(C)\longrightarrow M'(C)$ is a morphism of $R$-modules for each $C\in \mathcal{C}$. Then, $\mathrm{\mathrm{Mod}}(\mathcal{C})$ is an $R$-variety, which we identify with the category of
covariant functors $(\mathcal{C},\mathrm{Mod}(R))$. Moreover, the
category $(\mathcal{C},\mathrm{mod}(R))$ is abelian and the inclusion $(\mathcal{C},\mathrm{mod}(R))\rightarrow (\mathcal{C},\mathrm{\mathrm{Mod}}(R))$ is exact.

\begin{definition}
Let $\mathcal{C}$ be a Hom-finite $R$-variety. We denote by $\mathrm{mod}(\mathcal{C})$ the full subcategory of $\mathrm{Mod}(\mathcal{C})$ whose objects are the  $\textbf{finitely presented functors}$.
That is, $M\in \mathrm{mod}(\mathcal{C})$ if and only if, there exists an exact sequence in $\mathsf{Mod}(\mathcal{C})$
$$\xymatrix{P_{1}\ar[r] & P_{0}\ar[r] & M\ar[r] & 0,}$$
where $P_{1}$ and $P_{0}$ are finitely generated projective $\mathcal{C}$-modules. 
\end{definition}
It is easy to see that if $\mathcal{C}$ has finite coproducts, then
a functor $M$ is finitely presented if there exists an exact
sequence
\begin{equation*}
\mathcal{C}( -,C_1 )\rightarrow \mathcal{C}(-, C_0 )\rightarrow M\rightarrow 0
\end{equation*}
It was proved in \cite{Auslander2}  that $\mathrm{mod}(C)$ is abelian if and only if $\mathcal{C}$ has pseudokernels.\\
Consider the functors $\mathbb{D}_{\mathcal{C}^{op}}:(\mathcal{C}^{op},\mathrm{mod}
(R))\rightarrow (\mathcal{C},\mathrm{mod}(R))$, and $\mathbb{D}_{\mathcal{C}}:(\mathcal{C},\mathrm{mod}(R))\rightarrow (\mathcal{C}^{op},\mathrm{mod}(R))$, which are defined as
follows: for any object $C$ in $\mathcal{C}$, $\mathbb{D}(M)(C)=\mathrm{Hom}
_{R}(M(C),I(R/r)) $, with $r$ the Jacobson radical of $R$, and $I(R/r)$ is
the injective envelope of $R/r$. The functor $\mathbb{D}_{\mathcal{C}}$ defines a duality between $(
\mathcal{C},\mathrm{mod}(R))$ and $(\mathcal{C}^{op},\mathrm{mod}(R))$. We know that  since $\mathcal{C}$ is Hom-finite, $\mathrm{mod}(\mathcal{C})$ is a subcategory of $(\mathcal{C},\mathrm{mod}(R))$. Then we have the following definition due to Auslander and Reiten (see \cite{Auslander2}).

\begin{definition}\label{dualizinvar}
An $\mathrm{Hom}$-finite $R$-variety $\mathcal{C}$ is \textbf{dualizing}, if
the functor
\begin{equation}\label{duality}
\mathbb{D}_{\mathcal{C}}:(\mathcal{C},\mathrm{mod}(R))\rightarrow (\mathcal{C}^{op},\mathrm{mod}(R))
\end{equation}
induces a duality between the categories $\mathrm{mod}(\mathcal{C})$ and $
\mathrm{mod}(\mathcal{C}^{op}).$
\end{definition}

It is clear from the definition that for dualizing categories $\mathcal{C}$
the category $\mathrm{mod}(\mathcal{C})$ has enough injectives. To finish, we recall the following definition:

\begin{definition}
An additive category $\mathcal{C}$ is \textbf{Krull-Schmidt}, if every
object in $\mathcal{C}$ decomposes in a finite sum of objects whose
endomorphism ring is local.
\end{definition}

Asumme that $R$ is a commutative ring and $R$ is a dualizing $R$-variety.
Since the endomorphism ring of each object in $\mathcal C$ is an artin algebra, it follows that $\mathcal C$ is a 
Krull-Schmidt category \cite[p.337]{Auslander2}, moreover,  we have that for a dualizing  variety the
finitely presented functors have projective covers \cite[Cor. 4.13]{AuslanderRep1}, \cite[Cor. 4.4]{Krause}.
The following result appears in \cite[Prop. 2.6]{Auslander2}

\begin{theorem}
Let $\mathcal{C}$ a dualizing $R$-variety. Then $\mathrm{mod}(
\mathcal{C})$ is a dualizing variety.
\end{theorem}

\subsection{Tensor Product of Categories}
If $\mathcal{C}$ and $\mathcal{D}$ are additive categories, B. Mitchell defined in \cite{Mitchell} the  $\textbf{tensor product}$  $\mathcal{C}\otimes\mathcal{D} $  of two additive categories, whose objects are those of $\mathcal{C}\times \mathcal{D}$ and the abelian group of morphism from $(C,D)$ to $(C',D')$ is the ordinary tensor product of abelian groups $\mathcal{C}(C,C')\otimes_{\mathbb{Z}}\mathcal{D}(D,D')$. Since that the tensor product of abelian groups is associative and commutative and the composition in $\mathcal{C}$ and $\mathcal{D}$ is $\mathbb{Z}$-bilinear then the bilinear composition in $\mathcal{C}\otimes \mathcal{D}$ is given as follows:
\begin{equation*}
  (f_{2}\otimes g_{2})\circ (f_{1}\otimes g_{1}):=(f_{2}\circ f_{1})\otimes(g_{2}\circ g_{1})
\end{equation*}
for all $f_{1}\otimes g_1\in \mathcal{C}(C,C')\otimes \mathcal{D}(D,D')$ and  $f_{2}\otimes g_2\in\mathcal{C}(C',C'')\otimes \mathcal{D}(D',D'')$.
\bigskip

\subsection{The Path Category of a Quiver}

A quiver $\Delta$ consists of a set of vertices $\Delta_0$  and a set of arrows $\Delta_1$ which is the disjoint union of sets 
$\Delta(x,y)$ where the elements of $\Delta(x,y)$ are the arrows $\alpha:x\rightarrow y$ from the vertex $x$ to the vertex $y$. Given a quiver $\Delta$ its path category $\mathrm{Pth}\Delta$ has as objects the vertices of $\Delta$  and the morphisms $x\rightarrow y$ are paths  from $x$ to $y$ which are by definition the formal compositions $\alpha_n\cdots\alpha_1$ where $\alpha_1$ stars in $x$, $\alpha_n$  ends in $y$ and the end point of $\alpha_i$ coincides with the start point of $\alpha_{i+1}$ for all $i\in\{1,\ldots,n-1\}$. The positive integer $n$ is called the length of the path. There is a path $\xi_x$ of length $0$ for each vertex to itself. The composition in $\mathrm{Pth}\Delta$ of paths of positive length is just concatenations whereas the $\xi_x$ act as identities.

Given a quiver $\Delta$  and a field $K$, an additive  $K$-category $K\Delta$ is associated to $\Delta$ by taking as  objects  of $K\Delta $  the direct sum of indecomposable objects. The indecomposable objects in $K\Delta$ are given by the vertices of $\Delta$ and  given $x,y\in\Delta_0$  the set of maps from $x$ to $y$ is given by the  $K$-vector space with basis the set of all paths from $x$ to $y$. The composition in $K\Delta$ is of course obtained by $K$-linear extension of the composition in $\mathrm{Pth} \Delta$, that is, the product of two composable paths is defined to be the corresponding composition, the product of two non-composable paths is, by definition, zero. In this way we obtain an associative $K$-algebra which has unit element if and only if $\Delta_0$ is finite (the unit element is given by $\sum _{x\in \Delta_0}\xi_x$).\\
In $K\Delta$, we denote by $K\Delta^{+}$ the ideal generated by all arrows and by $(K\Delta^+)^n$ the ideal generated by all paths of length $\ge n$.\\
Given vertices $x,y\in\Delta_0$, a finite linear combination $\sum_w\lambda_ww, \lambda_w\in K$ where $w$ are paths of length $\ge 2$ from $x$ to $y$, is called a relation on $\Delta$. It can be seen that any ideal $I\subset (K\Delta^+)^2$ can be generated, as an ideal, by relations. If  $I$ is generated as an ideal by the set $\{\rho_i| i\}$ of relations, we write $I=\langle \rho_i| i\rangle$.\\
Given a quiver $\Delta=(\Delta_0,\Delta_1)$ a representation $V=(V_x,f_\alpha)$ of $\Delta$ over $K$ is given by vector spaces $V_x$ for all $x\in \Delta_0$, and linear maps $f_\alpha:V_x\rightarrow V_y$, for any arrow $\alpha:x\rightarrow y$.   The category of representations of $\Delta$ is the category with objects the representations, and a morphism of representations  $h=(h_x): V\rightarrow V'$ is given by maps $h_x:V_x\rightarrow V'_x$ $(x\in\Delta_0)$ such that $h_yf_\alpha=f_\alpha'h_x$ for any $\alpha:x\rightarrow y$. The category of representations of $\Delta$ y denoted by $\mathrm{Rep}(\Delta)$.\\
Given a set of relations $\langle \rho_{i}\mid i\rangle$ of $\Delta$,  we denote by $K\Delta/\langle\rho_i| i\rangle$ the path category given by the quiver $\Delta$ and relations $\rho_i$.  The category  of functors $\mathrm{Mod}(K\Delta/\langle \rho_i| i\rangle)=(K\Delta/\langle \rho_i| i\rangle, \mathrm{Mod \ K})$ can be identified with the representations  of $\Delta$ satisfying the relations $\rho_i$ which is denoted by $\mathrm{Rep}(\Delta,\{\rho_i|i\})$,  (see \cite[p. 42]{RingelTame}).

\subsection{Quotient category and the radical of a category}

A  $\textbf{two}$ $\textbf{sided}$ $\textbf{ideal}$  $I(-,?)$ is an additive subfunctor of the two variable functor $\mathcal{C}(-,?):\mathcal{C}^{op}\otimes\mathcal{C}\rightarrow\mathbf{Ab}$ such that: (a) if $f\in I(X,Y)$ and $g\in\mathcal C(Y,Z)$, then  $gf\in I(X,Z)$; and (b)
if $f\in I(X,Y)$ and $h\in\mathcal C(U,X)$, then  $fh\in I(U,Z)$. If $I$ is a two-sided ideal, then we can form the $\textbf{quotient category}$  $\mathcal{C}/I$ whose objects are those of $\mathcal{C}$, and where $(\mathcal{C}/I)(X,Y):=\mathcal{C}(X,Y)/I(X,Y)$. Finally the composition is induced by that of $\mathcal{C}$ (see \cite{Mitchell}). There is a canonical projection functor $\pi:\mathcal{C}\rightarrow \mathcal{C}/I$ such that:
\begin{enumerate}

\item $\pi(X)=X$, for all $X\in \mathcal{C}$.

\item For all $f\in \mathcal{C}(X,Y)$, $\pi(f)=f+I(X,Y):=\bar{f}.$
\end{enumerate}

Based on the Jacobson radical of a ring, we introduce the radical of an additive category. This concept goes back to work of Kelly (see \cite{Kelly}).

\begin{definition}
The (Jacobson) $\textbf{radical}$ of an additive category $\mathcal{C}$ is the two-sided ideal $\mathrm{rad}_{\mathcal C}$ in $\mathcal{C}$ defined by the formula
$$\mathrm{rad}_{\mathcal {C}}(X,Y)=\{h\in\mathcal{C}(X,Y)\mid 1_X-gh\text{ is invertible for any } g\in\mathcal{C}(Y,X)\} $$
for all objects $X$ and $Y$ of $\mathcal{C}$.

\end{definition}
\subsection{Comma Categories}
If $\mathcal{A}$ and $\mathcal{B}$ are abelian categories and $F:\mathcal{A}\rightarrow\mathcal{B}$ is an additive functor. The $\textbf{comma category}$ $(\mathcal{B},F \mathcal{A})$ is the category whose objects are triples $(B,f,A)$ where $f:B\rightarrow FA$; and whose morphisms between the objects $(B,f,A)$ and $(B',f',A')$ are pair $(\beta,\alpha)$ of morphisms in $\mathcal{B}\times\mathcal{A}$ such that the diagram

\[
\begin{diagram}
\node{B}\arrow{e,t}{\beta}\arrow{s,l}{f}
 \node{B'}\arrow{s,r}{f'}\\
\node{FA}\arrow{e,t}{F\alpha}
 \node{FA'}
\end{diagram}
\]
is commutative in $\mathcal{B}$ (see \cite{Robert}).

\section{Recollements in Functor Categories induced by an Auslander's triple adjoint}

We recall some basic definitions. Consider functors $F:\mathcal{C}\rightarrow\mathcal D$ and $G:\mathcal{D}\rightarrow\mathcal{C}$. We say that $F$ is $\textbf{left adjoint}$ to 
$G$ or that $G$ is $\textbf{right adjoint}$ to $F$, and  that $(F,G)$ is an adjoint pair if there is a natural equivalence
$$\eta =\Big\{\eta _{X,Y}:\eta _{X,Y}:\mathrm{Hom}_{\mathcal D}(FX,Y)\rightarrow \mathrm{Hom}_{\mathcal C}(X,GY)\Big\}_{X\in\mathcal C,Y\in\mathcal D }$$
between the functors $\mathrm{Hom}_{\mathcal D}(F(-),-)$ and $\mathrm{Hom}_{\mathcal C}(-,G(-))$. For every
$X\in\mathcal C$ and $Y\in\mathcal D$,  we set $\epsilon_X:=\eta_{X,FX}(1_{FX}):X\rightarrow GFX$ and
 $\delta_Y:=\eta^{-1}_{GY,Y}(1_{GY}):FGY\rightarrow Y$. Moreover, 
 $\epsilon=\{\epsilon_X\}_{X\in\mathcal C}:1_{\mathcal C}\rightarrow GF$ and 
 $\delta=\{\delta_Y\}_{Y\in\mathcal D}: FG\rightarrow 1_{\mathcal D}$ are natural transformations.
 
\begin{definition}\label{leftrigtrecol}
Let $\mathcal{A}$, $\mathcal{B}$ and $\mathcal{C}$ be abelian categories
\begin{itemize}
\item[(a)]  The diagram
$$\xymatrix{\mathcal{C}\ar@<-2ex>[rr]_{i_{\ast}} & & \mathcal{A}\ar@<-2ex>[rr]_{j^{!}}\ar@<-2ex>[ll]_{i^{\ast}} & &  \mathcal{B}\ar@<-2ex>[ll]_{j_{!}}}$$
is a called a  $\textbf{left}$ $\textbf{recollement}$ if the  additive functors $i^*,i_*,j_!$ and $j^{!}$ satisfy the following conditions:
\begin{itemize}
\item[(LR1)] $(i^*,i_*)$ and $(j_!,j^!)$ are adjoint pairs;
\item[(LR2)] $j^!i_*=0$;
\item[(LR3)] $i_*,j_!$ are full embedding functors.
\end{itemize}

\item[(b)] The diagram
$$\xymatrix{\mathcal{C}\ar@<2ex>[rr]^{i_{!}} & & \mathcal{A}\ar@<2ex>[rr]^{j^{\ast}}\ar@<2ex>[ll]^{i^{!}} & &  \mathcal{B}\ar@<2ex>[ll]^{j_{\ast}}}$$
is  called a  $\textbf{right recollement}$ if the additive functors $i_!,i^{!},j^*$  and $j_*$ satisfy  the following 
conditions:
\begin{itemize}
\item[(RR1)] $(i_!,i^!)$ and $(j^*,j_*)$ are adjoint pairs;
\item[(RR2)] $j^*i_!=0$;
\item[(RR3)] $i_!,j_*$ are full embedding functors.
\end{itemize}

\end{itemize}
\end{definition}

\begin{definition}\label{recolldef}
Let $\mathcal{A}$, $\mathcal{B}$ and $\mathcal{C}$ be abelian categories. Then the diagram

$$
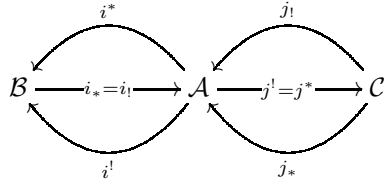
\xymatrix{\mathcal{B}\ar[rr]|{i_{\ast}=i_{!}} & &
\mathcal{A}\ar[rr]|{j^{!}=j^{\ast}}\ar@/^2pc/[ll]^{i^{!}}\ar@/_2pc/[ll]_{i^{\ast}}
& & \mathcal{C}\ar@/^2pc/[ll]^{j_{\ast}}\ar@/_2pc/[ll]_{j_{!}}}$$
is called a $\textbf{recollement}$, if the additive functors $i^{\ast},i_{\ast}=i_{!}, i^{!},j_{!},j^{!}=j^{*}$ and $j_{*}$ satisfy the following conditions:
\begin{itemize}
\item[(R1)] $(i^{*},i_{*}=i_{!},i^{!})$ and $(j_{!},j^{!}=j^{*},j_{*})$ are adjoint  triples, i.e. $(i^{*},i_{*})$, $(i_{!},i^{!})$  $(j_{!},j^{!})$ and $(j^{*},j_{*})$ are adjoint pairs;
\item[(R2)] $j^{*}i_{*}=0$;
\item[(R3)] $i_{*}, j_{!},j_{*}$ are full embedding functors.
\end{itemize}
\end{definition}
By the above definitions we see that a recollement can be seen as the gluing of  a left recollement and a right recollement, and if 
a left recollement and a right recollement satisfy that $i_{*}=i_{!}$ and $j^{!}=j^{*}$ then they can be glued to form a recollement.

Let $\mathcal{C}$ be an a additive category and $\mathcal{B}$ be a full additive subcategory of $\mathcal{C}$.  Maurice Auslander introduced  in \cite{AuslanderRep1} three functors such  that,  according with Propositions \ref{Rproposition1} and \ref{Rproposition2}, together form an adjoint triple
$(\mathcal{C}\otimes_{\mathcal{B}},\mathrm{res}_{\mathcal{B}},\mathcal{B}(\mathcal{C},-) )$

\begin{equation}\label{opo}
 \xymatrix{  \mathrm{Mod}(\mathcal{C}) \ar[rrr]_{\mathrm{res}_{\mathcal B} } &&& \ar@(d,d)[lll]^{\mathcal{B}(\mathcal{C},-)  }  \ar@(u,u)[lll]^{\mathcal{C}\otimes_{\mathcal{B}}} \mathrm{ Mod}(\mathcal{B})}
\end{equation}

In this subsection we show how to extend the adjoint triple (\ref{opo}) to a recollement of functor categories. Similar results are given in \cite{Yasuaki}, but we present them in a slightly different way.\\
Before defining a new triple adjoint, we need some preparatory results. Let $\mathcal{B}$ a full additive  subcategory of $ \mathcal{C} $. For all pair of objects $ C,C'\in\mathcal{C} $ we denote by $ I_{\mathcal{B}}\left( C,C'\right)$ the abelian subgroup of $ \mathcal{C}\left( C,C'\right)  $ whose elements are morphism $ f:C\longrightarrow C' $ which factor through $ \mathcal{B}.$ It is not hard to see that under these conditions $ I_{\mathcal{B}}\left( -,?\right)  $ is a two sided ideal of $\mathcal{C}\left( -,?\right) $.
Thus we can considerer the quotient category $ \mathcal{C}/I_{\mathcal{B}} $.\\
The canonical functor $ \pi:\mathcal{C}\longrightarrow \mathcal{C}/I_{\mathcal{B}} $ induces an exact functor  by restriction $  \mathrm{res}_{ \mathcal{C}}:\mathrm{Mod} \left( \mathcal{C}/I_{\mathcal{ B}}\right)\longrightarrow \mathrm{Mod} \left( \mathcal{C}\right)   $ defined by $  \mathrm{res}_{ \mathcal{C}}(F):=F\circ\pi $ for all $ F\in \mathsf{Mod} \left( \mathcal{C}/I_{\mathcal{ B}}\right) $. Thus,
\begin{itemize}
\item[(i)] $  \mathrm{res}_{ \mathcal{C}}(F)(C) =F\left( C\right)  $ for all $ C\in \mathcal{C} $. 
\item[(ii)] $  \mathrm{res}_{ \mathcal{C}}(F)\left( f\right) =F\left( \overline{f}\right)  $ for all $f\in \mathcal{C}\left( C,C'\right) $.
\end{itemize}

We denote by $ \mathcal{K} $ the full subcategory of $ \mathrm{Mod}(\mathcal{C})$ whose objects are the functors in $ \mathrm{Mod}(\mathcal{C})$ that vanish in $ \mathcal{B}$. That is, $ \mathcal{K}=\lbrace F\in \mathrm{Mod}(\mathcal{C})\mid \mathrm{res}_{\mathcal{B}}(F)=0\rbrace  $. We prove that $ \mathcal{K} $ is isomorphic to the category $ \mathrm{Mod}(\mathcal{C}/I_{\mathcal{B}})$

\begin{lemma} \label{ideal1}
Let $\mathcal{B}$ be a full additive  subcategory of $\mathcal{C}$.
\begin{itemize}
\item[(i)] Let $F$ be a $\mathcal{C}/I_{\mathcal{B}} $-module, then $  \mathrm{res}_{ \mathcal{C}}(F)\in \mathcal{K}$.
\item[(ii)] Let $ G\in \mathcal K $ be, then $ G(f)=0 $ for all $ f\in I_{\mathcal{B}} (C,C')$ and for all $ C,C'\in \mathcal{C} $.
\end{itemize}
\end{lemma}
\begin{proof}
(i) Let $B\in \mathcal{B} $. Then 
\[
1_{ \mathrm{res}_{ \mathcal{C}}(F)(B)}= \mathrm{res}_{ \mathcal{C}}(F)(1_{B})=F(\overline{1_{B}})=F(1_{B}+I_{\mathcal{B}}(B,B))=F(\overline{0})=0.
\]
Thus $  \mathrm{res}_{ \mathcal{C}}(F)(B)=0.$

(ii) It is clear.
\end{proof}
Now, we define a functor $\alpha:\mathcal{K}\longrightarrow  \mathrm{Mod}(\mathcal{C}/I_{\mathcal{ B}}) $ given by $ \alpha(G)(C)=G(C) $ for all $ G\in \mathcal{K} $ and for all $ C\in \mathcal{C}$ and also $ \alpha(G)(\overline{f})=G(f) $ for all $ \overline{f}=f+I_{\mathcal{B}}(C,C')\in \mathcal{C}/I_{\mathcal{ B}}(C,C')$.
It is clear that $ \alpha(G)$ is well defined in morphisms. Indeed, if $ \overline{f}=\overline{g}\in \mathcal{C}/I_{\mathcal{ B}}(C,C') $ then $ f-g\in I_{\mathcal{B}}(C,C') $ therefore by Lemma \ref{ideal1}, we get that  $  0=G(f-g)=G(f)-G(g) $, thus $ \alpha(G)(\overline{f})=\alpha(G)(\overline{g})$.

\begin{lemma}\label{recollF}
The functors $\mathrm{res}_{ \mathcal{C}}:\mathrm{Mod}(\mathcal{C}/I_{\mathcal{B}})\longrightarrow \mathrm{Mod}(\mathcal{C}) $ and $ \alpha:\mathcal{K}\longrightarrow \mathrm{Mod}(\mathcal{C}/I_{\mathcal{B}}) $ induce an isomorphism of categories between the categories $\mathrm{Mod}(\mathcal{C}/I_{\mathcal{B}}) $ and $ \mathcal{K}$. In this way $  \mathrm{res}_{ \mathcal{C}} $ is a full embedding functor which essential image is $\mathcal{K}$, and $\mathcal{K}=\mathrm{Ker}( \mathrm{res}_{ \mathcal{B}})$.
\end{lemma}
\begin{proof}
Let $ F\in \mathrm{Mod}(\mathcal{C}/I_{\mathcal{B}}) $. The for all $ C\in \mathcal{C} $ we have $(\alpha\circ  \mathrm{res}_{ \mathcal{C}})(F)(C)=F(C)$ and for all $ \overline{f}\in \mathcal{C}/I_{\mathcal{B}}(C,C') $ we obtain 

\begin{align*}
(\alpha\circ  \mathrm{res}_{ \mathcal{C}})(F)(\overline{f}) & = \alpha( \mathrm{res}_{ \mathcal{C}}(F))(\overline{f})\quad\quad  \quad (\mathrm{res}_{ \mathcal{C}}(F)\in \mathcal{K}\,\,\, by\,\,\, \ref{ideal1}) \\
& = \mathrm{res}_{ \mathcal{C}}(F)(f)\\
& = F(\overline{f}).
\end{align*}
On the other hand, let $ G\in \mathcal{K}$ be, then for all $ C\in \mathcal{C}/I_{\mathcal{B}} $ we have
 $( \mathrm{res}_{ \mathcal{C}}\circ\alpha)(G)(C)= \mathrm{res}_{ \mathcal{C}}(\alpha(G))(C)= \mathrm{res}_{ \mathcal{C}}(G(C))=G(C)$
and for all $ f\in \mathrm{Hom}_{\mathcal{C}}(C,C')$

\begin{align*}
( \mathrm{res}_{ \mathcal{C}}\circ\alpha)(G)(f)&= \mathrm{res}_{ \mathcal{C}}(\alpha(G))(f) = (\alpha(G))(\overline{f})= G(f).
\end{align*}
It follows that $ \alpha\circ \mathrm{res}_{ \mathcal{C}}=1_{\mathrm{Mod}(\mathcal{C})} $ and $  \mathrm{res}_{\mathcal{C}}\circ\alpha=1_{\mathcal{K}} $. The rest of the proof is clear.
\end{proof}
\bigskip

Now we will construct a triple adjoint $ \left( \mathcal{C}/I_{\mathcal{B}}\otimes_{\mathcal{C}}, \mathrm{res}_{\mathcal{C}},\mathcal{C}(\mathcal{C}/I_{\mathcal{B},-})\right)$:

\begin{equation}\label{op}
 \xymatrix{\mathrm{Mod}(\mathcal{C}/I_{\mathcal{B}}) \ar[rrr]_{\mathrm{res}_{\mathcal{C}} } &&& \ar@(d,d)[lll]^{\mathcal{C}( \mathcal{C}/I_{\mathcal{B}},-)}  \ar@(u,u)[lll]^{\mathcal{C}/I_{\mathcal{B}}\otimes_{\mathcal{C}}} \mathrm{Mod}(\mathcal{C}).}
\end{equation}
In order to construct this, we will need some preparatory results.

\begin{lemma}\label{cociente1}
Let $C,C'\in\mathcal C$. 
\begin{itemize}
\item[(i)] Asume $ f\in \mathrm{Hom}_{\mathcal{C}}(C,C') $. Then we have a morphism of $ \mathcal{C}^{op} $-modules 
\[
\frac{\mathcal{C}}{I_{\mathcal{B}}}(-,f):\frac{\mathcal{C}(-,C)}{I_{\mathcal{B}}(-,C)}\longrightarrow \frac{\mathcal{C}(-,C')}{I_{\mathcal{B}}(-,C')} 
\]
such that for all $ Z\in \mathcal{C} $ 
\[
\frac{\mathcal{C}}{I_{\mathcal{B}}}(Z,f):\frac{\mathcal{C}(Z,C)}{I_{\mathcal{B}}(Z,C)}\longrightarrow \frac{\mathcal{C}(Z,C')}{I_{\mathcal{B}}(Z,C')}, \quad\quad \overline{h}\longmapsto  \overline{h\circ f}
\]
\item[(ii)]  $\overline{f}=f+I_{\mathcal{B}}(C,C')=f'+I_{\mathcal{B}}(C,C') \in\mathcal{C}(C,C')/I_{\mathcal{B}}(C,C') $ implies  that $\frac{\mathcal{C}}{I_{\mathcal{B}}}(-,f)=\frac{\mathcal{C}}{I_{\mathcal{B}}}(-,f')$.
\end{itemize}
The same holds in the category of $\mathcal C$-modules.
\end{lemma}
\begin{proof}
Straightforward.
\end{proof}

\begin{definition}
We define the functor $ \frac{\mathcal{C}}{I_{\mathcal{B}}}\otimes_{\mathcal{C}}:\mathrm{Mod}(\mathcal{C})\longrightarrow\mathrm{Mod}(\mathcal{C}/I_{\mathcal{B}}) $ as follows: 
$\left( \frac{\mathcal{C}}{I_{\mathcal{B}}}\otimes _{\mathcal{C}}M\right)(C):= \frac{\mathcal{C}(-,C)}{I_{\mathcal{B}}(-,C)}\otimes_{\mathcal{C}}M$  for all $M\in \mathrm{Mod}(\mathcal{C})$
and  $\left( \frac{\mathcal{C}}{I_{\mathcal{B}}}\otimes _{\mathcal{C}}M\right)(\overline{f})=\frac{\mathcal{C}}{I_{\mathcal{B}}}(-,f)\otimes_{\mathcal{C}}M$ for all $\overline{f}=f+I_{\mathcal{B}}(C,C')\in \frac{\mathcal{C}(C,C')}{I_{\mathcal{B}}(C,C')}$.
\end{definition}

So we establish the following proposition.

\begin{proposition}\label{Rproposition3}
Let $\mathcal C$ be an additive category and $\mathcal B$ be a full additive subcategory of $\mathcal{C}$. Then the functor ${\mathcal{C}/I_{\mathcal{B}}}\otimes_{\mathcal{C}}:\mathrm{Mod}\left( \mathcal{C}\right) \longrightarrow \mathrm{Mod}\left( {\mathcal{C}/I_{\mathcal{B}}}\right) $ satisfies:
\begin{itemize}
\item[(i)] $\frac{\mathcal{C}}{I_{\mathcal{B}}}\otimes_{\mathcal{C}} $ is right exact and preserves sums.
\item [(ii)]$ \frac{\mathcal{C}}{I_{\mathcal{B}}}\otimes_{\mathcal{C}}\mathcal{C}\left( C,-\right)  = \frac{\mathcal{C}(C,-)}{I_{\mathcal{B}}(C,-)} $.
\item[(iii)] For all $M\in\mathrm{Mod}(\mathcal C)$ and $N\in\mathrm{Mod}(\mathcal C/I_{\mathcal B})$ there exists a natural isomorphism $ \left( \frac{\mathcal{C}}{I_{\mathcal{B}}}\otimes_{\mathcal{C}}M,N \right)\cong \left( M,\mathrm{res}_{\mathcal C}(N)\right)$.
\end{itemize}
\end{proposition}
\begin{proof}
(i) Let $M\longrightarrow N\longrightarrow L\longrightarrow 0$ be an exact sequence of $ \mathcal{C} $-modules.  Let us consider the $\mathcal{C}^{op}$-module $\frac{\mathcal{C}}{I_{\mathcal{B}}}(-,C)$, by \ref{AProposition0} we have that $ \frac{\mathcal{C}}{I_{\mathcal{B}}}(-,C)\otimes_{\mathcal{C}}-:\mathrm{Mod}(\mathcal{C})\longrightarrow \mathbf{Ab} $ is a right exact functor- Then, we obtain the following exact sequence
$$\frac{\mathcal{C}}{I_{\mathcal{B}}}(-,C)\otimes_{\mathcal{C}}M\longrightarrow\frac{\mathcal{C}}{I_{\mathcal{B}}}(-,C)\otimes_{\mathcal{C}}N\longrightarrow\frac{\mathcal{C}}{I_{\mathcal{B}}}(-,C)\otimes_{\mathcal{C}}L\longrightarrow 0,
$$
that is,
$$
\left( \frac{\mathcal{C}}{I_{\mathcal{B}}}\otimes_{\mathcal{C}}M\right) (C)\longrightarrow\left( \frac{\mathcal{C}}{I_{\mathcal{B}}}\otimes_{\mathcal{C}}N\right) (C)\longrightarrow\left( \frac{\mathcal{C}}{I_{\mathcal{B}}}\otimes_{\mathcal{C}}L\right) (C)\longrightarrow 0
$$
is exact for all $ C\in \mathcal{C}/I_{\mathcal{B}}$. The rest follows from the fact that $ \frac{\mathcal{C}(-,C)}{I_{\mathcal{B}}(-,C)}\otimes_{\mathcal{C}}-:\mathrm{Mod}(\mathcal{C})\longrightarrow \mathbf{Ab} $ preserves sums, for all $ C\in \mathcal{C}/I_{\mathcal{B}}$.

(ii) It follows from Proposition \ref{AProposition0} (3). Indeed,
 we have that $\Big(\frac{\mathcal{C}}{I_{\mathcal{B}}}\otimes_{\mathcal{C}}\mathcal{C}\left( C,-\right)\Big)(C')=\frac{\mathcal{C}(-,C')}{I_{\mathcal{B}}(-,C')}\otimes_{\mathcal{C}}\mathcal{C}\left( C,-\right)=\frac{\mathcal{C}(C,C')}{I_{\mathcal{B}}(C,C')}=\frac{\mathcal{C}(C,-)}{I_{\mathcal{B}}(C,-)}(C').$ In the same way, they coincide in morphisms.

(iii) Assume  that there is an exact sequence 
\begin{equation}\label{recollB}
\underset{i\in I}\coprod\mathcal{C}\left( C_{i},-\right) \longrightarrow \underset{j\in J}\coprod\mathcal{C}\left(C_{j},-\right)\longrightarrow M\longrightarrow 0
\end{equation} 
of $ \mathcal{C} $-modules with the $ C_{i} $ and $ C_{j} $ in the objects of $ \mathcal{C}. $ By (i) and (ii)  we get an exact sequence of $ \frac{\C}{I_{\mathcal B}}$-modules.
\begin{eqnarray}\label{recollA}
\underset{i\in I}\coprod \frac{\C(C_{i},-)}{I_{\mathcal B}(C_{i},-)} \longrightarrow \underset{j\in J}\coprod \frac{\C(C_{j},-)}{I_{\mathcal B}(C_{j},-)} \longrightarrow  \frac{\C}{I_{\mathcal B}}\otimes_{\mathcal{C}}M\longrightarrow 0 
\end{eqnarray}

Then, after applying $ (-,N) $ to (\ref{recollA})  and $(-,\mathrm{res}_\mathcal C(N))$ to (\ref{recollB}) respectively we get,  by Yoneda's Lemma  in the category of $\mathcal C/I_{\mathcal B}$-modules, the following commutative diagram.

$$\xymatrix{0\ar[r] & (\frac{\mathcal{C}}{I_{\mathcal{B}}}\otimes_{\mathcal{C}}M,N)\ar[r]\ar@{=}[d]& (\underset{j\in J}\coprod \frac{(C_{j},-)}{I_{\mathcal{B}}(C_{j},-)},N)\ar[r]\ar[d]^{\cong}& ( \underset{i\in I}\coprod \frac{(C_{i},-)}{I_{\mathcal{B}}(C_{i},-)} ,N)\ar[d]^{\cong}\\ 
0\ar[r] & ( \frac{\mathcal{C}}{I_{\mathcal{B}}}\otimes_{\mathcal{C}}M,N )\ar[r]\ar@{=}[d] & \prod_{j\in J}\big( \frac{(C_{j},-)}{I_{\mathcal{B}}(C_{j},-)},N\big)\ar[r]\ar[d]^{\cong} & \prod_{i\in I}\big(\frac{(C_{i},-)}{I_{\mathcal{B}}(C_{i},-)} ,N\big)\ar[d]^{\cong}\\
0\ar[r] & ( \frac{\C}{I_{\mathcal B}}\otimes_{\mathcal{C}}M,N )\ar[r]\ar[d] & \prod_{j\in J}N(C_j)\ar[r]\ar[d]^{\cong} & \prod_{i\in I}N(C_i)\ar[d]^{\cong}\\
0\ar[r] & (M,\mathrm{res}_{\mathcal{C}}(N))\ar[r] & \big(\underset{j\in J}\coprod\left(C_{i},-\right),\mathrm{res}_{\mathcal{C}}(N)\big)\ar[r] &  \big(\underset{i\in I}\coprod\left( C_{i},-\right),\mathrm{res}_{\mathcal{C}}(N) \big)}$$
\end{proof}

Finally we define the right adjoint adjoint $\mathcal{C}(\frac{\mathcal{C}}{I_{\mathcal{B}}},-)$ of the functor  $\mathrm{res}_{\mathcal{C}}:\mathrm{Mod} (\mathcal{C}/I_{\mathcal{ B}})\longrightarrow \mathrm{Mod}(\mathcal{C})$.\\

\begin{definition}
We define the functor $\mathcal{C}(\frac{\mathcal{C}}{I_{\mathcal{B}}},-):\mathrm{Mod}\left(\mathcal{C}\right) \longrightarrow \mathrm{Mod}\left(\mathcal{C}/I_{\mathcal{B}}\right) $ by: 
$\mathcal{C}(\frac{\mathcal{C}}{I_{\mathcal{B}}},M)(C)=\mathcal{C}\left(\frac{\mathcal{C}(C,-)}{I_{\mathcal{B}}(C,-)},M\right)$ for all $ M\in \mathrm{Mod}(\mathcal{C})$ and  $C\in \C/I_{\mathcal{B}}$ and 
$\mathcal{C}(\frac{\mathcal{C}}{I_{\mathcal{B}}},M)(\overline{f})=\mathcal{C}\left( \frac{\mathcal{C}}{I_{\mathcal{B}}}(f,-),M\right) 
$ for all $\overline{f}=f+I_{\mathcal{B}}(C,C')\in \C(C,C')/I_{\mathcal{B}}(C,C') $ with $C,C'$ in $\mathcal{C}$.
\end{definition}

So we establish the following proposition.
\begin{proposition}\label{Rproposition4}
Let $\mathcal{C}$ be an a additive category and $\mathcal{B}$ be a full additive subcategory of $\mathcal{C}$. Then the functor $\mathcal{C}(\frac{\mathcal{C}}{I_{\mathcal{B}}},-):\mathrm{Mod}\left( \mathcal{C}\right) \longrightarrow \mathrm{Mod}\left( \mathcal{C}/I_{\mathcal{B}}\right) $ satisfies:
\begin{itemize}
\item[ (i)] $\mathcal{C}(\frac{\mathcal{C}}{I_{\mathcal{B}}},-)$ is left exact.
\item[(ii)] For all $M\in\mathrm{Mod}(\mathcal{C})$ and $N\in\mathrm{Mod}(\mathcal{C}/I_{\mathcal{B}})$ there exists a natural isomorphism  $\frac{\mathcal{C}}{I_{\mathcal{B}}}\Big(N,\mathcal{C}(\frac{\mathcal{C}}{I_{\mathcal{B}}},M)\Big) \longrightarrow \mathcal{C}\Big(\mathrm{res}_{\mathcal{C}}(N),M\Big)$.
\end{itemize}
\end{proposition}

\begin{proof}
(i) Let $0\longrightarrow M\longrightarrow N\longrightarrow L$ a exact sequence of $ \mathcal{C} $-modules. Since $\mathcal{C}\left(\frac{\mathcal{C}(X,-)}{I_{\mathcal{B}}(X,-)},- \right) 
$ is a left exact functor, for all $ X\in \mathcal{C}/I_{\mathcal{B}}$, then we have the following exact sequence
\[
0\longrightarrow\mathcal{C}\left(\frac{\mathcal{C}(X,-)}{I_{\mathcal{B}}(X,-)},M\right)\longrightarrow \mathcal{C}\left(\frac{\mathcal{C}(X,-)}{I_{\mathcal{B}}(X,-)},N \right)\longrightarrow\mathcal{C}\left(\frac{\mathcal{C}(X,-)}{I_{\mathcal{B}}(X,-)},L \right)
\]
(ii) Let $ N $ be a $ \mathcal{C}/I_{\mathcal{B}} $-module. Then exist an exact sequence

\begin{eqnarray}\label{recollC}
\underset{i\in I}\coprod \frac{\mathcal{C}(X_i,-)}{I_{\mathcal{B}}(X_i,-)}\longrightarrow \underset{j\in J}\coprod\frac{\mathcal{C}(X_j,-)}{I_{\mathcal{B}}(X_j,-)}\longrightarrow N\longrightarrow 0 
\end{eqnarray}

of $ \mathcal{C}/I_{\mathcal{B}}$-modules with the $ X_{i},X_{j} \in\mathcal{C}/I_{\mathcal{B}} $. First note that for all family $\{X_i\}_{i\in I}$ of objects in $\mathcal C$  it follows from the Yoneda's  Lemma that 
\begin{eqnarray}\label{recollE}
\frac{\mathcal{C}}{I_{\mathcal{B}}}\Big(\underset{i\in I}\coprod \frac{\mathcal{C}}{I_{\mathcal{B}}}\big(X_{i},-\big),\mathcal{C}\big(\frac{\mathcal{C}}{I_{\mathcal{B}}},M\big)\Big) 
&\cong&\prod_{i\in I}\frac{\mathcal{C}}{I_{\mathcal{B}}}\Big(\frac{\mathcal{C}}{I_{\mathcal{B}}}\big(X_{i},-\big),\mathcal{C}\big(\frac{\mathcal{C}}{I_{\mathcal{B}}},M\big)\Big)\\ 
                          &\cong&\prod_{i\in I}\mathcal{C}\Big(\frac{\mathcal{C}}{I_{\mathcal{B}}},M\Big)(X_i)\notag\quad (\text{Yoneda's lemma})\\
                             &\cong& \prod_{i\in I}\mathcal{C}\Big(\frac{\mathcal{C}(X_i,-)}{I_{\mathcal{B}}(X_i,-)},M\Big)\notag\quad (\text{def. of}\,\, \mathcal{C}(\frac{C}{I_{\mathcal{B}}},-))\\
                              &\cong&\mathcal{C}\Big(\coprod_{i\in I}\frac{\mathcal{C}(X_i,-)}{I_{\mathcal{B}}(X_i,-)},M\Big)\notag
                              \end{eqnarray}
                              
On the other hand, we apply the  functor $\mathrm{res}_{\mathcal{C}} $ to (\ref{recollC}) we have the follow exact sequence of $\mathcal{C}$-modules
\begin{equation}\label{recollD}
\underset{i\in I}\coprod \frac{\mathcal{C}(-,X_i)}{I_{\mathcal{B}}(-,X_i)}\longrightarrow \underset{j\in J}\coprod\frac{\mathcal{C}(-,X_j)}{I_{\mathcal{B}}(-,X_j)}\longrightarrow \mathrm{res}_{\mathcal{C}}(N)\longrightarrow 0
\end{equation}
                              
Then after applying  $\frac{\mathcal{C}}{I_{\mathcal{B}}}\Big(-,\mathcal{C}(\frac{\mathcal{C}}{I_{\mathcal{B}}},M)\Big) $ to (\ref{recollC}), $\mathcal{C}(-,M)$ to (\ref{recollD})  and using (\ref{recollE})
 we have the following commutative diagram
 
$$\xymatrix{0\ar[r] & \frac{\mathcal{C}}{I_{\mathcal{B}}}\Big(N, \mathcal{C}(\frac{\mathcal{C}}{I_{\mathcal{B}}},M) \Big)\ar[r]\ar[d] & 
\prod_{j\in I}\mathcal{C}(\frac{\mathcal{C}}{I_{\mathcal{B}}},M)(X_j)\ar[r]\ar[d]^{\cong} & 
\prod_{i\in I}\mathcal{C}(\frac{\mathcal{C}}{I_{\mathcal{B}}},M)(X_i)\ar[d]^{\cong}\\
0\ar[r] & \mathcal{C}(\mathrm{res}_\mathcal C(N),M)\ar[r] & \prod_{i\in I}\mathcal{C}\big(\frac{(X_{j},-)}{I_{\mathcal B}(X_{j},-)} ,M\big)\ar[r] & \prod_{i\in I}\mathcal{C}\big(\frac{(X_{i},-)}{I_{\mathcal B}(X_{i},-)} ,M\big)}$$
 \end{proof}
 
Now we are ready to prove the main result of this section.

\begin{theorem}\label{recollementfunctor}
Let $\mathcal{C}$ be an a additive category and $\mathcal{B}$ be a full additive subcategory of $\mathcal{C}$. Then there is a recollement:
\begin{displaymath} 
\xymatrix{ \mathrm{Mod}(\mathcal{C}/I_{\mathcal{B}}) \ar[rrr]_{\mathrm{res}_{\mathcal C}} &&& \ar@(d,d)[lll]^{\mathcal{C}(\mathcal{C}/I_{\mathcal{B}},-) }  \ar@(u,u)[lll]^{\mathcal{C}/I_{\mathcal{B}}\otimes_{\mathcal{C}} }  \mathrm{Mod}(\mathcal{C}) \ar[rrr]_{\mathrm{res}_{\mathcal B} } &&& \ar@(d,d)[lll]^{\mathcal{B}(\mathcal{C},-)  }  \ar@(u,u)[lll]^{\mathcal{C}\otimes_{\mathcal{B}}} \mathrm{ Mod}(\mathcal{B})}
\end{displaymath}
\end{theorem}
\begin{proof}
\begin{itemize}
\item[(R1)] By Propositions \ref{Rproposition1}, \ref{Rproposition2}, \ref{Rproposition3} and \ref{Rproposition4}, the triples  $ \Big(\mathcal{C}\otimes_{\mathcal{B}},\mathrm{res}_\mathcal B,\mathcal{B}(\mathcal{C},-)\Big) $ and $ \Big(\mathcal{C}/I_{\mathcal{B}}\otimes_{\mathcal{C}}, \mathrm{res}_\mathcal C,\mathcal{C}(\mathcal{C}/I_{\mathcal{B}},-)\Big)$ are adjoint triples.
\item [(R2)]  By Lemma \ref{recollF}, we have $ \mathrm{res}_\mathcal B \mathrm{res}_\mathcal C=0$.
\item [(R3)] By Lemma \ref{recollF} and Propositions \ref{Rproposition1} and \ref{Rproposition2},  $ \mathrm{res}_{\mathcal C} $, $ \mathcal{C}\otimes_{\mathcal{B}} $ and $ \mathcal{B}(\mathcal{C},-) $ are full embedding functors.

\end{itemize}
\end{proof}

It is worth to mention that in the context of dualizing $K$-varieties,
Y, Ogawa in \cite[Theorem 2.5]{Yasuaki} have proved the following result.

\begin{theorem}
Let $\mathcal{C}$ be a dualizing $K$-variety and $\mathcal{B}$ a functorially finite subcategory of $\mathcal{C}$. Then the recollement in \ref{recollementfunctor}, restrics to a recollement:
\begin{displaymath} 
\xymatrix{ \mathrm{mod}(\mathcal{C}/I_{\mathcal{B}}) \ar[rrr]_{\mathrm{res}_{\mathcal C}} &&& \ar@(d,d)[lll]^{\mathcal{C}(\mathcal{C}/I_{\mathcal{B}},-) }  \ar@(u,u)[lll]^{\mathcal{C}/I_{\mathcal{B}}\otimes_{\mathcal{C}} }  \mathrm{mod}(\mathcal{C}) \ar[rrr]_{\mathrm{res}_{\mathcal B} } &&& \ar@(d,d)[lll]^{\mathcal{B}(\mathcal{C},-)  }  \ar@(u,u)[lll]^{\mathcal{C}\otimes_{\mathcal{B}}} \mathrm{mod}(\mathcal{B})}
\end{displaymath}
\end{theorem}

\section{Another Recollement}
Our purpose  in this section is to prove the following Theorem which generalizes the results given by Q. Chen and M. Zheng in \cite[Theo. 4.4]{Chen}.

\begin{theorem}\label{Recoll1}
Let $\mathcal {R,S, C}$ and $\mathcal{T}$ be aditive categories. For any $M\in\mathrm{Mod}(\mathcal{R}\otimes \mathcal{T}^{op})$,  consider the matrix categories $\mathbf{\Lambda}:=\left(\begin{smallmatrix}
\mathcal{T} &0\\
M& \mathcal{R}
\end{smallmatrix}\right )$  $\mathbf{\Lambda}^!:=\left (\begin{smallmatrix}
\mathcal{T} &0\\
j_{!}(M)& \mathcal{S}
\end{smallmatrix}\right )$, $\mathbf{\Lambda}^*:=\left (\begin{smallmatrix}
\mathcal{T} &0\\
j_{*}(M)& \mathcal{S}
\end{smallmatrix}\right ),$  where the bimodules $j_{!}(M)$ and $j_{\ast}(M)$ are constructed as in \ref{otrobimodulo}.
\begin{enumerate}
\item [(a)]
If the diagram 
$$\xymatrix{\mathrm{Mod}(\mathcal{C})\ar@<-2ex>[rr]_{i_{\ast}} & & \mathrm{Mod}(\mathcal{S})\ar@<-2ex>[rr]_{j^{!}}\ar@<-2ex>[ll]_{i^{\ast}} & &  \mathrm{Mod}(\mathcal{R})\ar@<-2ex>[ll]_{j_{!}}}$$
is a left recollement, then there is a left recollement
$$\xymatrix{\mathrm{Mod}(\mathcal{C})\ar@<-2ex>[rr]_{\tilde{i_{\ast}}} & & \mathrm{Mod}(\mathbf{\Lambda}^{!})\ar@<-2ex>[rr]_{\tilde{j^{!}}}\ar@<-2ex>[ll]_{\tilde{i^{\ast}}} & &  \mathrm{Mod}(\mathbf{\Lambda})\ar@<-2ex>[ll]_{\tilde{j_{!}}}}$$

\item [(b)] If the diagram 
$$\xymatrix{\mathrm{Mod}(\mathcal{C})\ar@<2ex>[rr]^{i_{!}} & & \mathrm{Mod}(\mathcal{S})\ar@<2ex>[rr]^{j^{\ast}}\ar@<2ex>[ll]^{i^{!}} & &  \mathrm{Mod}(\mathcal{R})\ar@<2ex>[ll]^{j_{\ast}}}$$
is a right recollement, then there is a right recollement
$$\xymatrix{\mathrm{Mod}(\mathcal{C})\ar@<2ex>[rr]^{\tilde{i_{!}}} & & \mathrm{Mod}(\mathbf{\Lambda}^{\ast})\ar@<2ex>[rr]^{\tilde{j^{\ast}}}\ar@<2ex>[ll]^{\tilde{i^{!}}} & & \mathrm{Mod}(\mathbf{\Lambda}).\ar@<2ex>[ll]^{\tilde{j_{\ast}}}}$$
\end{enumerate}
\end{theorem}
\bigskip

We adapt the arguments given in  \cite{Chen} to prove Theorem \ref{Recoll1}. Thus,  we first recall some notation  and results of \cite{LeOS}.\\

In \cite{LeOS} the notion of triangular matrix category was introduced. For convenience of the reader, we recall briefly these concepts.  Let $\mathcal{U,T}$ additive categories and $M\in\mathrm{Mod}(\mathcal U\otimes\mathcal {T}^{op})$,  the \textbf{triangular matrix category} $\mathbf{\Lambda}=\left[ \begin{smallmatrix}
\mathcal{T} & 0 \\ M & \mathcal{U}
\end{smallmatrix}\right]$ is defined  as follows.

\begin{enumerate}
\item [(a)] The class of objects of this category are matrices $ \left[
\begin{smallmatrix}
T & 0 \\ M & U
\end{smallmatrix}\right]  $ where the objects $ T $ and $U$ are in  $\mathcal{T} $ and $\mathcal{U} $ respectively.

\item [(b)] Given a pair of objects in
$\left[ \begin{smallmatrix}
T & 0 \\
M & U
\end{smallmatrix} \right] ,  \left[ \begin{smallmatrix}
T' & 0 \\
M & U'
\end{smallmatrix} \right]$ in
$\mathbf{\Lambda}$ , 

$$\mathsf{ Hom}_{\mathbf{\Lambda}}\left (\left[ \begin{smallmatrix}
T & 0 \\
M & U
\end{smallmatrix} \right] ,  \left[ \begin{smallmatrix}
T' & 0 \\
M & U'
\end{smallmatrix} \right]  \right)  := \left[ \begin{smallmatrix}
\mathsf{Hom}_{\mathcal{T}}(T,T') & 0 \\
M(U',T) & \mathsf{Hom}_{\mathcal{U}}(U,U')
\end{smallmatrix} \right].$$
\end{enumerate}
The composition is given by
\begin{eqnarray*}
\circ&:&\left[  \begin{smallmatrix}
{\mathcal{T}}(T',T'') & 0 \\
M(U'',T') & {\mathcal{U}}(U',U'')
\end{smallmatrix}  \right] \times \left[
\begin{smallmatrix}
{\mathcal{T}}(T,T') & 0 \\
M(U',T) & {\mathcal{U}}(U,U')
\end{smallmatrix} \right]\longrightarrow\left[
\begin{smallmatrix}
{\mathcal{T}}(T,T'') & 0 \\
M(U'',T) & {\mathcal{U}}(U,U'')\end{smallmatrix} \right] \\
&& \left( \left[ \begin{smallmatrix}
t_{2} & 0 \\
m_{2} & u_{2}
\end{smallmatrix} \right], \left[
\begin{smallmatrix}
t_{1} & 0 \\
m_{1} & u_{1}
\end{smallmatrix} \right]\right)\longmapsto\left[
\begin{smallmatrix}
t_{2}\circ t_{1} & 0 \\
m_{2}\bullet t_{1}+u_{2}\bullet m_{1} & u_{2}\circ u_{1}
\end{smallmatrix} \right].
\end{eqnarray*}
We recall that $ m_{2}\bullet t_{1}:=M(1_{U''}\otimes t_{1}^{op})(m_{2})$ and
$u_{2}\bullet m_{1}=M(u_{2}\otimes 1_{T})(m_{1})$,
and given 
an object $\left[
\begin{smallmatrix}
T & 0 \\
M & U
\end{smallmatrix} \right]\in \mathbf{\Lambda}$, the identity morphism is given by $1_{\left[
\begin{smallmatrix}
T & 0 \\
M & U
\end{smallmatrix} \right]}:=\left[
\begin{smallmatrix}
1_{T} & 0 \\
0 & 1_{U}
\end{smallmatrix} \right].$

In \cite[Theorem 3.14]{LeOS} it is proved the following  result.

\begin{theorem}\label{equivalenceLEOS}
Let $\mathcal{U}$ and $\mathcal{T}$  be additive categories and $ M\in \mathrm{Mod}(\mathcal{U}\otimes \mathcal{T}^{op})$. Then,  there exists a functor $\mathbb{G}:\mathrm{Mod}(\mathcal{U})\longrightarrow \mathrm{Mod}(\mathcal{T})$ for which there is an equivalence of categories 
$$\Big( \mathrm{Mod}(\mathcal T), \mathbb{G}\mathrm{Mod}(\mathcal U)\Big)\cong\mathrm{Mod}\big(\left[ \begin{smallmatrix}
\mathcal T & 0 \\
M & \mathcal U
\end{smallmatrix} \right] \big).$$
\end{theorem}
In \ref{defofG1} we will recall  briefly  the definition of the functor  above mentionated.

\begin{remark}\label{defofG1}
Let $\mathcal{R,S,T}$ additive categories and consider an additive functor  $ M \in \mathrm{Mod}(\mathcal{R}\otimes\mathcal{T}^{op})$. For all $ T\in \mathcal{T} $ we have the functor $ M_{T} :\mathcal{R}\longrightarrow \mathbf{Ab} $ defined as follows:
\begin{enumerate}
\item $ M_{T}(R):=M(R,T) $, for all $ R\in \mathcal{R} $.
\item $ M_{T}(r):=M(r\otimes 1_{T}): M_{T}(R)\longrightarrow M_{T}(R')$, for all $ r\in  \mathrm{Hom}_{\mathcal{R}}(R,R'). $
\end{enumerate}
Also for all $ t\in  \mathrm{Hom}_{\mathcal{T}}(T,T') $ we have a morphism of $ \mathcal{R} $-modules $ \bar{t}:M_{T'}\longrightarrow M_{T} $ such that $ \bar{t}=\lbrace [\bar{t}]_{R}:M_{T'}(R)\longrightarrow M_{T}(R)\rbrace_{ R\in \mathcal{R}} $ where $ [\bar{t}]_{R}=M(1_{R}\otimes t^{op}):M(R,T')\longrightarrow M(R,T) $.\\
So we have the functor $\mathbb{G}_{1}:\mathrm{Mod}(\mathcal{R})\longrightarrow \mathrm{Mod}(\mathcal{T}) $ as follows:
\begin{enumerate}
\item $\mathbb{G}_{1}(B)(T):= \mathrm{Hom}_{\mathrm{Mod}(\mathcal{R})}(M_{T},B) $, for all $ B\in \mathrm{Mod}(\mathcal{R}) $ and for all $ T\in \mathcal{T} $. Moreover $ \mathbb{G}_{1}(B)(t):=\mathrm{ Hom}_{\mathrm{Mod}(\mathcal{R})}(\bar{t},B) $ for all $ B\in \mathrm{Mod}(\mathcal{R}) $ and for all $ t\in \mathrm{Hom}_{\mathcal{T}}(T,T') $.
\item If $ \eta:B\longrightarrow B' $ is a morphism of $ \mathcal{R} $-modules,   $\mathbb{G}_{1}(\eta) :\mathbb{G}_{1}(B)\longrightarrow \mathbb{G}_{1}(B')$  is defined by \  \ $\mathbb{G}_{1}(\eta)= \left\lbrace [\mathbb{G}_{1}(\eta)]_{{T}}:\mathbb{G}_{1}(B)(T)\longrightarrow \mathbb{G}_{1}(B')(T)\right\rbrace _{T\in \mathcal{T}} $, with \  \ 
$[\mathbb{G}_{1}(\eta)]_{{T}}:=\mathrm{Hom}_{\mathrm{Mod}(\mathcal{R})}(M_{T},\eta)$.
\end{enumerate}

Hence we have the comma category $ \Big( \mathrm{Mod}(\mathcal{T}),\mathbb{G}_{1}\mathrm{Mod}(\mathcal{R})\Big) $ and  a equivalence of categories
$$ \Big( \mathrm{Mod}(\mathcal{T}),\mathbb{G}_{1}\mathrm{Mod}(\mathcal{R})\Big)\xrightarrow{\sim} \mathrm{Mod}(\left[ \begin{smallmatrix}
\mathcal T & 0 \\
M & \mathcal{R}
\end{smallmatrix} \right] ).$$

Similarly, given $M \in \mathrm{Mod}(\mathcal{R}\otimes\mathcal{T}^{op})$ we have $\overline{M}\in \mathrm{Mod}(\mathcal{T}^{op}\otimes\mathcal{R})$ and a functor $\overline{\mathbb{G}}:\mathrm{Mod}(\mathcal{T}^{op})\longrightarrow \mathrm{Mod}(\mathcal{R}^{op})$ (see section 4 in \cite{LeOS}).

\end{remark}

\begin{definition}\label{otrobimodulo}
Let $F:\mathrm{Mod}(\mathcal{R})\rightarrow \mathrm{Mod}(\mathcal{S}) $ an additive functor and $M\in\mathrm{Mod}(\mathcal{R}\otimes \mathcal{T}^{op})$. We define a bimodule in 
 $\mathrm{Mod}(\mathcal{S}\otimes\mathcal{T}^{op})$ denoted  by $N:=F(M)$ as follows, the functor $N=F (M) :\mathcal{S}\otimes\mathcal{T}^{op}\rightarrow \mathbf{Ab}$ is given by:\\
(i) $N(S,T):=F(M_{T})(S)$ for all $(S,T)\in \mathcal{S}\otimes\mathcal{T}^{op}$.\\
 (ii) Let $g\otimes t^{op}:(S,T)\rightarrow (S',T')$ where $g:S\rightarrow S'$ in $\mathcal{S}$ and $t:T'\rightarrow T$ in $\mathcal{T}$. Since $\bar{t}:M_{T}\rightarrow M_{T'}$ is a morphism of $\mathcal{R}$-modules, then $F(\bar{t}):F(M_{T})\rightarrow F(M_{T'})$ is a morphism of $\mathcal{S}$-modules. Thus we have the following commutative diagram.
      \[
      \begin{diagram}
      \node{F(M_{T})(S)}\arrow{e,t}{[F(\bar{t})]_{S}}\arrow{s,l}{F(M_{T})(g)}
       \node{F(M_{T'})(S)}\arrow{s,r}{F(M_{T'})(g)}\\
      \node{F(M_{T})(S')}\arrow{e,b}{[F(\bar{t})]_{S'}}
       \node{F(M_{T'})(S')}
      \end{diagram}
      \]
Hence we define $N(g\otimes t^{op}):=F(M_{T'})(g)\circ [F(\bar{t})]_{S}=[F(\bar{t})]_{S'}\circ F(M_{T})(g).$
\end{definition}

Now, that we have a bimodule $N\in \mathrm{Mod}(\mathcal{S}\otimes \mathcal{T}^{op})$ we define a functor $\mathbb{G}_{2}$ similar to $\mathbb{G}_{1}$. For convenience of the reader  we repeat its construction.

\begin{remark}\label{defofG2}
We define a functor $\mathbb{G}_{2}:\mathrm{Mod}(\mathcal{S})\rightarrow \mathrm{Mod}(\mathcal{T})$ as follows:

\begin{enumerate}
\item $\mathbb{G}_{2}(L)(T):= \mathrm{Hom}_{\mathrm{Mod}(\mathcal{S})}(N_{T},L) $, for all $ L\in \mathrm{Mod}(\mathcal{S}) $ and for all $ T\in \mathcal{T} $, where $N_{T}:=F(M_{T})\in \mathrm{Mod}(\mathcal{S})$. Moreover $\mathbb{G}_{2}(L)(t):=\mathrm{Hom}_{	\mathrm{Mod}(\mathcal{S})}(F(\bar{t}),L) $ for all $ L\in \mathrm{Mod}(\mathcal{S}) $ and for all $ t\in \mathrm{Hom}_{\mathcal{T}}(T',T) $.
\item If $ \gamma:L\longrightarrow L' $ is a morphism of $ \mathcal{S} $-modules we define $\mathbb{G}_{2}(\gamma) :\mathbb{G}_{2}(L)\longrightarrow \mathbb{G}_{2}(L')$ as: \\
$\mathbb{G}_{2}(\gamma)= \left\lbrace [\mathbb{G}_{2}(\gamma)]_{{T}}:=\mathrm{Hom}_{\mathrm{Mod}(\mathcal{S})}(N_{T},\gamma):\mathbb{G}_{2}(L)(T)\longrightarrow \mathbb{G}_{2}(L')(T)\right\rbrace _{T\in \mathcal{T}} $.
\end{enumerate}
\end{remark}
Since $N=F(M) \in \mathrm{Mod}(\mathcal{S}\otimes \mathcal{T}^{op})$ , we have the comma category $ \Big( \mathrm{Mod}(\mathcal{T}),\mathbb{G}_{2}\mathrm{Mod}(\mathcal{S})\Big)  $, 
and we have an equivalence of categories
$$ \Big( \mathrm{Mod}(\mathcal{T}),\mathbb{G}_{2}\mathrm{Mod}(\mathcal{S})\Big)\xrightarrow{\sim}\mathrm{Mod}(\left[ \begin{smallmatrix}
\mathcal T & 0 \\
F(M )& \mathcal{S}
\end{smallmatrix} \right] ).$$
\\
For all $ B\in \mathrm{Mod}(\mathcal{R})$ and $T\in \mathcal{T}$, $ F $ defines a mapping 
\[
F_{M_{T},B}:\mathrm{Hom}_{\mathrm{Mod}(\mathcal{R})}(M_{T},B)\longrightarrow \mathrm{Hom}_{\mathrm{Mod}(\mathcal{S})}(FM_{T},FB), \ f\mapsto F(f).
\] 
Similarly, for all $ L\in \mathrm{Mod}(\mathcal{S})$ and $T\in \mathcal{T}$, $ G $ defines a mapping 
\[
G_{N_{T},L}:\mathrm{Hom}_{\mathrm{Mod}(\mathcal{S})}(N_{T},L)\longrightarrow \mathrm{Hom}_{\mathrm{Mod}(\mathcal{R})}(GN_{T},GL), \ g\mapsto G(g).
\] 

In this way we have the following lemma.
 
\begin{lemma}\label{RecollH}
Let $M\in \mathrm{Mod}(\mathcal{R}\otimes \mathcal{T}^{op})$ be.
\begin{itemize}
\item[(a1)] For all $B\in \mathrm{Mod}(\mathcal{R})$,  the family $F_{M,B}:=\{F_{M_{T},B}:\mathbb{G}_1(B)(T)\rightarrow (\mathbb{G}_2 \circ F)(B)(T)\}_{T\in\mathcal T}$ is natural transformation, that is, 
$F_{M,B}:\mathbb{G}_1(B)\rightarrow (\mathbb{G}_2\circ F)(B)$ is a morphism of $\mathcal T$-modules.
                
\item[(a2)]  $ \xi:=F_{M,-}:\mathbb{G}_{1}\longrightarrow \mathbb{G}_{2}\circ F $ is a natural transformation.
\end{itemize}

\begin{itemize}
\item[(b1)] Suppose that $M_{T}=G(N_{T})$ for all $T\in \mathcal{T}$.  Then for all $L\in \mathrm{Mod}(\mathcal{S}) $,  the family $G_{N,L}:=\{G_{N_{T},L}:\mathbb{G}_2(L)(T)\rightarrow (\mathbb{G}_1 \circ G)(L)(T)\}_{T\in\mathcal T}$ is a natural transformation, that is, $G_{N,L}:\mathbb{G}_2(L)\rightarrow (\mathbb{G}_1\circ G)(L)$ is a morphism of $\mathcal T$-modules.

\item[(b2)] $\rho:=G_{N,-}=\mathbb{G}_{2}\longrightarrow \mathbb{G}_{1}\circ G$ is a natural transformation.
\end{itemize}
\end{lemma}
\begin{proof}
Since $F(M_{T})=N_{T}$,  for all $ B\in \mathrm{Mod}(\mathcal{R}) $  and $T\in\mathcal T$  we have $\mathbb{G}_1(B)(T)=\mathrm{Hom}_{\mathrm{Mod}(\mathcal{R})}(M_{T},B)$ and
 $(\mathbb{G}_2\circ F)(B)(T)=\mathrm{Hom}_{\mathrm{Mod}(\mathcal{S})}(F(M_{T}),F(B))$.\\
(a1) Let $ t\in \mathrm{Hom}_{\mathcal{T}}(T,T') $ and  $B\in \mathrm{Mod}(S) $. We have to show that the following diagram commutes
\[
\begin{diagram}
\node{\mathbb{G}_{1}(B)(T)}\arrow{e,t}{F_{M_{T},B}}\arrow{s,l}{\mathbb{G}_{1}(B)(t)}
 \node{\mathbb{G}_{2}(FB)(T)}\arrow{s,r}{\mathbb{G}_{2}(FB)(t)}\\
\node{\mathbb{G}_{1}(B)(T')}\arrow{e,b}{F_{M_{T'},B}}
 \node{\mathbb{G}_{2}(FB)(T').}
 \end{diagram}
\]
Note that if $ \varphi\in \mathbb{G}_{1}(B)(T) $ then
\[
(\mathbb{G}_{2}(FB)(t)\circ F_{M_{T},B})(\varphi)=\mathrm{Hom}_{\mathrm{Mod}(\mathcal{S})}(F(\overline{t}),FB)F(\varphi)=F(\varphi)F(\overline{t})
\]
and
\[
(F_{M_{T'},B}\circ \mathbb{G}_{1}(B)(t))(\varphi)=F_{M_{T'},B}\left( \mathrm{Hom}_{\mathrm{Mod}(\mathcal{R})}(\overline{t},B)(\varphi) \right)=F_{M_{T'},B}(\varphi\overline{t})=F(\varphi\overline{t}). 
\]
Proving that the diagram commutes.
\bigskip 

(a2) Let $ f\in \mathrm{Hom}_{\mathrm{Mod}(\mathcal{R})}(B,B') $ then we have to show the the following diagram commutes
\[
\begin{diagram}
\node{\mathbb{G}_{1}(B)}\arrow{e,t}{F_{M,B}}\arrow{s,l}{\mathbb{G}_{1}(f)}
 \node{\mathbb{G}_{2}F(B)}\arrow{s,r}{\mathbb{G}_{2}F(f)}\\
\node{\mathbb{G}_{1}(B')}\arrow{e,b}{F_{M,B'}}
 \node{\mathbb{G}_{2}F(B')}
 \end{diagram}
\]
For all $ T\in \mathcal{T} $ and for all $ \varphi\in \mathbb{G}_{1}(B)(T) $ we obtain the equalities.
\begin{align*}
([\mathbb{G}_{2}F(f)]_{T}\circ F_{M_T,B})(\varphi)= [\mathbb{G}_{2}F(f)]_{T}\left( F_{M_{T},B}(\varphi)\right) & =\mathrm{Hom}_{\mathrm{Mod}(\mathcal{S})}(F(M_{T}),F(f))F(\varphi)\\
& =F(f)F(\varphi)
\end{align*}
and 
\begin{align*}
([F_{M_T,B'}] \circ[\mathbb{G}_{1}(f)]_{T})(\varphi) =F_{M_{T},B'}\left( \mathrm{Hom}_{\mathrm{Mod}(\mathcal{R})}(M_{T},f)(\varphi)\right)=F(f\varphi)
\end{align*}
Proving that the previous diagram is commutative.\\
$(b1)$ and $(b2)$. Suppose that $M_{T}=G(N_{T})$ for all $T\in \mathcal{T}$. Then, for all  $L\in \mathrm{Mod}(\mathcal{S})$ and $T\in\mathcal T$ 
we have $\mathbb{G}_2(L)(T)=\mathrm{Hom}_{\mathrm{Mod}(\mathcal{S})}(N_{T},L)$ and
 $(\mathbb{G}_1\circ G) (L)(T)=\mathrm{Hom}_{\mathrm{Mod}(\mathcal{R})}(G(N_{T}),G(L))$. Therefore the prove of $(b1)$ and $(b2)$  is similar to $(a1)$ and $(a2)$.
\end{proof}

\begin{definition}
$\textnormal{\cite[Definition 3.2]{Chen}}$
Let $G_1: \mathscr{A}\rightarrow  \mathscr {D}$, $G_2: \mathscr {B}\rightarrow  \mathscr {D}$,  $F: \mathscr {A}\rightarrow  \mathscr {B}$ and 
$H: \mathscr B\rightarrow  \mathscr A$ be additive functors. Assume that $(F,H)$ is an adjoint pair, with $\eta$ being the adjugant equivalence. We say that the pair $(G_1,G_2)$ is compatible with the adjoint pair $(F,H)$ if there exist two natural transformations
$$ \xi:G_1\rightarrow G_2 F$$
and
$$ \rho:G_2\rightarrow G_1 H$$ 
such that $\rho_Y$ is a monomorphism and $G_1(\eta_{X,Y}(f))=\rho_YG_2(f)\xi_X$ for every $X\in  \mathscr{A}$, $Y\in  \mathscr{B}$ and $f\in\mathrm{Hom}_\mathscr {B}(FX,Y)$.
\end{definition}

The induced recollement of Theorem \ref{Recoll1} is described in the following theorem which details appear in \cite[Lemma 3.3, Lemma 3.4, Lemma 3.5]{Chen}.
 
\begin{theorem}
Let   $\mathscr{A,  B, C, }$ and  $\mathscr{D}$ abelian categories.
\begin{itemize}
\item[(i)]  Consider additive  functors $F: \mathscr {A}\rightarrow  \mathscr {B}$, $H: \mathscr{B}\rightarrow  \mathscr{A}$, $G_1: \mathscr{A}\rightarrow  \mathscr{D}$, $G_2: \mathscr{B}\rightarrow  \mathscr{D}$.  If $(F,H)$ is an adjoint pair and $(G_1,G_2)$ is compatible with the adjoint pair $(F,H)$, then $F$ and $H$ induce additive functors
$$ \widetilde{F}:(\mathscr{D},G_1\mathscr A)\rightarrow (\mathscr{D},G_2\mathscr B)\text{   and   }    \widetilde{H}:(\mathscr{D},G_2 \mathscr B)\rightarrow (\mathscr{D},G_1 \mathscr A)$$
such that  $(\widetilde{F},\widetilde{H})$ is an adjoint  pair. The pair $(\widetilde{F},\widetilde{H})$ is defined as follows: for every $(D,\varphi,A)\in(\mathscr{D},G_1\mathscr A)$, we set $\widetilde{F}((D,\varphi,A)=(D,\xi_{A}\circ \varphi,F(A))$, and for each 
$(f,g)\in \mathrm{Hom}_{(\mathscr{D},G_1 \mathscr A)}((D,\varphi,A),(D',\varphi', A'))$,  we set $\widetilde{F}((f,g))=(f,F(g))$. Similarly, $\widetilde{H}$ is defined. 

 \item[(ii)]  Consider additive  functors $F: \mathscr {A}\rightarrow  \mathscr {C}$, $H: \mathscr{C}\rightarrow  \mathscr {A}$ and $G: \mathscr {A}\rightarrow  \mathscr {D}$. 
 \begin{itemize}
\item [(a)]  If $(F,H)$ is an adjoint pair, then $F$ and $H$ induce additive functors $\widetilde{F}:(\mathscr {D}, G \mathscr {A})\rightarrow \mathscr C$ and $\widetilde{H}:\mathscr C\rightarrow(\mathscr {D}, G \mathscr {A})$ respectively, such that
 $(\widetilde{F},\widetilde{G})$ is an adjoint pair.  The pair  $(\widetilde{F},\widetilde{H})$ is defined as follows: for every $(D,\varphi,A)\in (\mathscr D, G \mathscr A)$ and $(f,g)\in 
   \mathrm{Hom}_{(\mathscr{D},G\mathscr A)}((D,\varphi,A),(D',\varphi',A'))$, $\widetilde{F}((D,\varphi,A))=F(A)$ and $\widetilde{F}((f,g))=F(g)$, and for every $C\in\mathscr C$ and
    $h\in\mathrm{Hom}_{\mathscr C}(C,C')$, $\widetilde{H}(C)=(GH(C), \mathrm{id}_{GH(C)},H(C))$ and $\widetilde{H}(h)=(GH(h),H(h))$.
    
\item[(b)]  If $(H,F)$ is an adjoint pair, then $F$ y $H$ induce additive functors $\widetilde{F}:(\mathscr {D}, G \mathscr {A})\rightarrow \mathscr C$ and $\widetilde{H}:\mathscr C\rightarrow(\mathscr {C}, G \mathscr {A})$ respectively, such that
  $(\widetilde{H},\widetilde{F})$ is an adjoint pair.  The pair  $(\widetilde{H},\widetilde{F})$ is defined as follows:  for every $C\in\mathscr C$ and $h\in\mathrm{Hom}_{\mathscr C}(C,C')$, $\widetilde{H}(C)=(0, 0,H(C))$ and $\widetilde{H}(h)=(0,H(h))$; and $\widetilde F$ is the same one in (a).
  
 \end{itemize}

\end{itemize}
\end{theorem}

The proof of  Theorem \ref{Recoll1} is based in the following result.

\begin{theorem}\cite[Theo. 3.6]{Chen}\label{RecollJ}
Let $\mathscr{A,B, C}$ and $\mathscr{D}$ be abelian categories, and let  $G_{1}:\mathscr{A}\longrightarrow \mathscr{D}$ and 
$G_{2}:\mathscr{B}\longrightarrow \mathscr{D}$ be left exact additive functors.
\begin{enumerate}
\item [(a)] If the diagram 
$$\xymatrix{\mathscr{C}\ar@<-2ex>[rr]_{i_{\ast}} & & \mathscr{A}\ar@<-2ex>[rr]_{j^{!}}\ar@<-2ex>[ll]_{i^{\ast}} & &  \mathscr{B}\ar@<-2ex>[ll]_{j_{!}}}$$
is a left recollement where, $(G_{2},G_{1})$ is compatible with the adjoint $(j_{!},j^{!})$,  then there is a left recollement
$$\xymatrix{\mathscr{C}\ar@<-2ex>[rr]_{\tilde{i_{\ast}}} & & \Big(\mathscr{D},G_{1}\mathscr{A}\Big)\ar@<-2ex>[rr]_{\tilde{j^{!}}}\ar@<-2ex>[ll]_{\tilde{i^{\ast}}} & &  \Big(\mathscr{D},G_{2}\mathscr{B}\Big)\ar@<-2ex>[ll]_{\tilde{j_{!}}}}$$

\item [(b)] If the diagram 
$$\xymatrix{\mathscr{C}\ar@<2ex>[rr]^{i_{!}} & & \mathscr{A}\ar@<2ex>[rr]^{j^{\ast}}\ar@<2ex>[ll]^{i^{!}} & &  \mathscr{B}\ar@<2ex>[ll]^{j_{\ast}}}$$
is a right recollement, where $(G_{1},G_{2})$ is compatible with the adjoint $(j^{\ast},j_{\ast})$, then there is a right recollement
$$\xymatrix{\mathscr{C}\ar@<2ex>[rr]^{\tilde{i_{!}}} & & \Big(\mathscr{D},G_{1}\mathscr{A}\Big)\ar@<2ex>[rr]^{\tilde{j^{\ast}}}\ar@<2ex>[ll]^{\tilde{i^{!}}} & & \Big(\mathscr{D},G_{2}\mathscr{B}\Big)\ar@<2ex>[ll]^{\tilde{j_{\ast}}}}$$
\end{enumerate}
\end{theorem}
\bigskip

In order to prove Theorem \ref{Recoll1}  we need the following result which generalizes \cite[Lemma 4.2]{Chen}.
\begin{lemma}\label{RecollG}
Let $ \mathcal{R},\, \mathcal{S}  $ and $ \mathcal{T} $ be abelian categories and $ F:\mathrm{Mod}(\mathcal{R})\longrightarrow \mathrm{Mod}(\mathcal{S}) $ and $ G:\mathrm{Mod}(\mathcal{S}) \longrightarrow \mathrm{Mod}(\mathcal{R})$ be additive functors. For $ M\in \mathrm{Mod}(\mathcal{R}\otimes \mathcal{T}^{op}) $ considere the additive functors $\mathbb{G}_{1}:\mathsf{Mod}(\mathcal{R})\longrightarrow \mathrm{Mod}(\mathcal{T}) $ and $ \mathbb{G}_{2}:\mathrm{Mod}(\mathcal{S})\longrightarrow \mathrm{Mod}(\mathcal{T}) $ as we have defined in \ref{defofG1} and \ref{defofG2} where $N_{T}=F(M_{T} )$. If $ (F,G) $ is an adjoint pair and its unit $ \varepsilon:1_{\mathsf{Mod}(\mathcal{R})}\longrightarrow GF $ satisfies $ \varepsilon_{M_{T}}=1_{M_{T}}$  for all $T\in \mathcal{T}$, then the pair $ (\mathbb{G}_{1},\mathbb{G}_{2})  $ is compatible with $(F,G) $.
\end{lemma}
\begin{proof}
Since $ (F,G) $ is an adjoint pair there exist a natural equivalence 
\[
\eta:=\!\!\lbrace \eta_{B,L}\!:\!\mathrm{Hom}_{\mathrm{Mod}(\mathcal{S})}(FB,L)\rightarrow \mathrm{Hom}_{\mathrm{Mod}(\mathcal{R})}(B,GL)\rbrace_{B\in \mathrm{Mod}(\mathcal{R}),\,\,L\in \mathrm{Mod}(\mathcal{S})}
\]
By Lemma \ref{RecollH} (a2) we have natural transformations $ \xi:=F_{M,-}:\mathbb{G}_{1}\longrightarrow \mathbb{G}_{2}F $. Since $ \varepsilon_{M_{T}}=1_{M_{T}}$, we have that $G(N_{T})=M_{T}$ for all $T\in \mathcal{T}$ and by \ref{RecollH} (b2), we have a natural transformation $\rho:=G_{N,-}=\mathbb{G}_{2}\longrightarrow \mathbb{G}_{1}G$.\\
First we will see that for all $L\in \mathrm{Mod}(\mathcal S)$ the morphism 
$$\rho_{L}:\mathbb{G}_{2}(L)\longrightarrow \mathbb{G}_{1}G(L)$$
is a monomorphism in $\mathrm{Mod}(\mathcal{T})$. Indeed, for $T\in \mathcal{T}$ we have to show that 
$$[\rho_{L}]_{T}:=G_{N_{T},L}:\mathrm{Hom}_{\mathrm{Mod}(\mathcal{S})}(N_{T},L)\rightarrow \mathrm{Hom}_{\mathrm{Mod}(\mathcal {R})}(M_{T},G(L))$$
is a monomorphism. Consider the morphism
$$\eta_{M_{T},L}:\mathrm{Hom}_{\mathrm{Mod}(\mathcal{S})}(F(M_{T}),L)\longrightarrow \mathrm{Hom}_{\mathrm{Mod}(\mathcal{R})}(M_{T},G(L)).$$
We assert that $\eta_{M_{T},L}=G_{N_{T},L}.$ Indeed, let 
$\beta \in  \mathrm{Hom}_{\mathrm{Mod}(\mathcal{S})}(F(M_{T}),L)$. 
Then we have the following commutative diagram
$$\xymatrix{\mathrm{Hom}_{\mathrm{Mod}(\mathcal{S})}(
F(M_{T}),F(M_{T}))\ar[rr]^{\eta_{M_{T},F(M_{T})}}\ar[d]_{\mathrm{Hom}_{\mathrm{Mod}(\mathcal{S})}(F(M_{T}),\beta)} & & \mathrm{Hom}_{\mathrm{Mod}(\mathcal{R})}(M_{T},GF(M_{T}))\ar[d]^{\mathrm{Hom}_{\mathrm{Mod}(\mathcal{R})}(M_{T},G(\beta))}\\
\mathrm{Hom}_{\mathrm{Mod}(\mathcal{S})}(F(M_{T}),L)\ar[rr]_{\eta_{M_{T},L}} & & \mathrm{Hom}_{\mathrm{Mod}(\mathcal{R})}(M_{T},G(L))}$$
Then
\begin{eqnarray*}
G(\beta)\circ \big( \eta_{M_{T},F(M_{T})}(1_{F(M_{T})})\big)& = &\big(\Hom_{\mathrm{Mod}(\mathcal{R})}(M_{T},G(\beta))\circ\eta_{M_{T},F(M_{T})}\big)(1_{F(M_{T})})\\
& = &\big(\eta_{M_{T},L}\circ \Hom_{\mathrm{Mod}(\mathcal{S})}(N_{T},\beta)\big)(1_{F(M_{T})})\\
& = &\eta_{M_{T},L}\Big(\Hom_{\mathrm{Mod}(\mathcal{S})}(N_{T},\beta)\big)(1_{F(M_{T})})\Big)\\
&=&\eta_{M_{T},L}(\beta\circ 1_{F(M_{T})})\\
&= & \eta_{M_{T},L}(\beta)
\end{eqnarray*}
Since $1_{M_{T}}=\varepsilon_{M_{T}}= \eta_{M_{T},F(M_{T})}(1_{F(M_{T})})$ then $ G(\beta)= \eta_{M_{T},L}(\beta)$.
Since $G_{N_{T},L}(\beta)=G(\beta)$ and $\eta_{M_{T},L}$ is injective, it follows that $G_{N_{T},L}$ is injective, for all $ T\in \T $.
Proving that $\rho_{L}$ is a monomorphism for each $L\in \mathrm{Mod}(\mathcal{S})$.\\
Now we have to show that $\mathbb{G}_1(\eta_{B,L}(f))=\rho_{L}\mathbb{G}_2(f)\xi_B$ for all $f\in \mathrm{Hom}_{\mathrm{Mod}(\mathcal {S})}(F(B),L)$. That is, we have to show that
$$
[\mathbb{G}_{1}\left( \eta_{B,L}(f)\right)]_{T}=\rho_{N_{T},L} [\mathbb{G}_2(f)]_T\xi_{M_{T},B}\,\,\,\,\,\,\forall T\in\mathcal T.$$ 
Let  $ \alpha\in \Hom_{\mathrm{Mod}(\mathcal{R})}(M_{T},B) be $. It follows from the  following commutative diagram
\[
\begin{diagram}
\node{\Hom_{\mathrm{Mod}(\mathcal{S})}(F(B),L)}\arrow{e,t}{\eta_{B,L}}\arrow{s,l}{\Hom_{\mathrm{Mod}(\mathcal{S})}(F(\alpha),L)}
 \node{\Hom_{\mathrm{Mod}(\mathcal{R})}(B,G(L))}\arrow{s,r}{\Hom_{\mathrm{Mod}(\mathcal{R})}(\alpha,G(L))}\\
\node{\Hom_{\mathrm{Mod}(\mathcal{S})}(F(M_{T}),L)}\arrow{e,b}{\eta_{M_{T},L}}
 \node{\Hom_{\mathrm{Mod}(\mathcal{R})}(M_{T},G(L))}
\end{diagram}
\]
that $G(f\circ F(\alpha))=\eta_{M_{T},L}\Big(f\circ F(\alpha)\Big)=\eta_{B,L}(f)\circ \alpha.$
We note that
\[
\Big[\mathbb{G}_{1}\left( \eta_{B,L}(f)\right)\Big]_{T}(\alpha)=\Hom_{\mathrm{Mod}(\mathcal{R})}\Big(M_{T},\eta_{B,L}(f)\Big)(\alpha)=\eta_{B,L}(f)\circ \alpha.
\]
On the other hand,
\begin{eqnarray*}
\Big(\rho_{N_T,L}\circ [\mathbb{G}_{2}(f)]_T\circ\xi_{M_{T},B}\Big)(\alpha)&=& (G_{N_{T},L}\circ\Hom_{\mathrm{Mod}(\mathcal{S})}(F(M_{T}),f))( F_{M_{T},B}(\alpha))\\
&=& (G_{N_{T},L}\circ\Hom_{\mathrm{Mod}(\mathcal{S})}(F(M_{T}),f))(F(\alpha)) \\
&=& G_{N_{T},L}(f\circ F(\alpha)) \\
&=& G( f\circ F(\alpha)).
\end{eqnarray*}
Proving that $(\mathbb{G}_{1},\mathbb{G}_{2})$ is compatible with $ (F,G) $.
\end{proof}

\begin{proof}[Proof of Theorem \ref{Recoll1}.]
We only prove (a), since (b) is similar. Set $N=j_!(M)$ as in \ref{otrobimodulo}, and consider the additive functors $\mathbb{G}_1:\mathrm{Mod}(\mathcal R)\rightarrow \mathrm{Mod}(\mathcal T)$ and
$\mathbb{G}_2:\mathrm{Mod}(\mathcal S)\rightarrow \mathrm{Mod}(\mathcal T)$ as defined in \ref{defofG1} and \ref{defofG2}. Since $j_!$ is a full embedding,
 by \cite[Lemma 2.1]{Chen}, we may assume that the unit
$\epsilon:1\rightarrow j^!j_!$, of the adjoint pair $(j_!,j^!)$, is the identity. In particular, we have that $\epsilon_{M_T}=1_{M_T}$. Thus, from Lemma \ref{RecollG}, the pair $(\mathbb{G}_1,\mathbb{G}_2)$ is 
compatible with $(j_!,j^!)$  and the rest follows from Theorem \ref{RecollJ} and \ref{equivalenceLEOS}.
\end{proof}
We note that recollement can be seen as the gluing of a left recollement and a right recollement. Since $i_*=i_!$ and $j^!=j^*$, it follows that
$\widetilde{i_*}=\widetilde{i_!}$ and $\widetilde{j^!}=\widetilde{j^*}$. For  any $M\in\mathrm{Mod}(\mathcal{R}\otimes \mathcal{T}^{op})$,  consider
the matrix categories 
$\mathbf{\Lambda}:=\left (\begin{smallmatrix}
\mathcal T&0\\
M& \mathcal R
\end{smallmatrix}\right )$ and 
$\mathbf{\Lambda}^!:=\left (\begin{smallmatrix}
\mathcal T&0\\
j_!(M)& \mathcal S
\end{smallmatrix}\right ).$ 
It follows, that if the diagram $\xymatrix{\mathrm{Mod}(\mathcal{C})\ar[r]_{} &  \mathrm{Mod}(\mathcal{S})\ar[r]_{}\ar@<-2ex>[l]_{}\ar@<2ex>[l]_{}  &   \mathrm{Mod}(\mathcal{R})\ar@<-2ex>[l]_{}\ar@<2ex>[l]_{}}$ is a recollement, then there is a recollement
$$\xymatrix{\mathrm{Mod}(\mathcal{C})\ar[r]_{} &  \mathrm{Mod}(\mathbf{\Lambda}^{!})\ar[r]_{}\ar@<-2ex>[l]_{}\ar@<2ex>[l]_{}  &   \mathrm{Mod}(\mathbf{\Lambda}).\ar@<-2ex>[l]_{}\ar@<2ex>[l]_{}}$$
Recall that two additive categories  $\mathcal{C}$ and $\mathcal{D}$ are called Morita equivalent if the functor categories $\mathrm{Mod}(\mathcal{C})$  and $\mathrm{Mod}(\mathcal{D})$ are equivalent.

\begin{corollary}
Let $\mathcal{R}$, $\mathcal{S}$,  and $\mathcal{T}$ additive categories, and $M\in\mathrm{Mod}(\mathcal R\otimes \mathcal T^{op})$. If $\mathcal{R}$ is Morita equivalent to $\mathcal{S}$, then there exists a functor $N\in\mathrm{Mod}(\mathcal{S}\otimes \mathcal{T}^{op})$ making the triangular matrix category 
$\left (\begin{smallmatrix}
\mathcal T&0\\
M& \mathcal{R}
\end{smallmatrix}\right )$ Morita equivalent to $ \left (\begin{smallmatrix}
\mathcal T&0\\
N& \mathcal S
\end{smallmatrix}\right )$.
\end{corollary}
\begin{proof}
The same proof as in \cite[Corollary 4.7]{Chen} works for this setting.
\end{proof}

The following result tell us that under certain conditions we can restrict to the category of finitely presented modules.

\begin{theorem}\label{Recollfinitos}
Let $\mathcal {R,S, C}$ and $\mathcal{T}$ be dualizing $K$-varieties. For $M\in\mathrm{Mod}(\mathcal{R}\otimes_{K}\mathcal{T}^{op})$
such that $M_{T}\in \mathrm{mod}(\mathcal{R})$ and  $M_{R}\in \mathrm{mod}(\mathcal{T}^{op})$ for all $T\in \mathcal{T}$ and $R\in \mathcal{U}$, consider the matrix categories $\mathbf{\Lambda}:=\left(\begin{smallmatrix}
\mathcal{T} &0\\
M& \mathcal{R}
\end{smallmatrix}\right )$  $\mathbf{\Lambda}^!:=\left (\begin{smallmatrix}
\mathcal{T} &0\\
j_{!}(M)& \mathcal{S}
\end{smallmatrix}\right )$, $\mathbf{\Lambda}^*:=\left (\begin{smallmatrix}
\mathcal{T} &0\\
j_{*}(M)& \mathcal{S}
\end{smallmatrix}\right ),$  where the bimodules $j_{!}(M)$ and $j_{\ast}(M)$ are constructed as in \ref{otrobimodulo}. Moreover suppose that $j_{!}(M)_{S},\,\, j_{\ast}(M)_{S}\in \mathrm{mod}(\mathcal{T}^{op})$ for all $S\in \mathcal{S}$.
\begin{enumerate}
\item [(a)]
If the diagram 
$$\xymatrix{\mathrm{mod}(\mathcal{C})\ar@<-2ex>[rr]_{i_{\ast}} & & \mathrm{mod}(\mathcal{S})\ar@<-2ex>[rr]_{j^{!}}\ar@<-2ex>[ll]_{i^{\ast}} & &  \mathrm{mod}(\mathcal{R})\ar@<-2ex>[ll]_{j_{!}}}$$
is a left recollement, then there is a left recollement
$$\xymatrix{\mathrm{mod}(\mathcal{C})\ar@<-2ex>[rr]_{\tilde{i_{\ast}}} & & \mathrm{mod}(\mathbf{\Lambda}^{!})\ar@<-2ex>[rr]_{\tilde{j^{!}}}\ar@<-2ex>[ll]_{\tilde{i^{\ast}}} & &  \mathrm{mod}(\mathbf{\Lambda})\ar@<-2ex>[ll]_{\tilde{j_{!}}}}$$

\item [(b)] If the diagram 
$$\xymatrix{\mathrm{mod}(\mathcal{C})\ar@<2ex>[rr]^{i_{!}} & & \mathrm{mod}(\mathcal{S})\ar@<2ex>[rr]^{j^{\ast}}\ar@<2ex>[ll]^{i^{!}} & &  \mathrm{mod}(\mathcal{R})\ar@<2ex>[ll]^{j_{\ast}}}$$
is a right recollement, then there is a right recollement
$$\xymatrix{\mathrm{mod}(\mathcal{C})\ar@<2ex>[rr]^{\tilde{i_{!}}} & & \mathrm{mod}(\mathbf{\Lambda}^{\ast})\ar@<2ex>[rr]^{\tilde{j^{\ast}}}\ar@<2ex>[ll]^{\tilde{i^{!}}} & & \mathrm{mod}(\mathbf{\Lambda}).\ar@<2ex>[ll]^{\tilde{j_{\ast}}}}$$
\end{enumerate}
\end{theorem}
\begin{proof}
First, we note that by the definition \ref{otrobimodulo}, we have that $j_{!}(M)_{T}=j_{!}(M_{T})$; and since $M_{T}\in \mathrm{mod}(\mathcal{R})$ and $j_{!}:\mathrm{mod}(\mathcal{R})\longrightarrow \mathrm{mod}(\mathcal{S})$ we have that $j_{!}(M)_{T}\in \mathrm{mod}(\mathcal{S})$. Similarly $j_{\ast}(M)_{T}\in \mathrm{mod}(\mathcal{S})$.  Then by \cite[Proposition 6.3]{LeOS}, we have equivalences
$$\Big( \mathrm{mod}(\mathcal{T}),\mathbb{G}_{1}\mathrm{mod}(\mathcal{R})\Big)\xrightarrow{\sim} \mathrm{mod}(\left[ \begin{smallmatrix}
\mathcal T & 0 \\
M & \mathcal{R}
\end{smallmatrix} \right] ).$$
and $$ \Big( \mathrm{mod}(\mathcal{T}),\mathbb{G}_{2}\mathrm{Mod}(\mathcal{S})\Big)\xrightarrow{\sim} \mathrm{mod}(\left[ \begin{smallmatrix}
\mathcal T & 0 \\
j_{!}(M) & \mathcal{S}
\end{smallmatrix} \right] ),$$
where
$\mathbb{G}_1:\mathrm{Mod}(\mathcal R)\rightarrow \mathrm{Mod}(\mathcal T)$ and
$\mathbb{G}_2:\mathrm{Mod}(\mathcal S)\rightarrow \mathrm{Mod}(\mathcal T)$ are defined in \ref{defofG1} and \ref{defofG2}. Under our conditions we have that they restrict well to 
$\mathbb{G}_1|_{\mathrm{mod}(\mathcal R)}:\mathrm{mod}(\mathcal R)\rightarrow \mathrm{mod}(\mathcal T)$ and
$\mathbb{G}_2|_{\mathrm{mod}(\mathcal S)}:\mathrm{mod}(\mathcal S)\rightarrow \mathrm{mod}(\mathcal T)$. It can be seen as in the proof of \ref{Recoll1},  that they are compatible with $(j_!,j^!)$  and the rest follows from Theorem \ref{RecollJ}.

\end{proof}

\section{The maps category $\mathrm{maps}(\mathcal C)$}
Assume that  $\mathcal{C}$ is an $R$-variety. The maps category, $\mathrm{maps}(\mathcal{C})$  is defined as follows. 
The objects in $\mathrm{maps}(\mathcal{C})$ are morphisms $(f_{1},A_{1},A_{0}):A_{1}\xrightarrow{f_1}A_{0}$, and the maps are pairs $(h_{1},h_{0}):(f_{1},A_{1},A_{0})\rightarrow (g_{1},B_{1},B_{0})$, such that
the following square commutes
$$\xymatrix{A_{1}\ar[r]^{f_{1}}\ar[d]_{h_{1}} & A_{0}\ar[d]_{h_{0}}\\
B_{1}\ar[r]_{g_{1}} & B_{0}.}$$

In this section, we study the category $\mathrm{maps}(\mathrm{Mod}(\mathcal{C})):=\Big(\mathrm{Mod}(\mathcal{C}),\mathrm{Mod}(\mathcal{C})\Big)$ of maps of the category $\mathrm{Mod}(\mathcal{C})$ and we give in this setting a description of the functor $\widehat{\Theta}:
\Big(\mathrm{Mod}(\mathcal{T}),\mathbb{G}\mathrm{Mod}(\mathcal{U})\Big)\longrightarrow \Big(\mathrm{Mod}(\mathcal{U}^{op}),\overline{\mathbb{G}}\mathrm{Mod}(\mathcal{T}^{op})\Big)$
constructed in \cite[Proposition 4.9]{LeOS} (see \ref{descrimaps}). We also give a description of the projective and injective objects of the category $\mathrm{maps}(\mathrm{mod}(\mathcal{C}))$  when $\mathcal{C}$ is a dualizing variety and we also describe its radical (see  \ref{projinjec}).  Let $\mathcal{C}$ be a  dualizing $K$-variety and consider  the category
$\left[ \begin{smallmatrix}
\mathcal C& 0 \\ 
 \widehat{\mathbbm{Hom}}& \mathcal C
\end{smallmatrix}\right]$ of triangular matrices
with $\widehat{\mathbbm{Hom}}:\mathcal C\otimes \mathcal C^{op}\rightarrow \mathbf{Ab}$ defined as in \cite{LeOS}.  In \cite[Proposition 7.3]{LeOS}, we proved that  $\left[ \begin{smallmatrix}
\mathcal C& 0 \\ 
 \widehat{\mathbbm{Hom}}& \mathcal C
\end{smallmatrix}\right]$  is a dualizing category.  We will  show in this section, that  
the category  $\mathrm{mod}\Big( 
\left[\begin{smallmatrix}
\mathcal{C}& 0 \\ 
 \widehat{\mathbbm{Hom}}& \mathcal{C}
\end{smallmatrix}\right]\Big)$ is equivalent to the category $\mathrm{maps}(\mathrm{mod}(\mathcal{C}))$. Some results in this section are generalizations of results given in the chapter 3 of \cite{ARS}.\\
Finally, inspired by \cite{Ingmar}, we show as an example that  the category  $\left[ \begin{smallmatrix}
K\Delta/\langle\rho\rangle & 0 \\
\widehat{\mathbbm{{Hom}}} & K\Delta/\langle\rho\rangle
\end{smallmatrix} \right]$, where  $K\Delta/\langle\rho\rangle$ is the path category  of the quiver
\[ 
\begin{diagram}
\node{\cdots}\arrow{e,t}{}
  \node{i-1}\arrow{e,t}{\alpha_{i-1}}
   \node{i}\arrow{e,t}{\alpha_{i}}
    \node{i+1}\arrow{e,t}{}
     \node{\cdots}
 \end{diagram}
\] 
defined by  $\Delta=(\Delta_0,\Delta_1)$,   with $\Delta_0=\mathbb  Z$, $\Delta_1=\{\alpha_i:i\rightarrow i+1: i\in\mathbb Z\}$ and  the set of relations $\rho=\{\alpha_{i}\alpha_{i-1}:i\in\mathbb Z\}$, 
is again  a path category.\\
Moreover, we show that the category of $\left[ \begin{smallmatrix}
K\Delta/\langle\rho\rangle & 0 \\
\widehat{\mathbbm{{Hom}}} & K\Delta/\langle\rho\rangle
\end{smallmatrix} \right]$-modules 
is equivalent to the category $\mathrm{maps}(\mathbf{Ch}(\mathrm{Mod} (K)))$, where $\mathbf{Ch}(\mathrm{Mod} (K))$ is the category of chain complexes in $\mathrm{Mod} (K)$.\\

\begin{definition}
Define a functor
$$\xymatrix{\Big(\mathrm{Mod}(\mathcal{C}),\mathrm{Mod}(\mathcal{C})\Big)\ar[rr]^{\overline{\Theta}} && \Big(\mathrm{Mod}(\mathcal{C}^{op}),\mathrm{Mod}(\mathcal{C}^{op})\Big)}$$ 
in objects as $\overline{\Theta}(\xymatrix{A\ar[r]^{f} & B})=
\xymatrix{\mathbb{D}_{\mathcal{C}}(B)\ar[r]^{\mathbb{D}_{\mathcal{C}}(f)} & \mathbb{D}_{\mathcal{C}}(A)}$ and if 
$(\alpha,\beta):(A,f,B)\longrightarrow (A',f',B')$ is a morphism in $\mathrm{maps}(\mathrm{Mod}(\mathcal{C}))$ then $\overline{\Theta}(\alpha,\beta)=(\mathbb{D}_{\mathcal{C}}(\beta),\mathbb{D}_{\mathcal{C}}(\alpha))$. 
\end{definition}

First we have the following result, which tell us that we can identify the functor  $\widehat{\Theta}:
\Big(\mathrm{Mod}(\mathcal{T}),\mathbb{G}\mathrm{Mod}(\mathcal{U})\Big)\longrightarrow \Big(\mathrm{Mod}(\mathcal{U}^{op}),\overline{\mathbb{G}}\mathrm{Mod}(\mathcal{T}^{op})\Big)$
constructed in \cite[Proposition 4.9]{LeOS} with $\overline{\Theta}$.

\begin{proposition}\label{descrimaps}
Let $\mathcal{C}$ be a $K$-variety and $\mathrm{maps}(\mathrm{Mod}(\mathcal{C})):=\Big(\mathrm{Mod}(\mathcal{C}),\mathrm{Mod}(\mathcal{C})\Big)$ the maps category. Let 
$M:=\widehat{\mathbbm{Hom}}\in \mathrm{Mod}(\mathcal{C}\otimes_{K}\mathcal{C}^{op})$ be where $M_{T}=\mathrm{Hom}_{\mathcal{C}}(T,-)\in \mathrm{Mod}(\mathcal{C})$ for $T\in \mathcal{C}^{op}$ and $M_{U}=\mathrm{Hom}_{\mathcal{C}}(-,U)\in \mathrm{Mod}(\mathcal{C}^{op})$ for $U\in \mathcal{C}$ and
consider the induced functors $\mathbb{G}:\mathrm{Mod}(\mathcal C)\rightarrow \mathrm{Mod}(\mathcal C)$, $\overline{\mathbb{G}}:\mathrm{Mod}(\mathcal{C}^{op})\rightarrow \mathrm{Mod}(\mathcal {C}^{op})$ (see \ref{defofG1}). Then, the there exists isomorphisms $J_{1}$ and $J_{2}$ such that the following diagram commutes
$$\xymatrix{\Big(\mathrm{Mod}(\mathcal{C}),\mathbb{G}\mathrm{Mod}(\mathcal{C})\Big)\ar[rr]^{\widehat{\Theta}}& & 
\Big(\mathrm{Mod}(\mathcal{C}^{op}),\overline{\mathbb{G}}\mathrm{Mod}(\mathcal{C}^{op})\Big)\\
\Big(\mathrm{Mod}(\mathcal{C}),\mathrm{Mod}(\mathcal{C})\Big)\ar[rr]^{\overline{\Theta}}\ar[u]^{J_{1}} && \Big(\mathrm{Mod}(\mathcal{C}^{op}),\mathrm{Mod}(\mathcal{C}^{op})\Big)\ar[u]^{J_{2}}}$$
where $\widehat{\Theta}$ is the functor defined in \cite[Proposition 4.9]{LeOS}.
\end{proposition}
\begin{proof}
Let us define $J_{1}:\Big(\mathrm{Mod}(\mathcal{C}),\mathrm{Mod}(\mathcal{C})\Big)\longrightarrow \Big(\mathrm{Mod}(\mathcal{C}),\mathbb{G}\mathrm{Mod}(\mathcal{C})\Big)$. For this, consider the Yoneda isomorphism $\theta_{B,T}^{-1}:B(T)\longrightarrow \mathrm{Hom}_{\mathrm{Mod}(\mathcal{C})}(\mathrm{Hom}_{\mathcal{C}}(T,-),B)$. Let $f:A\longrightarrow B\in \mathrm{maps}(\mathrm{Mod}(\mathcal{C}))$ we set
$J_{1}(f):=\widehat{f}:A\longrightarrow \mathbb{G}(B)$ where for 
$T\in \mathcal{C}$ we have that $\widehat{f}_{T}:=\theta_{B,T}^{-1}\circ f_{T}:A(T)\longrightarrow \mathrm{Hom}_{\mathrm{Mod}(\mathcal{C})}(\mathrm{Hom}_{\mathcal{C}}(T,-),B)=\mathbb{G}(B)(T)$.\\ It is easy to see that $J_{1}^{-1}:\Big(\mathrm{Mod}(\mathcal{C}),\mathbb{G}\mathrm{Mod}(\mathcal{C})\Big)\longrightarrow \Big(\mathrm{Mod}(\mathcal{C}),\mathrm{Mod}(\mathcal{C})\Big)$ is defined as follows: for $f:A\longrightarrow \mathbb{G}(B)$ we set  $J_{1}^{-1}(f)=f':A\longrightarrow B$ where for $T\in \mathcal{T}$ we have that
$f_{T}':=\theta_{B,T}\circ f_{T}:A(T)\longrightarrow B(T)$.\\
In a similar way is defined $J_{2}$.\\
Now we define $\xymatrix{\overline{\mathbb{G}}(\mathbb{D}_{\mathcal{C}}\mathbb{G}(B))\ar[r]^{\Psi'_{B}} & \mathbb{D}_{\mathcal{C}}(B)}$ as follows. 
For $U\in \mathcal{C}^{op}$ we take
$$\delta\in \overline{\mathbb{G}}(\mathbb{D}_{\mathcal{C}}(\mathbb{G}(B)))(U)=\mathrm{Hom}_{\mathsf{Mod}(\mathcal{C}^{op})}\Big(\mathrm{Hom}_{\mathcal{C}}(-,U),\,\,\,\mathbb{D}_{\mathcal{C}}\mathbb{G}(B) \Big)$$ and consider the Yoneda isomorphism
$$\mathbb{Y}_{B,U}:\mathrm{Hom}_{\mathsf{Mod}(\mathcal{C}^{op})}\Big(\mathrm{Hom}_{\mathcal{C}}(-,U),\,\,\,\mathbb{D}_{\mathcal{C}}\mathbb{G}(B) \Big)\longrightarrow \mathbb{D}_{\mathcal{C}}\mathbb{G}(B)(U).$$
Thus we define $[\Psi'_{B}]_{U}(\delta)=\delta_{U}(1_{U})\circ \theta^{-1}_{B,U}:B(U)\longrightarrow K,$
where $\theta^{-1}_{B,U}:B(U)\longrightarrow \mathrm{Hom}_{\mathrm{Mod}(\mathcal{C})}(\mathrm{Hom}_{\mathcal{C}}(U,-),B)$ is the Yoneda isomorphism.\\
By \cite[Proposition 4.7]{LeOS}, there exists a morphism of $\mathcal{C}^{op}$-modules $\Psi_{B}:\mathbb{D}_{\mathcal{C}}(B)\rightarrow \overline{\mathbb{G}}\mathbb{D}_{\mathcal{C}}\mathbb{G}(B)$. Hence, for $B\in \mathrm{Mod}(\mathcal{C})$, $U\in \mathcal{C}^{op}$ and  $s\in \mathrm{Hom}_{K}(B(U), K)$, we have a morphism of $\mathcal{C}^{op}$-modules $[\Psi_{B}]_{U}(s):=\mathbb{S}_{B,U}:\mathrm{Hom}_{\mathcal{C}}(-,U)\longrightarrow \mathbb{D}_{\mathcal{C}}\mathbb{G}(B).$
Therefore, we get that the morphism $[\mathbb{S}_{B,U}]_{U}(1_{U}):\mathrm{Hom}_{\mathrm{Mod}(\mathcal{C})}(\mathrm{Hom}_{\mathcal{C}}(U,-),B) \longrightarrow K$
is defined as
$$\Big([\mathbb{S}_{B,U}]_{U}(1_{U})\Big)(\eta):=\Big(s\circ [\eta]_{U}\Big)(1_{U})=s\Big([\eta]_{U}(1_{U})\Big),$$
for every $\eta\in \mathrm{Hom}_{\mathrm{Mod}(\mathcal{C})}(\mathrm{Hom}_{\mathcal{C}}(U,-),B)$.
Then $s=[\mathbb{S}_{B,U}]_{U}(1_{U})\circ \theta_{B,U}^{-1}$.
From this it follows that the following diagram is commutative for every $U\in \mathcal{C}^{op}$
$$(\ast):\quad \xymatrix{\mathbb{D}_{\mathcal{C}}(B)(U)\ar[rr]^{[\Psi_{B}]_{U}}\ar@{=}[d] & &\overline{\mathbb{G}}(\mathbb{D}_{\mathcal{C}}\mathbb{G}(B))(U)\ar[d]^{[\Psi'_{B}]_{U}}\\
\mathbb{D}_{\mathcal{C}}(B)(U)\ar[rr]^{1} &&\mathbb{D}_{\mathcal{C}}(B)(U).}$$
Moreover is easy to see that $[\Psi_{B'}]_{U}$ is an isomorphism.
We conclude that $\Psi_{B}^{-1}=\Psi_{B}'$.\\
Now, given $f:A\longrightarrow \mathbb{G}(B)$ we have $f':A\longrightarrow B$ defined for $T\in \mathcal{C}$ as  $f'_{T}=\theta_{B,T} \circ f_{T}:A(T)\longrightarrow B(T)$ where $\theta_{B,T}:\mathrm{Hom}_{\mathrm{Mod}(\mathcal{C})}(\mathrm{Hom}_{\mathcal{C}}(T,-),B)\longrightarrow B(T)$ is the Yoneda isomorphism. Consider the Yoneda isomorphism
$$Y_{A,U}:\mathrm{Hom}_{\mathrm{Mod}(\mathcal{C}^{op})}\Big(\mathrm{Hom}_{\mathcal{C}}(-,U),\mathbb{D}_{\mathcal{C}}(A)\Big)\longrightarrow \mathbb{D}_{\mathcal{C}}(A)(U).$$
Then the following diagram is commutative
$$(\ast\ast):\quad \xymatrix{\overline{\mathbb{G}}(\mathbb{D}_{\mathcal{C}}\mathbb{G}(B))(U)\ar[d]^{[\Psi'_{B}]_{U}}\ar[rr]^{\overline{[\mathbb{G}}\mathbb{D}_{\mathcal{C}}(f)]_{U}} & & \overline{\mathbb{G}}\mathbb{D}_{\mathcal{C}}(A)(U)\ar[d]^{Y_{A,U}}\\
\mathbb{D}_{\mathcal{C}}(B)(U)\ar[rr]^{[\mathbb{D}_{\mathcal{C}}(f')]_{U}} &  & \mathbb{D}_{\mathcal{C}}(A)(U).}$$
Let us see that the following diagram commutes
$$\xymatrix{\Big(\mathrm{Mod}(\mathcal{C}),\mathbb{G}\mathrm{Mod}(\mathcal{C})\Big)\ar[rr]^{\widehat{\Theta}}& & 
\Big(\mathrm{Mod}(\mathcal{C}^{op}),\overline{\mathbb{G}}\mathrm{Mod}(\mathcal{C}^{op})\Big)\\
\Big(\mathrm{Mod}(\mathcal{C}),\mathrm{Mod}(\mathcal{C})\Big)\ar[rr]^{\overline{\Theta}}\ar[u]^{J_{1}} && \Big(\mathrm{Mod}(\mathcal{C}^{op}),\mathrm{Mod}(\mathcal{C}^{op})\Big).\ar[u]^{J_{2}}}$$
Indeed,  given a map $f:A\longrightarrow B$ we have
$J_{1}(f)=\widehat{f}:A\longrightarrow \mathbb{G}(B),$
and clearly $J_{1}^{-1}(\widehat{f})=(\widehat{f})'=f$. Then by the diagrams $(\ast)$ and $(\ast\ast)$, we have that we have  the following commutative diagram

$$\xymatrix{\mathbb{D}_{\mathcal{C}}(B)(U)\ar[r]^(.4){[\Psi_{B}]_{U}}\ar@{=}[d] & \overline{\mathbb{G}}(\mathbb{D}_{\mathcal{C}}\mathbb{G}(B))(U)\ar[d]^{[\Psi'_{B}]_{U}}\ar[rr]^{\overline{[\mathbb{G}}\mathbb{D}_{\mathcal{C}}(\widehat{f})]_{U}} & & \overline{\mathbb{G}}\mathbb{D}_{\mathcal{C}}(A)(U)\ar[d]^{Y_{A,U}}\\
\mathbb{D}_{\mathcal{C}}(B)(U)\ar@{=}[r] & 
\mathbb{D}_{\mathcal{C}}(B)(U)\ar[rr]^{[\mathbb{D}_{\mathcal{C}}(f)]_{U}} &  & \mathbb{D}_{\mathcal{C}}(A)(U).}$$
We recall that 
$[\widehat{\Theta}(\widehat{f})]_{U}$ is given by the upper composition of the diagram above (see \cite[Proposition 4.9]{LeOS}). By considering the Yoneda isomorphism $$Y_{A,U}:\mathrm{Hom}_{\mathrm{Mod}(\mathcal{C}^{op})}\Big(\mathrm{Hom}_{\mathcal{C}}(-,U),\mathbb{D}_{\mathcal{C}}(A)\Big)\longrightarrow \mathbb{D}_{\mathcal{C}}(A)(U),$$
we have  that  $J_{2}^{-1}(\widehat{\Theta}(\widehat{f}))=\Big(\widehat{\Theta}(\widehat{f})\Big)'$ is such that for $U\in \mathcal{C}^{op}$ 
$$\Big(\widehat{\Theta}(\widehat{f})\Big)'_{U}:=Y_{A,U}\circ [\widehat{\Theta}(\widehat{f})]_{U}=Y_{A,U}\circ \Big( [\overline{\mathbb{G}}\mathbb{D}_{\mathcal{C}}(\widehat{f})]_{U}\circ [\Psi_{B}]_{U}\Big)=[\mathbb{D}_{\mathcal{C}}(f)]_{U}$$
(recall the construction of $J_{1}^{-1}$). Therefore, $J_{2}^{-1}\circ \widehat{\Theta}\circ J_{1}=\overline{\Theta}$ in objects. Is easy to show that the same happens in morphisms. Therefore we conclude that $\widehat{\Theta}\circ J_{1}=J_{2}\circ \overline{\Theta}$.
\end{proof}

\begin{corollary}\label{mapsfini}
Let $\mathcal{C}$ be a dualizing variety. Then with the conditions as in \ref{descrimaps}, we have a commutative diagram with $J_{1}$ and $J_{2}$ isomorphisms
$$\xymatrix{\Big(\mathrm{mod}(\mathcal{C}),\mathbb{G}\mathrm{mod}(\mathcal{C})\Big)\ar[rr]^{\widehat{\Theta}}& & 
\Big(\mathrm{mod}(\mathcal{C}^{op}),\overline{\mathbb{G}}\mathrm{mod}(\mathcal{C}^{op})\Big)\\
\Big(\mathrm{mod}(\mathcal{C}),\mathrm{mod}(\mathcal{C})\Big)\ar[rr]^{\overline{\Theta}}\ar[u]^{J_{1}} && \Big(\mathrm{mod}(\mathcal{C}^{op}),\mathrm{mod}(\mathcal{C}^{op})\Big)\ar[u]^{J_{2}}}$$
\end{corollary}
\begin{proof}
It follows from \ref{descrimaps} and \cite[Proposition 6.4]{LeOS}.
\end{proof}

\begin{proposition}\label{mapsmodlam}
Let $\mathcal{C}$ be a  $K$-variety and consider  the category
$\mathbf{\Lambda}=\left[ \begin{smallmatrix}
\mathcal C& 0 \\ 
 \widehat{\mathbbm{Hom}}& \mathcal C
\end{smallmatrix}\right]$.
\begin{itemize}
\item[(i)] There is an equivalence of categories
$$
\mathrm{Mod}(\mathbf{\Lambda})\xrightarrow{\sim} \mathrm{maps}(\mathrm{Mod}(\mathcal{C}))
$$
\item[(ii)]  If $\mathcal{C}$ is dualizing, there is an equivalence of categories
$$
\mathrm{mod}(\mathbf{\Lambda})\xrightarrow{\sim} \mathrm{maps}(\mathrm{mod}(\mathcal C))
$$
\end{itemize}
\end{proposition}
\begin{proof}
(i) Is proved in \cite[Theorem 3.14]{LeOS} that $\mathrm{Mod}(\mathbf{\Lambda})$   is equivalent to the comma category 
 $\Big(\mathrm{Mod}(\mathcal C),\mathbb{G}\mathrm{Mod}(\mathcal{C})\Big)$. Thus by \ref{descrimaps}, the category $\mathrm{Mod}\Big(\left[ \begin{smallmatrix}
\mathcal C& 0 \\ 
 \widehat{\mathbbm{Hom}}& \mathcal C
\end{smallmatrix}\right]\Big)$ is equivalent to the category which objects are morphisms of $\mathcal C$-modules 
$A\xrightarrow{f} B$, with 
 $A,B\in\mathrm{Mod}(\mathcal{C})$. In this way we have the equivalence of categories
 \begin{eqnarray*}\label{maps1}
 {\textswab{F}}\circ J_{1}:\mathrm{maps}(\mathrm{Mod}(\mathcal C))\rightarrow\mathrm{Mod}\Big(\left[ \begin{smallmatrix}
\mathcal C& 0 \\ 
 \widehat{\mathbbm{Hom}}& \mathcal C
\end{smallmatrix}\right]\Big).
\end{eqnarray*}

(ii)  Note that   $\widehat{\mathbbm{Hom}}_{C}=\mathrm{Hom}(C,-)\in\mathrm{mod}(\mathcal C)$  and 
$\widehat{\mathbbm{Hom}}_{C'}=\mathrm{Hom}(-,C')\in\mathrm{mod}(\mathcal C^{op})$, for all $C,C'\in\mathcal C$. Therefore, the equivalence follows from 
the fact that $\mathcal{C}$ is dualizing, \ref{mapsfini}  and \cite[Proposition 6.3]{LeOS}.
\end{proof}

In the following we will write $(C,-)$ and $(-,C)$ instead of $\mathrm{Hom}_{\mathcal{C}}(C,-)$ and $\mathrm{Hom}_{\mathcal{C}}(-,C)$.
 
\begin{proposition}\label{projinjec}
Let $\mathcal{C}$ be a $K$-variety. Then, 
\begin{itemize}
\item[(i)] 
$\mathrm{rad}\Big(\left[ \begin{smallmatrix}
 C_0& 0 \\ 
 \widehat{\mathbbm{Hom}}& C_1
\end{smallmatrix}\right],\left[ \begin{smallmatrix}
 C_0'& 0 \\ 
 \widehat{\mathbbm{Hom}}& C_1'
\end{smallmatrix}\right]\Big)=\left[ \begin{smallmatrix}
 \mathrm{rad}_{\mathcal C}(C_0,C_0')& 0 \\ 
 \mathrm{Hom}_{\mathcal C}(C_0,C_1')& \mathrm{rad}_{\mathcal C}(C_1,C_1')
\end{smallmatrix}\right]$

\item[(ii)]  Suppose that $\mathcal C$ is a dualizing variety.
\item[(a)] The indecomposable projective objects in  $\mathrm{maps}(\mathcal C)$ are objects which isomorphic to: objects of the form
$\Big((C,-),(1_C,-), (C,-)\Big)$ where $C$ is an indecomposable  object in $\mathcal C$; and to $\Big(0,0, (C,-)\Big)$ where $C$ is a indecomposable  object in $\mathcal C$.
\item[(b)]  The indecomposable injective objects in  $\mathrm{maps}(\mathcal C)$ are objects which are isomorphic to: objects of the form$\Big(\mathbb{D}_{\mathcal{C}}(C,-),\mathbb{D}_{\mathcal{C}}(1_C,-), \mathbb{D}_{\mathcal{C}}(C,-)\Big)$ where $C$ is an indecomposable  object in $\mathcal C$; and to $\Big(\mathbb{D}_{\mathcal{C}}(C,-)0,0\Big)$ where $C$ is an indecomposable  object in $\mathcal{C}$.

\end{itemize}

\end{proposition}
\begin{proof}
(i) It follows from \cite[Proposition 5.4]{LeOS}.

(ii) (a)  Let $P=\Big(\left[ \begin{smallmatrix}
 C& 0 \\ 
 \widehat{\mathbbm{Hom}}& C'
\end{smallmatrix}\right],-\Big)$ a projective object in $\mathrm{Mod}\Big(\left[ \begin{smallmatrix}
\mathcal C& 0 \\ 
 \widehat{\mathbbm{Hom}}& \mathcal C
\end{smallmatrix}\right]\Big)$. Consider the object $g:(C,-)\rightarrow \mathbb{G}(\widehat{\mathbbm{Hom}}_C\coprod (C',-))$.  Then, by the equivalence 
$${\textswab{F}}:\mathrm{maps}(\mathrm{Mod}(\mathcal C))\rightarrow\mathrm{Mod}\Big(\left[ \begin{smallmatrix}
\mathcal C& 0 \\ 
 \widehat{\mathbbm{Hom}}& \mathcal C
\end{smallmatrix}\right]\Big)$$
we get by  \cite[Proposition 5.4]{LeOS},
$$\textswab{F}\Big((C,-),g, \widehat{\mathbbm{Hom}}_C\coprod (C',-)\Big)\cong P.$$ 
Moreover, we have the following commutative diagram, where the vertical maps are isomorphisms in $\mathrm{maps}(\mathrm{mod}(\mathcal C))$.

\[
\begin{diagram}
\node{(C,-)}\arrow{e,t}{g}\arrow{s,l}{{(1_C,-)}}
 \node{\mathbb{G}(\widehat{\mathbbm{Hom}}_C\coprod (C',-))}\arrow{s,r}{Y=\{Y_C\}_{C\in\mathcal C} }\\
\node{(C,-)}\arrow{e,t}{\left[ \begin{smallmatrix}
 {(1_C,-)} \\ 
 0
\end{smallmatrix}\right]}
 \node{(C,-)\coprod (C',-).}
\end{diagram}
\]
Since $\mathcal{C}$ is a Krull-Schmidt  category, $C$ and $C'$ decomposes as $C=\coprod_{i=1}^n C_i$ and $C'=\coprod_{j=1}^m C_j'$
such that $\mathrm{End}_{\mathcal{C}}(C_i)$ and $\mathrm{End}_{\mathcal{C}}(C_j')$ are local rings. Thus, we have decompositions
$$\Big((C,-),1_{(C,-)}, (C,-)\Big)=\coprod_{i=1}^n \Big((C_i,-),1_{(C_i,-)}, (C_i,-)\Big)$$
 and 
$$\Big(0,0, (C',-)\Big)=\coprod_{j=1}^m \Big(0,0,(C_j',-)\Big)$$ for which  $\mathrm{End}_{\mathrm{maps}(\mathcal{C})}\Big((0,0,(C_j',-)\Big)\cong \mathrm{End}_{\mathcal{C}}(C_j')$
and  also we have the isomorphism $\mathrm{End}_{\mathrm{maps}(\mathcal{C})}\Big(((C_i,-),1_{(C_i,-)}, (C_i,-))\Big)\cong \mathrm{End}_{\mathcal{C}}(C_i)$.
\end{proof}

\subsection{Example}
Now, we describe a triangular matrix category $\mathbf{\Lambda}$  such that the category $\mathrm{Mod}(\mathbf{\Lambda})$ is equivalent to $\mathrm{maps}(\mathbf{Ch}(\mathrm{Mod} (K)))$.\\
Let $\Delta=(\Delta_0,\Delta_1)$ be the quiver with $\Delta_0=\mathbb  Z$ and  $\Delta_{1}=\{\alpha_i:i\rightarrow i+1\mid i\in\mathbb Z\}$,  with the set of relations $\rho=\{\alpha_{i}\alpha_{i-1}\mid i\in\mathbb Z\}$
 \[
\begin{diagram}
\node{\cdots}\arrow{e,t}{}
  \node{i-1}\arrow{e,t}{\alpha_{i-1}}
   \node{i}\arrow{e,t}{\alpha_{i}}
    \node{i+1}\arrow{e,t}{}
     \node{\cdots}
 \end{diagram}
\]
On the other hand, let $\widetilde{\Delta}=(\widetilde\Delta_0,\widetilde\Delta_1)$  be the quiver with $\widetilde\Delta_0=(\Delta_0\times\{1\})\cup (\Delta_0\times \{2\})$ and 
$$\widetilde\Delta_1=(\Delta_1\times\{1\})\bigcup (\Delta_1\times \{2\})\bigcup\Big\{\beta_i:(i,1)\rightarrow (i,2)\Big\}_{i\in\mathbb{Z} }\bigcup\Big\{\gamma_i:(i,1)\rightarrow (i+1,2)\Big\}_{i\in\mathbb{Z}}$$ with  relations  given by the set $$\widetilde\rho=\Big\{(\alpha_{i+1},1)(\alpha_i,1),\,\, (\alpha_{i+1},2)(\alpha_i,2),\,\,  (\alpha_i,2)\beta_i-\gamma_i,\,\, \beta_{i+1}(\alpha_i,1)-\gamma_i\Big\}_{i\in \mathbb{Z}}$$.

 \[
\begin{diagram}
\node{\cdots}\arrow{e,t}{}
\node{(i-1,1)}\arrow{e,t}{(\alpha_{i-1},1)}\arrow{se,t}{\gamma_{i-1}}\arrow{s,r}{\beta_{i-1}}
  \node{(i,1)}\arrow{e,t}{(\alpha_{i},1)}\arrow{se,t}{\gamma_{i}}\arrow{s,r}{\beta_{i}}
    \node{(i+1,1)}\arrow{e,t}{}\arrow{s,r}{\beta_{i+1}}
     \node{\cdots}\\
\node{\cdots}\arrow{e,t}{}
\node{(i-1,2)}\arrow{e,t}{(\alpha_{i-1},2)}
  \node{(i,2)}\arrow{e,t}{(\alpha_{i},2)}
    \node{(i+1,2)}\arrow{e,t}{}
     \node{\cdots}     
 \end{diagram}
\]
We will show that the category $\mathbf\Lambda=\left[ \begin{smallmatrix}
K\Delta/\langle\rho\rangle & 0 \\
\widehat{\mathbbm{{Hom}}} & K\Delta/\langle\rho\rangle
\end{smallmatrix} \right]$ with $\widehat{\mathbbm{{Hom}}}: K\Delta/\langle\rho\rangle\otimes (K\Delta/\langle\rho\rangle)^{op}\rightarrow \mathrm{Mod}\ K$ is equivalent to the category $K\widetilde\Delta/\langle\widetilde\rho\rangle.$\\
First, we note that we have two inclusion functors 
$$\varPhi_{1},\,\,\varPhi_{2}:K\Delta/\langle \rho\rangle \longrightarrow K\widetilde{\Delta}/\langle \widetilde\rho \rangle$$ defined as follows: for $i\in \Delta$ and $\alpha_{i}:i\longrightarrow i+1$ we set $\varPhi_{1}(i)=(i,1)$ and 
$\varPhi_{1}(\alpha_{i})=(\alpha_{i},1)$; and $\varPhi_{2}(i)=(i,2)$ and $\varPhi_{2}(\alpha_{i})=(\alpha_{i},2)$.\\
Now, we establish a functor  
$$\varPhi:\mathbf\Lambda\rightarrow K\widetilde\Delta/\langle\widetilde\rho\rangle$$  on objects  by
$\varPhi\Big(\left[ \begin{smallmatrix}
 i& 0 \\ 
 \widehat{\mathbbm{Hom}}& j
\end{smallmatrix}\right]\Big)=(i,1)\oplus(j,2)$, for all $i,j\in\mathbb{Z}$.
In order to define $\varPhi$ on morphisms, we note that
$$\mathrm{Hom}_{\mathbf\Lambda}\Big(\left[ \begin{smallmatrix}
 i& 0 \\ 
 \widehat{\mathbbm{Hom}}& j
\end{smallmatrix}\right],\left[ \begin{smallmatrix}
 i'& 0 \\ 
 \widehat{\mathbbm{Hom}}& j'
\end{smallmatrix}\right] \Big)=\left[ \begin{smallmatrix}
 \mathrm{Hom}_{K\Delta/\langle \rho\rangle}(i,i')& 0 \\ 
 \mathrm{Hom}_{K\Delta/\langle \rho\rangle}(i,j')& \mathrm{Hom}_{K\Delta/\langle \rho\rangle}(j,j')
\end{smallmatrix}\right]$$ 
and
$$\mathrm{Hom}_{K\widetilde\Delta/\langle \widetilde\rho\rangle} \Big((i,1)\oplus (j,2),(i',1)\oplus (j',2)\Big)=\left( \begin{smallmatrix}
 \mathrm{Hom}_{K\widetilde\Delta/\langle \widetilde\rho\rangle}((i,1),(i',1))& 0 \\ 
\mathrm{Hom}_{K\widetilde\Delta/\langle \widetilde\rho\rangle}((i,1),(j',2))& \mathrm{Hom}_{k\widetilde\Delta/\langle \widetilde\rho\rangle}((j,2),(j',2))
\end{smallmatrix}\right),$$
for all $i,i', j,j'\in\mathbb Z$.\\
Note that $\mathrm{Hom}_{K\Delta/\langle \rho\rangle}(i,i')=\mathrm{Hom}_{K\Delta/\langle \rho\rangle}(j,j')=\mathrm{Hom}_{K\Delta/\langle \rho\rangle}(i,j')=0$ unless $i'-i$, $j'-j$, $j'-i\in\{0,1\}$, and
$\mathrm{Hom}_{K\widetilde\Delta/\langle \widetilde\rho\rangle}((i,1),(i',1))= 
\mathrm{Hom}_{K\widetilde\Delta/\langle \widetilde\rho\rangle}((i,1),(j',2))= \mathrm{Hom}_{K\widetilde\Delta/\langle \widetilde\rho\rangle}((j,2),(j',2))=0$ unless $i'-i,j'-j,j'-i\in\{0,1\}$.\\
We  will define a morphism of abelian groups  
\begin{equation}\label{isoexample}
\varPhi:\mathrm{Hom}_{\mathbf\Lambda}\Big(\left[ \begin{smallmatrix}
 i& 0 \\ 
 \widehat{\mathbbm{Hom}}& j
\end{smallmatrix}\right],\left[ \begin{smallmatrix}
 i'& 0 \\ 
 \widehat{\mathbbm{Hom}}& j'
\end{smallmatrix}\right] \Big)\rightarrow \mathrm{Hom}_{K\widetilde\Delta/\langle \widetilde\rho\rangle} \Big((i,1)\oplus (j,2),(i',1)\oplus (j',2)\Big), 
\end{equation}
as follows:
 
\begin{itemize}
\item If $j'-i=0$, and $i'-i,j'-j\in \{0,1\}$. Consider
$$\left[ \begin{smallmatrix}
 \xi_{i} & &  0 \\ 
 \theta_{i} & & \eta_{j}
\end{smallmatrix}\right] \in\mathrm{Hom}_{\mathbf{\Lambda}}\Big(\left[ \begin{smallmatrix}
 i& 0 \\ 
 \widehat{\mathbbm{Hom}}& j
\end{smallmatrix}\right],\left[ \begin{smallmatrix}
 i' & 0 \\ 
 \widehat{\mathbbm{Hom}}& i
\end{smallmatrix}\right] \Big)=\left[ \begin{smallmatrix}
 \mathrm{Hom}_{K\Delta/\langle \rho\rangle}(i,i')& 0 \\ 
 \mathrm{Hom}_{K\Delta/\langle \rho\rangle}(i,i)& \mathrm{Hom}_{K\Delta/\langle \rho\rangle}(j,i)
\end{smallmatrix}\right],$$
since $\theta_{i}\in \mathrm{Hom}_{K\Delta/\langle \rho\rangle}(i,i)=K1_{i}$ we have that $\theta_{i}=\lambda 1_{i}$ for some $\lambda\in K$, thus we set
$$\varPhi \Big(\left[ \begin{smallmatrix}
 \xi_{i} & 0 \\ 
 \theta_{i} & \eta_{j}
\end{smallmatrix}\right]\Big)=\left[ \begin{smallmatrix}
 \varPhi_{1}(\xi_{i})& 0 \\ 
 \lambda \beta_{i} & \varPhi_{2}(\eta_{j})
\end{smallmatrix}\right].$$

\item If $j'-i=1$,  and $i'-i,j'-j\in \{0,1\}$.  Consider
$$\left[ \begin{smallmatrix}
 \xi_{i} & &  0 \\ 
 \theta_{i} & & \eta_{j}
\end{smallmatrix}\right] \in\mathrm{Hom}_{\mathbf{\Lambda}}\Big(\left[ \begin{smallmatrix}
 i& 0 \\ 
 \widehat{\mathbbm{Hom}}& j
\end{smallmatrix}\right],\left[ \begin{smallmatrix}
 i'& 0 \\ 
 \widehat{\mathbbm{Hom}}& i+1
\end{smallmatrix}\right] \Big)=\left[ \begin{smallmatrix}
 \mathrm{Hom}_{K\Delta/\langle \rho\rangle}(i,i')& 0 \\ 
 \mathrm{Hom}_{K\Delta/\langle \rho\rangle}(i,i+1)& \mathrm{Hom}_{K\Delta/\langle \rho\rangle}(j,i+1)
\end{smallmatrix}\right],$$
since $\theta_{i}\in \mathrm{Hom}_{K\Delta/\langle \rho\rangle}(i,i+1)=K\alpha_{i}$ we have that $\theta_{i}=\lambda\alpha_{i}$ for some $\lambda\in K$, thus
we set
$$\varPhi \Big(\left[ \begin{smallmatrix}
 \xi_{i} & 0 \\ 
 \theta_{i} & \eta_{j}
\end{smallmatrix}\right]\Big)=\left[ \begin{smallmatrix}
 \varPhi_{1}(\xi_{i})& 0 \\ 
 \lambda \gamma_{i} & \varPhi_{2}(\eta_{j})
\end{smallmatrix}\right].$$

\item If $j'-1\notin \{0,1\}$.  In this case, we set $\varPhi=0$.

\end{itemize}

In order to check that $\varPhi$ is a functor, consider the morphisms
$$\left[ \begin{smallmatrix}
 \xi_{i} & &  0 \\ 
 \theta_{i} & & \eta_{j}
\end{smallmatrix}\right] \in\mathrm{Hom}_{\mathbf{\Lambda}}\Big(\left[ \begin{smallmatrix}
 i& 0 \\ 
 \widehat{\mathbbm{Hom}}& j
\end{smallmatrix}\right],\left[ \begin{smallmatrix}
 i' & 0 \\ 
 \widehat{\mathbbm{Hom}}& j'
\end{smallmatrix}\right] \Big),\,\,\left[ \begin{smallmatrix}
 a_{i'} & &  0 \\ 
 \delta_{i'} & & u_{j'}
\end{smallmatrix}\right] \in\mathrm{Hom}_{\mathbf{\Lambda}}\Big(\left[ \begin{smallmatrix}
 i'& 0 \\ 
 \widehat{\mathbbm{Hom}}& j'
\end{smallmatrix}\right],\left[ \begin{smallmatrix}
 i'' & 0 \\ 
 \widehat{\mathbbm{Hom}}& j''
\end{smallmatrix}\right] \Big)$$

Then, we have that
\begin{align*}
\left[ \begin{smallmatrix}
 a_{i'} & &  0 \\ 
 \delta_{i'} & & u_{j'}
\end{smallmatrix}\right]\left[ \begin{smallmatrix}
 \xi_{i} & &  0 \\ 
 \theta_{i} & & \eta_{j}
\end{smallmatrix}\right] & =\left[ \begin{smallmatrix}
 a_{i'}\xi_{i} & &  0 \\ 
 \delta_{i'}\bullet \xi_{i} +u_{j'}\bullet \theta_{i} 
 && u_{j'}\eta_{j}
\end{smallmatrix}\right]\\
& =\left[ \begin{smallmatrix}
 a_{i'}\xi_{i} & &  0 \\ 
 \delta_{i'}\circ \xi_{i} +u_{j'}\circ \theta_{i} 
 && u_{j'}\eta_{j}
\end{smallmatrix}\right]\in \mathrm{Hom}_{\mathbf{\Lambda}}\Big(\left[ \begin{smallmatrix}
 i& 0 \\ 
 \widehat{\mathbbm{Hom}}& j
\end{smallmatrix}\right],\left[ \begin{smallmatrix}
 i'' & 0 \\ 
 \widehat{\mathbbm{Hom}}& j''
\end{smallmatrix}\right] \Big).
\end{align*}
In order to prove that $\varPhi$ is a functor we have several cases.
These cases are straightforward but tedious. For convenience of the reader we just illustrate some cases:\\

$(\mathbf{I})$ If $j'\notin \{i,i+1\}$;  and if $j''\notin \{i',i'+1\}$ we have that $\mathrm{Hom}_{K\Delta/\langle \rho\rangle}(i,j')=0$ and  $\mathrm{Hom}_{K\Delta/\langle \rho\rangle}(i',j'')=0$. In this cases, we have that $\theta_{i}=0=\delta_{i'}$ and therefore, on one side we have that
$$\varPhi\Big(\left[ \begin{smallmatrix}
 a_{i'} & &  0 \\ 
 0 & & u_{j'}
\end{smallmatrix}\right]\left[ \begin{smallmatrix}
 \xi_{i} & &  0 \\ 
 0 & & \eta_{j}
\end{smallmatrix}\right]\Big)=\varPhi\Big(\left[ \begin{smallmatrix}
 a_{i'}\xi_{i} & &  0 \\ 
 0 & & u_{j'}\eta_{j}
\end{smallmatrix}\right]\Big)=\left[ \begin{smallmatrix}
 \varPhi_{1}(a_{i'}\xi_{i}) & &  0 \\ 
 0 & & \varPhi_{2}(u_{j'}\eta_{j})
\end{smallmatrix}\right].$$
On the other side, we have that
\begin{align*}
\varPhi\Big(\left[ \begin{smallmatrix}
 a_{i'} & &  0 \\ 
 0 & & u_{j'}
\end{smallmatrix}\right]\Big)\varPhi\Big(\left[ \begin{smallmatrix}
 \xi_{i} & &  0 \\ 
 0 & & \eta_{j}
\end{smallmatrix}\right]\Big) & =\left[ \begin{smallmatrix}
 \varPhi_{1}(a_{i'}) & &  0 \\ 
 0 & & \varPhi_{2}(u_{j'})
\end{smallmatrix}\right]\left[ \begin{smallmatrix}
 \varPhi_{1}(\xi_{i}) & &  0 \\ 
 0 & & \varPhi_{2}(\eta_{j})
\end{smallmatrix}\right]\\
& =\left[ \begin{smallmatrix}
 \varPhi_{1}(a_{i'})\varPhi_{1}(\xi_{i}) & &  0 \\ 
 0 & & \varPhi_{2}(u_{j'})\varPhi_{2}(\eta_{j})
\end{smallmatrix}\right]\\
& =\left[ \begin{smallmatrix}
 \varPhi_{1}(a_{i'}\xi_{i}) & &  0 \\ 
 0 & & \varPhi_{2}(u_{j'}\eta_{j})
\end{smallmatrix}\right].
\end{align*}
Then, 
$\varPhi\Big(\left[ \begin{smallmatrix}
 a_{i'} & &  0 \\ 
 \delta_{i'}
 && u_{i}
\end{smallmatrix}\right]\Big)\varPhi\Big(\left[ \begin{smallmatrix}
 \xi_{i} & &  0 \\ 
 \theta_{i} 
 && \eta_{j}
\end{smallmatrix}\right]\Big)=\varPhi\Big(\left[ \begin{smallmatrix}
 a_{i'} & &  0 \\ 
 \delta_{i'}
 && u_{i}
\end{smallmatrix}\right]
\left[ \begin{smallmatrix}
 \xi_{i} & &  0 \\ 
 \theta_{i} 
 && \eta_{j}
\end{smallmatrix}\right]\Big).$\\

$(\mathbf{II})$ Suppose that $j'\in \{i,i+1\}$  and if $j''\in \{i',i'+1\}$. If  $j'=i$ and $j''=i'$.
Consider two morphisms
$$\left[ \begin{smallmatrix}
 \xi_{i} & &  0 \\ 
 \theta_{i} & & \eta_{j}
\end{smallmatrix}\right] \in\mathrm{Hom}_{\mathbf{\Lambda}}\Big(\left[ \begin{smallmatrix}
 i& 0 \\ 
 \widehat{\mathbbm{Hom}}& j
\end{smallmatrix}\right],\left[ \begin{smallmatrix}
 i' & 0 \\ 
 \widehat{\mathbbm{Hom}}& i
\end{smallmatrix}\right] \Big),\,\,\left[ \begin{smallmatrix}
 a_{i'} & &  0 \\ 
 \delta_{i'} & & u_{i}
\end{smallmatrix}\right] \in\mathrm{Hom}_{\mathbf{\Lambda}}\Big(\left[ \begin{smallmatrix}
 i'& 0 \\ 
 \widehat{\mathbbm{Hom}}& i
\end{smallmatrix}\right],\left[ \begin{smallmatrix}
 i'' & 0 \\ 
 \widehat{\mathbbm{Hom}}& i'
\end{smallmatrix}\right] \Big).$$
Then,
$\left[ \begin{smallmatrix}
 a_{i'} & &  0 \\ 
 \delta_{i'} & & u_{i}
\end{smallmatrix}\right]\left[ \begin{smallmatrix}
 \xi_{i} & &  0 \\ 
 \theta_{i} & & \eta_{j}
\end{smallmatrix}\right]=\left[ \begin{smallmatrix}
 a_{i'}\xi_{i} & &  0 \\ 
 \delta_{i'}\bullet \xi_{i} +u_{i}\bullet \theta_{i} 
 && u_{i}\eta_{j}
\end{smallmatrix}\right]=\left[ \begin{smallmatrix}
 a_{i'}\xi_{i} & &  0 \\ 
 \delta_{i'}\circ \xi_{i} +u_{i}\circ \theta_{i} 
 && u_{i}\eta_{j}
\end{smallmatrix}\right]$.\\

$(\mathbf{IIa})$ If $i=i'$ we have that $\xi_{i}=k_{1}1_{i}$, $\theta_{i}=k_{2}1_{i}$, $\delta_{i'}=k_{3}1_{i}$ and $u_{i}=k_{4}1_{i}$ for some $k_{1},k_{2},k_{3},k_{4}\in K$. Then, we get that
$$\left[ \begin{smallmatrix}
 a_{i'}\xi_{i} & &  0 \\ 
 \delta_{i'}\circ \xi_{i} +u_{i}\circ \theta_{i} 
 && u_{i}\eta_{j}
\end{smallmatrix}\right]=\left[ \begin{smallmatrix}
 a_{i'}\xi_{i} & &  0 \\ 
 (k_{1}k_{3}+k_{2}k_{4})1_{i}
 && u_{i}\eta_{j}
\end{smallmatrix}\right]=\left[ \begin{smallmatrix}
 k_{1}a_{i'} & &  0 \\ 
 (k_{1}k_{3}+k_{2}k_{4})1_{i}
 && k_{4}\eta_{j}
\end{smallmatrix}\right].$$ 
On one side, we have that
$$\varPhi(\left[ \begin{smallmatrix}
 a_{i'}\xi_{i} & &  0 \\ 
 \delta_{i'}\circ \xi_{i} +u_{i}\circ \theta_{i} 
 && u_{i}\eta_{j}
\end{smallmatrix}\right])=\left[ \begin{smallmatrix}
 \varPhi_{1}(a_{i'}\xi_{i})& 0 \\ 
(k_{1}k_{3}+k_{2}k_{4})\beta_{i} & \varPhi_{2}(u_{i}\eta_{j})
\end{smallmatrix}\right]=\left[ \begin{smallmatrix}
 k_{1}\varPhi_{1}(a_{i'}) & &  0 \\ 
 (k_{1}k_{3}+k_{2}k_{4})\beta_{i}
 && k_{4}\varPhi_{2}(\eta_{j})
\end{smallmatrix}\right].$$
On the other side, we have that 
$$\varPhi(\left[ \begin{smallmatrix}
 \xi_{i} & &  0 \\ 
 \theta_{i} 
 && \eta_{j}
\end{smallmatrix}\right])=\left[ \begin{smallmatrix}
 \varPhi_{1}(\xi_{i})& 0 \\ 
 k_{2}\beta_{i} & \varPhi_{2}(\eta_{j})
\end{smallmatrix}\right]=\left[ \begin{smallmatrix}
 k_{1}(1_{i},1) & &  0 \\ 
 k_{2}\beta_{i}
 && \varPhi_{2}(\eta_{j})
\end{smallmatrix}\right]$$ and 
$$\varPhi(\left[ \begin{smallmatrix}
 a_{i'} & &  0 \\ 
 \delta_{i'}
 && u_{i}
\end{smallmatrix}\right])=\left[ \begin{smallmatrix}
 \varPhi_{1}(a_{i'})& 0 \\ 
 k_{3}\beta_{i} & \varPhi_{2}(u_{i})
\end{smallmatrix}\right]=\left[ \begin{smallmatrix}
 \varPhi_{1}(a_{i'}) & &  0 \\ 
 k_{3}\beta_{i}
 && k_{4}(1_{i},2)
\end{smallmatrix}\right].$$ 
Therefore, we get that
\begin{align*}
\left[ \begin{smallmatrix}
 \varPhi_{1}(a_{i'})& 0 \\ 
 k_{3}\beta_{i} & k_{4}(1_{i},2)
\end{smallmatrix}\right]\left[ \begin{smallmatrix}
 k_{1}(1_{i},1)& 0 \\ 
 k_{2}\beta_{i} & \varPhi_{2}(\eta_{j})
\end{smallmatrix}\right] & =\left[ \begin{smallmatrix}
k_{1}\varPhi_{1}(a_{i'}) & 0 \\ 
(k_{3}\beta_{i})\circ (k_{1}(1_{i},1)) + (k_{4}(1_{i},2))\circ (k_{2}\beta_{i}) & k_{4}\varPhi_{2}(\eta_{j})
\end{smallmatrix}\right]\\
& = \left[ \begin{smallmatrix}
k_{1}\varPhi_{1}(a_{i'}) & 0 \\ 
(k_{3}k_{1})\beta_{i} + (k_{4}k_{2})\beta_{i} & k_{4}\varPhi_{2}(\eta_{j})
\end{smallmatrix}\right]\\
& = \left[ \begin{smallmatrix}
k_{1}\varPhi_{1}(a_{i'}) & 0 \\ 
(k_{3}k_{1}+k_{4}k_{2})\beta_{i} & k_{4}\varPhi_{2}(\eta_{j})
\end{smallmatrix}\right],
\end{align*}
and thus
$\varPhi\Big(\left[ \begin{smallmatrix}
 a_{i'} & &  0 \\ 
 \delta_{i'}
 && u_{i}
\end{smallmatrix}\right]\Big)\varPhi\Big(\left[ \begin{smallmatrix}
 \xi_{i} & &  0 \\ 
 \theta_{i} 
 && \eta_{j}
\end{smallmatrix}\right]\Big)=\varPhi\Big(\left[ \begin{smallmatrix}
 a_{i'} & &  0 \\ 
 \delta_{i'}
 && u_{i}
\end{smallmatrix}\right]
\left[ \begin{smallmatrix}
 \xi_{i} & &  0 \\ 
 \theta_{i} 
 && \eta_{j}
\end{smallmatrix}\right]\Big).$
Then, $\varPhi:\mathbf{\Lambda}\longrightarrow K\widetilde{\Delta}/\langle \widetilde{\rho}\rangle$ is a functor. Now, it is easy to show that
$$\varPhi:\mathrm{Hom}_{\mathbf\Lambda}\Big(\left[ \begin{smallmatrix}
 i& 0 \\ 
 \widehat{\mathbbm{Hom}}& j
\end{smallmatrix}\right],\left[ \begin{smallmatrix}
 i'& 0 \\ 
 \widehat{\mathbbm{Hom}}& j'
\end{smallmatrix}\right] \Big)\longrightarrow \mathrm{Hom}_{K\widetilde\Delta/\langle \widetilde\rho\rangle} \Big((i,1)\oplus (j,2),(i',1)\oplus (j',2)\Big)$$
is an isomorphism of abelian groups. Since $\varPhi$ is clearly a dense functor we conclude that $\varPhi$ is an equivalence.
Then $\mathrm{Mod}(\mathbf{\Lambda})$ is equivalent to $\mathrm{Mod}( K\widetilde{\Delta}/\langle \widetilde{\rho}\rangle)$. But $\mathrm{Mod}( K\widetilde{\Delta}/\langle \widetilde{\rho}\rangle)$ is the category $\mathrm{maps}(\mathbf{Ch}(\mathrm{Mod}(K)))$, proving our assertion.


\section{Auslander-Reiten translate in the category of maps}

Let $R$ be a commutative ring. Almost split sequences  for dualizing varietes were studied by M. Auslander and Idun Reiten in $\mathrm{mod}(\mathcal C)$ for  a $R$-dualizing variety  $\mathcal C$ as a generalization of the concept of  Almost split sequences  for  $\mathrm{mod}(\Lambda)$ for  an artin $R$-algebra  (see \cite{Auslander2}).  The crucial ingredient is the explicit construction of the
Auslander-Reiten translate $\tau$ by taking the dual of the transpose  $D \mathrm{Tr} M$ of a finitely presented $\mathcal C$-module $M$. The duality  $\mathbb{D}_{\mathcal{C}}:\mathrm{mod}(\mathcal{C})\rightarrow \mathrm{mod}(\mathcal{C}^{op})$ for dualizing varieties $\mathcal{C}$ is given in  definition \ref{dualizinvar} and the transpose $\mathrm{Tr}: \mathrm{mod}(\mathcal{C})\rightarrow \mathrm{mod}(\mathcal{C}^{op})$ is defined as follows: consider the functor
$(-)^{*}: \mathrm{mod}(\mathcal{C})\rightarrow \mathrm{mod}(\mathcal{C}^{op})$ defined by 
$(M)^{*}(C)=(M, (C,-))$, for all  $M\in  \mathrm{mod}(\mathcal{C})$, and $C\in\mathcal C$, then take a  minimal projective resolution
 for $M$: $(X,-)\xrightarrow{(f,-)}(Y,-)\rightarrow M\rightarrow 0$, thus $\mathrm{Tr} M:= \mathrm{Coker} \ ((f,-)^{\ast})$.\\
For that reason in this section, we study the transpose and the dual in the category of maps. In the same way as in the classic case, we have a duality $(-)^{\ast}:\mathrm{proj}\Big( \mathrm{mod}(\mathcal{C}),\mathrm{mod}(\mathcal{C})\Big)\longrightarrow
\mathrm{proj}\Big( \mathrm{mod}(\mathcal{C}^{op}),\mathrm{mod}(\mathcal{C}^{op})\Big)$  between the projectives in the category of maps (see proposition \ref{dualprojmaps}).\\ 
One of the main results in this section is to describe the Auslander-Reiten translate in the category of maps which will be denoted by $\mathrm{Tau}$. In particular, we show that  if $f:C_{1}\longrightarrow C_{2}$ is a morphism in $\mathrm{mod}(\mathcal{C})$ such that there exists exact sequence $\xymatrix{C_{1}\ar[r]^{f} & C_{2}\ar[r] & C_{3}\ar[r] & 0}$ with $C_{3}\neq 0$ and $C_{3}$ not projective. Then $$\mathrm{Tau}(C_{1},f,C_{2})=(\mathbb{D}_{\mathcal{C}^{op}}(Y),\mathbb{D}_{\mathcal{C}^{op}}(g),\mathbb{D}_{\mathcal{C}^{op}}\mathrm{Tr}(C_{3}))$$ for some morphism $g:\mathrm{Tr}(C_{3})\longrightarrow Y$ such that there exists an exact sequence
$\xymatrix{0\ar[r] &\mathbb{D}_{\mathcal{C}^{op}}\mathrm{Tr}(C_{1})\ar[r] & \mathbb{D}_{\mathcal{C}^{op}}(Y)\ar[r]^{\mathbb{D}_{\mathcal{C}^{op}}(g)} & \mathbb{D}_{\mathcal{C}^{op}}\mathrm{Tr}(C_{3})}$ and where $\mathrm{Tr}$ denotes the Auslander-Reiten translate in $\mathrm{mod}(\mathcal{C})$ (see theorem \ref{ARsequncedesc}).\\
In order to have all the ingredients to prove the above result, we consider  $\mathbf{\Lambda}=\left[
\begin{smallmatrix}
T&0 \\
M&U
\end{smallmatrix} \right]$ the matrix triangular category.
Now, we recall the construction of a  functor $(-)^{\ast}: \mathrm{Mod}(\mathbf{\Lambda})\longrightarrow \mathrm{Mod}(\mathbf{\Lambda}^{op})$ which is a generalization of the functor $ \mathrm{Mod}(\Lambda)\longrightarrow \mathrm{Mod}(\Lambda^{op})$ given by 
$M\mapsto \mathrm{Hom}_{\Lambda}(M,\Lambda)$ for all $\Lambda$-modules $M$, where $\Lambda$ is an artin algebra.\\
For each $\mathbf{\Lambda}$-module $M$ define $M^{\ast}:\mathbf{\Lambda}^{op}\longrightarrow \mathbf{Ab}$ by 
$$M^{\ast}\Big(\left[
\begin{smallmatrix}
T&0 \\
M&U
\end{smallmatrix} \right]\Big):=\mathrm{Hom}_{\mathrm{Mod}(\mathbf{\Lambda})}\Big(M,\mathrm{Hom}_{\mathbf{\Lambda}}\Big(\left[
\begin{smallmatrix}
T&0 \\
M&U
\end{smallmatrix} \right],-\Big)\Big)$$
and $M^{\ast}$ is defined in morphisms in the obvious way.  Then we have a contravariant functor 
$(-)^{\ast}: \mathrm{Mod}(\mathbf{\Lambda})\longrightarrow \mathrm{Mod}(\mathbf{\Lambda}^{op})$.

Now, taking into account that we have equivalences given in \cite[Theorem 3.14]{LeOS}
$$\xymatrix{
\Big(\mathrm{Mod}(\mathcal{T}),\mathbb{G}\mathrm{Mod}(\mathcal{U})\Big)\ar[rr]^{\textswab{F}} & & \mathrm{Mod}(\mathbf{\Lambda})}$$
$$\xymatrix{\Big(\mathrm{Mod}(\mathcal{U}^{op}),\overline{\mathbb{G}}\mathrm{Mod}(\mathcal{T}^{op})\Big)\ar[rr]_{\mathbb{T}^{\ast}\circ \overline{\textswab{F}}} & &   \mathrm{Mod}(\mathbf{\Lambda}^{op})}$$
we define a contravariant functor $\Psi:=(\mathbb{T}^{\ast}\circ \overline{\textswab{F}})^{-1}\circ (-)^{\ast}\circ \textswab{F}:\Big(\mathsf{Mod}(\mathcal{T}),\mathbb{G}\mathsf{Mod}(\mathcal{U})\Big)\longrightarrow  \Big(\mathsf{Mod}(\mathcal{U}^{op}),\overline{\mathbb{G}}\mathsf{Mod}(\mathcal{T}^{op})\Big)$
which we will denote also by $(-)^{\ast}$,
such that the following diagram commutes up to a natural equivalence
$$\xymatrix{
\Big(\mathrm{Mod}(\mathcal{T}),\mathbb{G}\mathrm{Mod}(\mathcal{U})\Big)\ar[rr]^{\textswab{F}}\ar[d]_{\Psi=(-)^{\ast}} & & \mathrm{Mod}(\mathbf{\Lambda})\ar[d]^{(-)^{\ast}}\\
\Big(\mathrm{Mod}(\mathcal{U}^{op}),\overline{\mathbb{G}}\mathrm{Mod}(\mathcal{T}^{op})\Big)\ar[rr]_{\mathbb{T}^{\ast}\circ \overline{\textswab{F}}} & &   \mathsf{Mod}(\mathbf{\Lambda}^{op}).}$$

\begin{remark}\label{estrellaP}
It is easy to show that if $P:=\mathrm{Hom}_{\mathbf{\Lambda}}\Big(\left[
\begin{smallmatrix}
T&0 \\
M&U
\end{smallmatrix} \right],-\Big):\mathbf{\Lambda} \rightarrow \mathbf{Ab}$, then
$P^{\ast}:=\mathrm{Hom}_{\mathbf{\Lambda}}\Big(-,\left[
\begin{smallmatrix}
T&0 \\
M&U
\end{smallmatrix} \right]\Big).$
\end{remark}

\begin{remark}\label{projectives}
We recall the following result from \cite{LeOS}:\\
$(a)$  Consider the projective $\mathbf \Lambda$-module, $P:=\mathrm{Hom}_{\mathbf{\Lambda}}\Big(\left[
\begin{smallmatrix}
T&0 \\
M&U
\end{smallmatrix} \right],-\Big):\mathbf{\Lambda}\rightarrow \mathbf{Ab}$ and the morphism of $\mathcal T$-modules 
$g:\mathrm{Hom}_{\mathcal{T}}(T,-)\longrightarrow \mathbb{G}\Big( M_T\amalg \mathrm{Hom}_{\mathcal{U}}(U,-)\Big )$
given by $g:=\Big \{[g]_{T'}: \mathrm{Hom}_{\mathcal{T}}(T,T')\longrightarrow \mathrm{Hom}_{\mathcal {U}}\big(M_{T'},M_{T}\amalg \mathrm{Hom}_{\mathcal{U}}(U,-)\big)\Big \}_{T'\in\mathcal {T}}$, with  $[g]_{T'}(t):=\left[
\begin{smallmatrix}
\overline{t} \\
0
\end{smallmatrix} \right]:M_{T'}\rightarrow M_{T}\amalg \mathrm{Hom}_{\mathcal{U}}(U,-)$ for all $t\in \mathrm{Hom}_{\mathcal{T}}(T,T')$.
Then
$$P\cong \big(\mathrm{Hom}_{\mathcal{T}}(T,-)\big)\underset{g}\amalg \big(M_T\amalg \mathrm{Hom}_{\mathcal{U}}(U,-)\big).$$\\
$(b)$ Consider the projective $\overline{\mathbf {\Lambda}}$-module, $P:=\mathrm{Hom}_{\overline{\mathbf{\Lambda}}}\Big(\left[
\begin{smallmatrix}
U&0 \\
\overline{M}&T
\end{smallmatrix} \right],-\Big):\overline{\mathbf{\Lambda}} \rightarrow \mathbf{Ab}$ and the morphism of $\mathcal{U}^{op}$-modules $\overline{g}:\mathrm{Hom}_{\mathcal{U}^{op}}(U,-)\longrightarrow \overline{\mathbb{G}}\Big( \overline{M}_U\amalg \mathrm{Hom}_{\mathcal{T}^{op}}(T,-)\Big )$
given by $\overline{g}:=\Big \{[\overline{g}]_{U'}: \mathrm{Hom}_{\mathcal{U}^{op}}(U,U')\longrightarrow \mathrm{Hom}_{\mathcal {T}^{op}}\big(\overline{M}_{U'},\overline{M}_{U}\amalg \mathrm{Hom}_{\mathcal{T}^{op}}(T,-)\big)\Big \}_{U'\in\mathcal {U}^{op}}$, with  $[\overline{g}]_{U'}(u^{op}):=\left[
\begin{smallmatrix}
\overline{u^{op}} \\
0
\end{smallmatrix} \right]:\overline{M}_{U'}\rightarrow \overline{M}_{U}\amalg \mathrm{Hom}_{\mathcal{T}^{op}}(T,-)$ for all $u^{op}\in \mathrm{Hom}_{\mathcal{U}^{op}}(U,U')$. Then
$$P\cong \big(\mathrm{Hom}_{\mathcal{U}^{op}}(U,-)\big)\underset{\overline{g}}\amalg \big(\overline{M}_U\amalg \mathrm{Hom}_{\mathcal{T}^{op}}(T,-)\big).$$
We note that since $\overline{u^{op}}:=\overline{u}:M_{U'}\longrightarrow M_{U}$ and $\mathrm{Hom}_{\mathcal{U}^{op}}(U,-)\simeq \mathrm{Hom}_{\mathcal{U}}(-,U)$ and $\mathrm{Hom}_{\mathcal{T}^{op}}(T,-)\simeq \mathrm{Hom}_{\mathcal{T}}(-,T)$
we can think $\overline{g}$ of the following form
$$\overline{g}:\mathrm{Hom}_{\mathcal{U}}(-,U)\longrightarrow \overline{\mathbb{G}}\Big(M_U\amalg \mathrm{Hom}_{\mathcal{T}}(-,T)\Big )$$ given by
$\overline{g}:=\Big \{[\overline{g}]_{U'}: \mathrm{Hom}_{\mathcal{U}}(U',U)\longrightarrow \mathrm{Hom}_{\mathcal {T}^{op}}\big(M_{U'},M_{U}\amalg \mathrm{Hom}_{\mathcal{T}}(-,T)\big)\Big \}_{U'\in\mathcal {U}}$, with  $[\overline{g}]_{U'}(u):=\left[
\begin{smallmatrix}
\overline{u} \\
0
\end{smallmatrix} \right]:M_{U'}\rightarrow M_{U}\amalg \mathrm{Hom}_{\mathcal{T}}(-,T)$ for all $u\in \mathrm{Hom}_{\mathcal{U}}(U',U)$.
\end{remark}

By section 5 in \cite{LeOS}, we know that there exists a functor $\mathbb{F}:\mathrm{Mod}(\mathcal{T})\longrightarrow \mathrm{Mod}(\mathcal{U})$ such that $\mathbb{F}$ is left adjoint to $\mathbb{G}$. That is, there exist a natural bijection
$$\varphi_{A,B}:\mathrm{Hom}_{\mathrm{Mod}(\mathcal{U})}(\mathbb{F}(A),B)\longrightarrow \mathrm{Hom}_{\mathrm{Mod}(\mathcal{T})}(A,\mathbb{G}(B)).$$

\begin{proposition}\label{adjunotro}
Consider the isomorphism of categories given in \cite[Proposition 5.3]{LeOS}
$$H:	\Big(\mathbb{F}(\mathrm{Mod}(\mathcal{T})),\mathrm{Mod}(\mathcal{U})\Big)\longrightarrow \Big(\mathrm{Mod}(\mathcal{T}),\mathbb{G}(\mathrm{Mod}(\mathcal{U}))\Big).$$
and the object $\left[
\begin{smallmatrix}
1_{M_{T}} \\
0
\end{smallmatrix} \right]:\mathbb{F}(\mathrm{Hom}_{\mathcal{T}}(T,-))=M_{T}\longrightarrow M_{T}\amalg \mathrm{Hom}_{\mathcal{U}}(U,-)$ (see \cite[Lemma 5.7]{LeOS}),  in the category $\Big(\mathbb{F}(\mathrm{Mod}(\mathcal{T})),\mathrm{Mod}(\mathcal{U})\Big)$. Then
$H\Big(\left[
\begin{smallmatrix}
1_{M_{T}} \\
0
\end{smallmatrix} \right]\Big)$  corresponds to the object $g:\mathrm{Hom}_{\mathcal{T}}(T,-)\longrightarrow \mathbb{G}\Big( M_T\amalg \mathrm{Hom}_{\mathcal{U}}(U,-)\Big )$ in the category $\Big(\mathrm{Mod}(\mathcal{T}),\mathbb{G}(\mathrm{Mod}(\mathcal{U}))\Big).$

\end{proposition}
\begin{proof}
 Let $h:\mathbb{F}(A)\longrightarrow B$ be an object in $\Big(\mathbb{F}(\mathrm{Mod}(\mathcal{T})),\mathrm{Mod}(\mathcal{U})\Big)$ and consider the bijection
$\varphi_{A,B}:\mathrm{Hom}_{\mathrm{Mod}(\mathcal{U})}(\mathbb{F}(A),B)\longrightarrow \mathrm{Hom}_{\mathrm{Mod}(\mathcal{T})}(A,\mathbb{G}(B)).$\\
By definition, we have that
$H(A,h,B):=(A,\varphi_{A,B}(h),B)$.
Then we have
$$\varphi:=\varphi_{\mathrm{Hom}_{\mathcal{T}}(T,-),M_{T}}:\mathrm{Hom}_{\mathrm{Mod}(\mathcal{U})}(M_{T},M_{T})\longrightarrow \mathrm{Hom}_{\mathrm{Mod}(\mathcal{T})}(\mathrm{Hom}_{\mathcal{T}}(T,-),\mathbb{G}(M_{T}))$$
because $\mathbb{F}(\mathrm{Hom}_{\mathcal{T}}(T,-))=M_{T}$ (see \cite[Lemma 5.7(i)]{LeOS}).
By \cite[Theorem 6.3] {Popescu}, we have that in this case the isomorphism $\varphi$ coincides with the Yoneda isomorphism. Then for $\lambda:M_{T}\longrightarrow M_{T}$, we have that
$\varphi(\lambda):\mathrm{Hom}_{\mathcal{T}}(T,-)\longrightarrow \mathbb{G}(M_{T})$ is such that for $T'\in \mathcal{T}$
$$[\varphi(\lambda)]_{T'}:\mathrm{Hom}_{\mathcal{T}}(T,T')\longrightarrow \mathbb{G}(M_{T})(T')=\mathrm{Hom}(M_{T'},M_{T})$$ is defined as $[\varphi(\lambda)]_{T'}(t)=\Big(\mathbb{G}(M_{T})(t)\Big)(\lambda)=\mathrm{Hom}_{\mathrm{Mod}(\mathcal{U})}(\overline{t},M_{T})(\lambda)=\lambda\circ \overline{t},$
for $t:T\longrightarrow T'$. Then for $1_{M_{T}}:M_{T}\longrightarrow M_{T}$ we have that
$\varphi(1_{M_{T}}):\mathrm{Hom}_{\mathcal{T}}(T,-)\longrightarrow \mathbb{G}(M_{T})$ is such that for $T'\in \mathcal{T}$
$$[\varphi(1_{M_{T}}]_{T'}:\mathrm{Hom}_{\mathcal{T}}(T,T')\longrightarrow \mathbb{G}(M_{T})(T')=\mathrm{Hom}(M_{T'},M_{T})$$ is defined as $[\varphi(1_{M_{T}})]_{T'}(t)=\Big(\mathbb{G}(M_{T})(t)\Big)(1_{M_{T}})=\mathrm{Hom}_{\mathrm{Mod}(\mathcal{U})}(\overline{t},M_{T})(1_{M_{T}})=\overline{t}$.\\
Since $\mathbb{G}\Big( M_T\amalg \mathrm{Hom}_{\mathcal{U}}(U,-)\Big )=\mathbb{G}\Big( M_T\Big)\amalg\mathbb{G}\Big( \mathrm{Hom}_{\mathcal{U}}(U,-)\Big )$, we  can see $g$ as follows
$$g=\left[
\begin{smallmatrix}
g_{1}\\
g_{2}
\end{smallmatrix} \right]:\mathrm{Hom}_{\mathcal{T}}(T,-)\longrightarrow  \mathbb{G}\Big( M_T\Big)\amalg\mathbb{G}\Big( \mathrm{Hom}_{\mathcal{U}}(U,-)\Big ).$$
It is straighforward to show that $g_{1}=\varphi(1_{M_{T}})$ and $g_{2}=0$.
Then 
$$g=\left[
\begin{smallmatrix}
\varphi(1_{M_{T}})\\
0
\end{smallmatrix} \right]=\left[
\begin{smallmatrix}
\varphi(1_{M_{T}})\\
\varphi(0)
\end{smallmatrix} \right]=\varphi\Big(\left[
\begin{smallmatrix}
1_{M_{T}}\\
0
\end{smallmatrix} \right]\Big).$$
Therefore,
$$H\Big(\mathrm{Hom}_{\mathcal{T}}(T,-),\left[
\begin{smallmatrix}
1_{M_{T}}\\
0
\end{smallmatrix} \right],M_{T}\amalg \mathrm{Hom}_{\mathcal{U}}(U,-)\Big)=\Big(\mathrm{Hom}_{\mathcal{T}}(T,-),g,M_{T}\amalg \mathrm{Hom}_{\mathcal{U}}(U,-)\Big).$$

\end{proof}

\begin{lemma}
There is isomorphism of $\mathrm{Mod}(\mathcal{U}^{op})$-modules $\lambda: \mathrm{Hom}_{\mathcal{U}}(-,U)\longrightarrow \mathrm{Hom}_{\mathcal{U}^{op}}(U,-)$ given by
$\lambda_{U'}:\mathrm{Hom}_{\mathcal{U}}(U',U)\longrightarrow 
\mathrm{Hom}_{\mathcal{U}^{op}}(U,U') $ as $\lambda_{U'}(\alpha):=\alpha^{op}$. 
\end{lemma}
\begin{proof}
Straightforward.
\end{proof}

\begin{proposition}\label{dualproyec}
Consider the projective objects
$$\xymatrix{\mathrm{Hom}_{\mathcal{T}}(T,-)\ar[r]^(.4){g}& \mathbb{G}(M_{T}\amalg \mathrm{Hom}_{\mathcal{U}}(U,-))}\in \Big(\mathsf{Mod}(\mathcal{T}),\mathbb{G}\mathsf{Mod}(\mathcal{U})\Big)$$
$$\xymatrix{\mathrm{Hom}_{\mathcal{U}}(-,U)\ar[r]^(.4){\overline{g}} & \overline{\mathbb{G}}(M_{U}\amalg \mathrm{Hom}_{\mathcal{T}}(-,T))}\in \Big(\mathsf{Mod}(\mathcal{U}^{op}),\overline{\mathbb{G}}\mathsf{Mod}(\mathcal{T}^{op})\Big)$$ 
as in \ref{projectives}.
Then $g^{\ast}=\overline{g}$, that is:
$$\Big(\mathrm{Hom}_{\mathcal{T}}(T,-)\stackrel{g}{\longrightarrow }\mathbb{G}(M_{T}\amalg \mathrm{Hom}_{\mathcal{U}}(U,-))\Big)^{\ast}\!\!\!\!=\!\!
\mathrm{Hom}_{\mathcal{U}}(-,U)\stackrel{\overline{g}}{\longrightarrow} \overline{\mathbb{G}}(M_{U}\amalg \mathrm{Hom}_{\mathcal{T}}(-,T)).$$
\end{proposition}
\begin{proof}
Consider the equivalences given in \cite[Theorem 3.14]{LeOS}, and the induced by the functor $\mathbb{T}:\mathbf{\Lambda}^{op}\longrightarrow \overline{\mathbf{\Lambda}}$ given in \cite[Proposition 4.3]{LeOS}, 
$$\textswab{F}:\Big( \mathsf{Mod}(\mathcal{T}), \mathbb{G}\mathsf{Mod}(\mathcal{U})\Big) \longrightarrow 
\mathrm{Mod}(\mathbf{\Lambda}),\quad \overline{\textswab{F}}:\Big( \mathsf{Mod}(\mathcal{U}^{op}), \overline{\mathbb{G}}\mathsf{Mod}(\mathcal{T}^{op})\Big) \longrightarrow 
\mathrm{Mod}(\overline{\mathbf{\Lambda}})$$
and $\mathbb{T}^{\ast}:\mathrm{Mod}(\overline{\mathbf{\Lambda}})\longrightarrow \mathrm{Mod}(\mathbf{\Lambda}^{op}).$\\
Since $P=\textswab{F}\Big(\mathrm{Hom}_{\mathcal{T}}(T,-)\stackrel{g}{\longrightarrow} \mathbb{G}(M_{T}\amalg \mathrm{Hom}_{\mathcal{U}}(U,-))\Big)=\mathrm{Hom}_{\mathbf{\Lambda}}\Big(\left[
\begin{smallmatrix}
T&0 \\
M&U
\end{smallmatrix} \right],-\Big):\mathbf{\Lambda} \rightarrow \mathbf{Ab}$, then
$P^{\ast}:=\mathrm{Hom}_{\mathbf{\Lambda}}\Big(-,\left[
\begin{smallmatrix}
T&0 \\
M&U
\end{smallmatrix} \right]\Big) \in \mathrm{Mod}(\mathbf{\Lambda}^{op})$ (see \ref{estrellaP}).
We also have that 
$$\overline{\textswab{F}}\Big(\mathrm{Hom}_{\mathcal{U}}(-,U)\longrightarrow \overline{\mathbb{G}}(M_{U}\amalg \mathrm{Hom}_{\mathcal{T}}(-,T))\Big):=\mathrm{Hom}_{\overline{\mathbf{\Lambda}}}\Big(\left[
\begin{smallmatrix}
U&0 \\
\overline{M}&T 
\end{smallmatrix} \right],-\Big).$$
It is straightforward to see that
 $\mathbb{T}^{\ast}\Big(\overline{\textswab{F}}\Big(\mathrm{Hom}_{\mathcal{U}}(-,U)\longrightarrow \overline{\mathbb{G}}(M_{U}\amalg \mathrm{Hom}_{\mathcal{T}}(-,T))\Big)\Big)\simeq \mathrm{Hom}_{\mathbf{\Lambda}}\Big(-,\left[
\begin{smallmatrix}
T&0 \\
M&U
\end{smallmatrix} \right]\Big).$ We conclude that
$\mathrm{Hom}_{\overline{\mathbf{\Lambda}}}\Big(\left[
\begin{smallmatrix}
U&0 \\
\overline{M}&T 
\end{smallmatrix} \right],-\Big)\circ \mathbb{T}\simeq 
\mathrm{Hom}_{\mathbf{\Lambda}}\Big(,-\left[
\begin{smallmatrix}
T &0 \\
M & U
\end{smallmatrix} \right]\Big).$
\end{proof}

\begin{proposition}\label{descripcionmor}
Consider the morphism between projectives in the comma category $\Big( \mathrm{Mod}(\mathcal{T}), \mathbb{G}\mathrm{Mod}(\mathcal{U})\Big)$ given by the diagram
$$\xymatrix{\mathrm{Hom}_{\mathcal{T}}(T,-)\ar[rrr]^{\alpha}\ar[d]^{f}  & & & \mathrm{Hom}_{\mathcal{T}}(T',-)\ar[d]_{f'}\\
\mathbb{G}(M_{T}\amalg \mathrm{Hom}_{\mathcal{U}}(U,-))\ar[rrr]^{\mathbb{G}(\beta)}
 & & &   \mathbb{G}(M_{T'}\amalg \mathrm{Hom}_{\mathcal{U}}(U',-)).}$$
 Then $\beta={\left[
\begin{smallmatrix}
\mathbb{F}(\alpha) & a_{12} \\
0 & \mathrm{Hom}_{\mathcal{U}}(u,-)
\end{smallmatrix} \right]}$ and via the functor $(-)^{\ast}: \Big( \mathrm{Mod}(\mathcal{T}), \mathbb{G}\mathrm{Mod}(\mathcal{U})\Big)\longrightarrow 
 \Big( \mathrm{Mod}(\mathcal{U}^{op}), \overline{\mathbb{G}}\mathrm{Mod}(\mathcal{T}^{op})\Big)$ the previous morphism corresponds to 
 $$\xymatrix{\mathrm{Hom}_{\mathcal{U}}(-,U')\ar[rrr]^{\overline{\alpha}=\mathrm{Hom}_{\mathcal{U}}(-,u)}
\ar[d]^{\overline{f'}}  & & & \mathrm{Hom}_{\mathcal{U}}(-,U)\ar[d]_{\overline{f}}  \\
\overline{\mathbb{G}}(M_{U'}\amalg \mathrm{Hom}_{\mathcal{T}}(-,T'))\ar[rrr]^{\overline{\mathbb{G}}\Big({\left[
\begin{smallmatrix}
\overline{u} & & b_{12}\\
 0 & & \mathrm{Hom}_{\mathcal{T}}(-,t)
\end{smallmatrix} \right]}\Big)} 
 &  & &  \overline{\mathbb{G}}(M_{U}\amalg \mathrm{Hom}_{\mathcal{T}}(-,T))}$$
 in the category $ \Big( \mathrm{Mod}(\mathcal{U}^{op}), \overline{\mathbb{G}}\mathrm{Mod}(\mathcal{T}^{op})\Big)$
 with
 $b_{12}:=\Psi^{-1}_{}(\Theta_{}(a_{12}))$ where the morphisms
 $\Theta_{}:\mathrm{Hom}_{\mathrm{Mod}(\mathcal{U})}\Big(\mathrm{Hom}_{\mathcal{U}}(U,-),M_{T'}\Big)\longrightarrow M(U,T')$ and \\
 $\Psi_{}:\mathrm{Hom}_{\mathrm{Mod}(\mathcal{T}^{op})}\Big(\mathrm{Hom}_{\mathcal{T}}(-,T'),M_{U}\Big)\longrightarrow 
M(U,T')$ are the Yoneda Isomorphisms.
\end{proposition}
\begin{proof}
Consider $\beta:M_{T}\amalg \mathrm{Hom}_{\mathcal{U}}(U,-)\longrightarrow M_{T'}\amalg \mathrm{Hom}_{\mathcal{U}}(U',-)$ and the following commutative diagram
$$\xymatrix{\mathrm{Hom}_{\mathcal{T}}(T,-)\ar[rrr]^{{\alpha}}\ar[d]^{f}  & & & \mathrm{Hom}_{\mathcal{T}}(T',-)\ar[d]_{f'}\\
\mathbb{G}(M_{T}\amalg \mathrm{Hom}_{\mathcal{U}}(U,-))\ar[rrr]^{\mathbb{G}(\beta)}
 & & &   \mathbb{G}(M_{T'}\amalg \mathrm{Hom}_{\mathcal{U}}(U',-))}$$
By adjunction and \ref{adjunotro}, we have the following commutative diagram
$$\xymatrix{M_{T}=\mathbb{F}(\mathrm{Hom}_{\mathcal{T}}(T,-))\ar[d]^{{\left[
\begin{smallmatrix}
1 \\
0
\end{smallmatrix}\right]}}\ar[r]^{\mathbb{F}(\alpha)} & M_{T'}=\mathbb{F}(\mathrm{Hom}_{\mathcal{T}}(T',-))\ar[d]^{{\left[
\begin{smallmatrix}
1 \\
0
\end{smallmatrix}\right]}} \\
M_{T}\amalg \mathrm{Hom}_{\mathcal{U}}(U,-)\ar[r]^{\beta}
 &  M_{T'}\amalg \mathrm{Hom}_{\mathcal{U}}(U',-).}$$
Since $\beta=\left[
\begin{smallmatrix}
a_{11}& a_{12} \\
a_{21} & a_{22}
\end{smallmatrix} \right]$ where $a_{11}:M_{T}\longrightarrow M_{T'}$, $a_{12}:\mathrm{Hom}_{\mathcal{U}}(U,-)\longrightarrow M_{T'}$, $a_{21}:M_{T}\longrightarrow \mathrm{Hom}_{\mathcal{U}}(U',-)$ and $a_{22}:\mathrm{Hom}_{\mathcal{U}}(U,-)\longrightarrow
\mathrm{Hom}_{\mathcal{U}}(U',-)$ we have that
$$\left[
\begin{smallmatrix}
a_{11}& a_{12} \\
a_{21} & a_{22}
\end{smallmatrix}\right] \left[
\begin{smallmatrix}
1 \\
0
\end{smallmatrix}\right] =\left[
\begin{smallmatrix}
\mathbb{F}(\alpha) \\
0
\end{smallmatrix}\right].$$
Therefore, $a_{11}=\mathbb{F}(\alpha)$ and $a_{21}=0$. By Yoneda, $a_{12}$ is determined by an element $m\in M(U,T')$ and $a_{22}:=\mathrm{Hom}_{\mathcal{U}}(u,-):\mathrm{Hom}_{\mathcal{U}}(U,-)\longrightarrow
\mathrm{Hom}_{\mathcal{U}}(U',-)$ and $\alpha=\mathrm{Hom}_{\mathcal{T}}(t,-):\mathrm{Hom}_{\mathcal{T}}(T,-)\longrightarrow
\mathrm{Hom}_{\mathcal{T}}(T',-)$ for some $u:U'\longrightarrow U$ and $t:T'\longrightarrow T$.\\
We define
$$\overline{\beta}=\left[
\begin{smallmatrix}
\overline{u} & b_{12} \\
0 &  \mathrm{Hom}_{\mathcal{T}}(-,t)
\end{smallmatrix}\right]:M_{U'}\amalg \mathrm{Hom}_{\mathcal{T}}(-,T')\longrightarrow M_{U}\amalg \mathrm{Hom}_{\mathcal{T}}(-,T)$$ where $\overline{u}:M_{U'}\longrightarrow M_{U}$, $\mathrm{Hom}_{\mathcal{T}}(-,t):\mathrm{Hom}_{\mathcal{T}}(-,T')\longrightarrow
\mathrm{Hom}_{\mathcal{T}}(-,T)$ and  $b_{12}:\mathrm{Hom}_{\mathcal{T}}(-,T')\longrightarrow M_{U}$ is defined as $b_{12}:=\Psi^{-1}_{}(\Theta_{}(a_{12}))$ where the morphisms
 $\Theta_{}:\mathrm{Hom}_{\mathrm{Mod}(\mathcal{U})}\Big(\mathrm{Hom}_{\mathcal{U}}(U,-),M_{T'}\Big)\longrightarrow M(U,T')$ and\\ 
 $\Psi_{}:\mathrm{Hom}_{\mathrm{Mod}(\mathcal{T}^{op})}\Big(\mathrm{Hom}_{\mathcal{T}}(-,T'),M_{U}\Big)\longrightarrow 
M(U,T')$ are the Yoneda Isomorphisms.\\
We assert that the following diagram commutes
$$\xymatrix{M_{U'}=\overline{\mathbb{F}}(\mathrm{Hom}_{\mathcal{U}}(-,U'))\ar[d]^{{\left[
\begin{smallmatrix}
1 \\
0
\end{smallmatrix}\right]}}\ar[rr]^{\overline{\mathbb{F}}(\mathrm{Hom}_{\mathcal{U}}(-,u))=\overline{u}} & & M_{U}=\overline{\mathbb{F}}(\mathrm{Hom}_{\mathcal{U}}(-,U))\ar[d]_{{\left[
\begin{smallmatrix}
1 \\
0
\end{smallmatrix}\right]}}\\
 M_{U'}\amalg \mathrm{Hom}_{\mathcal{T}}(-,T')\ar[rr]^{\overline{\beta}} 
 & &  M_{U}\amalg \mathrm{Hom}_{\mathcal{T}}(-,T).}$$
Indeed,  $\left[
\begin{smallmatrix}
\overline{u}& b_{12} \\
0 & \mathrm{Hom}_{\mathcal{T}}(-,t)
\end{smallmatrix}\right]\left[
\begin{smallmatrix}
1 \\
0
\end{smallmatrix}\right]=\left[
\begin{smallmatrix}
\overline{u} \\
0
\end{smallmatrix}\right]=\left[
\begin{smallmatrix}
1 \\
0
\end{smallmatrix}\right]\circ \overline{u}$ (see \cite[Lemma 5.8(i)]{LeOS}) .
Then by adjunction we have the following commutative diagram 
 $$\xymatrix{\mathrm{Hom}_{\mathcal{U}}(-,U')\ar[rrr]^{\overline{\alpha}=\mathrm{Hom}_{\mathcal{U}}(-,u)}
\ar[d]^{\overline{f'}}  & & & \mathrm{Hom}_{\mathcal{U}}(-,U)\ar[d]_{\overline{f}}  \\
\overline{\mathbb{G}}(M_{U'}\amalg \mathrm{Hom}_{\mathcal{T}}(-,T'))\ar[rrr]^{\overline{\mathbb{G}}\Big({\left[
\begin{smallmatrix}
\overline{u} & & b_{12}\\
 0 & & \mathrm{Hom}_{\mathcal{T}}(-,t)
\end{smallmatrix} \right]}\Big)} 
 &  & &  \overline{\mathbb{G}}(M_{U}\amalg \mathrm{Hom}_{\mathcal{T}}(-,T))}$$
Via the functor $\overline{\textswab{F}}:\Big( \mathrm{Mod}(\mathcal{U}^{op}), \overline{\mathbb{G}}\mathrm{Mod}(\mathcal{T}^{op})\Big) \longrightarrow 
\mathrm{Mod}(\overline{\mathbf{\Lambda}})$ we have the morphism
in $\mathrm{Mod}(\overline{\mathbf{\Lambda}})$:
$$\overline{\alpha}\amalg \overline{\beta}:
\mathrm{Hom}_{\mathcal{U}}(-,U')\amalg_{\overline{f'}} (M_{U'}\amalg \mathrm{Hom}_{\mathcal{T}}(-,T'))\longrightarrow
\mathrm{Hom}_{\mathcal{U}}(-,U)\amalg_{\overline{f}}  (M_{U}\amalg \mathrm{Hom}_{\mathcal{T}}(-,T))$$
where for $\left[\begin{smallmatrix}
U_{1} & 0 \\
\overline{M} & T_{1} \\
\end{smallmatrix} \right]\in \overline{\mathbf{\Lambda}}$, we have $\Big[\overline{\alpha}\amalg \overline{\beta}\Big]_{\left[\begin{smallmatrix}
U_{1} & 0 \\
\overline{M} & T_{1} \\
\end{smallmatrix} \right]}=\overline{\alpha}_{U_{1}}\amalg \overline{\beta}_{T_{1}}$ with $\overline{\beta}_{T_{1}}=\left[
\begin{smallmatrix}
[\overline{u}]_{T_{1}} & [b_{12}]_{T_{1}} \\
0 &  [\mathrm{Hom}_{\mathcal{T}}(-,t)]_{T_{1}}
\end{smallmatrix}\right]$. Now, is straightforward to show that there is an isomorphism
$$\overline{\alpha}\amalg \overline{\beta}\simeq
\mathrm{Hom}_{\overline{\mathbf{\Lambda}}}\Big(
\left[
\begin{smallmatrix}
u^{op} &0 \\
 m & t^{op}
\end{smallmatrix} \right],-\Big): \mathrm{Hom}_{\overline{\mathbf{\Lambda}}}\Big(\left[
\begin{smallmatrix}
U' &0 \\
\overline{M}&T'
\end{smallmatrix} \right],-\Big)\longrightarrow 
\mathrm{Hom}_{\overline{\mathbf{\Lambda}}}\Big(\left[
\begin{smallmatrix}
U &0 \\
\overline{M}&T
\end{smallmatrix} \right],-\Big),
$$
where $\left[
\begin{smallmatrix}
u^{op} &0 \\
 m & t^{op}
\end{smallmatrix} \right]:\left[
\begin{smallmatrix}
U &0 \\
\overline{M}&T
\end{smallmatrix} \right]\longrightarrow \left[
\begin{smallmatrix}
U' &0 \\
\overline{M}&T'
\end{smallmatrix} \right]$ is a morphism in $\overline{\mathbf{\Lambda}}$.\\

On the other side, considering $\textswab{F}:\Big( \mathrm{Mod}(\mathcal{T}), \mathbb{G}\mathrm{Mod}(\mathcal{U})\Big) \longrightarrow 
\mathrm{Mod}(\mathbf{\Lambda})$, we have $\textswab{F}(\alpha,\beta)=\alpha\amalg \beta$ where
$$\alpha\amalg \beta:
\mathrm{Hom}_{\mathcal{T}}(T,-)\amalg_{f} (M_{T}\amalg \mathrm{Hom}_{\mathcal{U}}(U,-))\longrightarrow
\mathrm{Hom}_{\mathcal{T}}(T',-)\amalg_{f'}  (M_{T'}\amalg \mathrm{Hom}_{\mathcal{U}}(U',-)).$$
Therefore
$$\alpha\amalg \beta=
\mathrm{Hom}_{\mathbf{\Lambda}}\Big(\left[
\begin{smallmatrix}
t &0 \\
m & u
\end{smallmatrix} \right],-\Big):\mathrm{Hom}_{\mathbf{\Lambda}}\Big(\left[
\begin{smallmatrix}
T &0 \\
M & U
\end{smallmatrix} \right],-\Big)\longrightarrow 
\mathrm{Hom}_{\mathbf{\Lambda}}\Big(\left[
\begin{smallmatrix}
T' &0 \\
M & U'
\end{smallmatrix} \right],-\Big)$$
where $\left[
\begin{smallmatrix}
t &0 \\
m & u
\end{smallmatrix} \right]:\left[
\begin{smallmatrix}
T' &0 \\
M & U'
\end{smallmatrix} \right]\longrightarrow \left[
\begin{smallmatrix}
T &0 \\
M & U
\end{smallmatrix} \right]$ with $m:=\Theta(a_{12})\in M(U,T')$.\\
Hence, $(\alpha\amalg \beta)^{\ast}=\mathrm{Hom}_{\mathbf{\Lambda}}\Big(-,\left[
\begin{smallmatrix}
t &0 \\
m & u
\end{smallmatrix} \right]\Big):
\mathrm{Hom}_{\mathbf{\Lambda}}\Big(-,\left[
\begin{smallmatrix}
T' &0 \\
M & U'
\end{smallmatrix} \right]\Big)\longrightarrow 
\mathrm{Hom}_{\mathbf{\Lambda}}\Big(-,\left[
\begin{smallmatrix}
T &0 \\
M & U
\end{smallmatrix} \right]\Big).$
It is easy to show that $(\alpha\amalg \beta)^{\ast}\simeq (\overline{\alpha}\amalg \overline{\beta})\circ \mathbb{T}$, and this proves the proposition.
\end{proof}

Given an abelian category $\mathcal{A}$ let us denote by $\mathrm{proj}(\mathcal{A})$ the full subcategory of finitely geneerated objects.

\begin{proposition}\label{funtorestrellita}
Let us denote by $\mathrm{proj}\Big(\Big( \mathrm{Mod}(\mathcal{T}), \mathbb{G}\mathrm{Mod}(\mathcal{U})\Big)\Big)$
the category of finitely generated projective objects in $\Big( \mathrm{Mod}(\mathcal{T}), \mathbb{G}\mathrm{Mod}(\mathcal{U})\Big)$. 
\begin{enumerate}
\item [(a)]
Then we have a duality
$$(-)^{\ast}:\mathrm{proj}\Big(\Big( \mathrm{Mod}(\mathcal{T}), \mathbb{G}\mathrm{Mod}(\mathcal{U})\Big)\Big)\longrightarrow
\mathrm{proj}\Big(\Big( \mathrm{Mod}(\mathcal{U}^{op}), \overline{\mathbb{G}}\mathrm{Mod}(\mathcal{T}^{op})\Big)\Big)$$

\item [(b)] Suppose that $\mathcal{U}$ and $\mathcal{T}$ are dualizing varieties and $M_{T}\in \mathrm{mod}(\mathcal{U})$ and $M_{U}\in \mathrm{mod}(\mathcal{T}^{op})$ for all $U\in \mathcal{U}$ and $T\in \mathcal{T}^{op}$. Then  we have a duality
$$(-)^{\ast}:\mathrm{proj}\Big(\Big( \mathrm{mod}(\mathcal{T}), \mathbb{G}\mathrm{mod}(\mathcal{U})\Big)\Big)\longrightarrow
\mathrm{proj}\Big(\Big( \mathrm{mod}(\mathcal{U}^{op}), \overline{\mathbb{G}}\mathrm{mod}(\mathcal{T}^{op})\Big)\Big)$$
\end{enumerate}

\end{proposition}

\begin{proof}.\\
$(a)$ It is known that  the funtor $(-)^{\ast}$ restricts to a duality $(-)^{\ast}:\mathrm{proj}(\mathrm{Mod}(\mathbf{\Lambda}))\longrightarrow \mathrm{proj}(\mathrm{Mod}(\mathbf{\Lambda}^{op}))$ (see \cite{Auslander2} in page 337). Since 
$\Big( \mathrm{Mod}(\mathcal{T}), \mathbb{G}\mathsf{Mod}(\mathcal{U})\Big)\simeq \mathrm{Mod}(\mathbf{\Lambda})$ and 
$\Big( \mathrm{Mod}(\mathcal{U}^{op}), \overline{\mathbb{G}}\mathrm{Mod}(\mathcal{T}^{op})\Big)\simeq \mathrm{Mod}(\mathbf{\Lambda}^{op})$ we have the result.\\
$(b)$.  Since $\mathrm{proj}(\mathrm{Mod}(\mathbf{\Lambda}))\subseteq \mathrm{mod}(\mathbf{\Lambda})$, we have that
$\mathrm{proj}(\mathrm{mod}(\mathbf{\Lambda}))=\mathrm{proj}(\mathrm{Mod}(\mathbf{\Lambda}))$. Since $\Big( \mathrm{mod}(\mathcal{T}), \mathbb{G}\mathrm{mod}(\mathcal{U})\Big)\simeq \mathrm{mod}(\mathbf{\Lambda})$ and 
$\Big( \mathrm{mod}(\mathcal{U}^{op}), \overline{\mathbb{G}}\mathrm{mod}(\mathcal{T}^{op})\Big)\simeq \mathrm{mod}(\mathbf{\Lambda}^{op})$. The result follows from (a).

\end{proof}

In the following we will write $\mathcal{C}(C,-)$ and $\mathcal{C}(-,C)$ instead of $\mathrm{Hom}_{\mathcal{C}}(C,-)$ and $\mathrm{Hom}_{\mathcal{C}}(-,C)$.

\begin{proposition}\label{dualprojmaps}
Let $M:=\widehat{\mathbbm{Hom}}\in \mathrm{Mod}(\mathcal{C}\otimes\mathcal{C}^{op})$ and consider the induced functor $\mathbb{G}:\mathrm{Mod}(\mathcal C)\rightarrow \mathrm{Mod}(\mathcal C)$ and the duality
$$(-)^{\ast}:\mathrm{proj}\Big(\Big( \mathrm{Mod}(\mathcal{C}), \mathbb{G}\mathrm{Mod}(\mathcal{C})\Big)\Big)\longrightarrow
\mathrm{proj}\Big(\Big( \mathrm{Mod}(\mathcal{C}^{op}), \overline{\mathbb{G}}\mathrm{Mod}(\mathcal{C}^{op})\Big)\Big)$$
given in \ref{funtorestrellita} and the isomorphisms $J_{1}$ and $J_{2}$ given in \ref{descrimaps}. Then we have an induced duality, 
$(-)^{\ast}:\xymatrix{\mathrm{proj}\Big(\mathrm{mod}(\mathcal{C}),\mathrm{mod}(\mathcal{C})\Big)\ar[rr] && \mathrm{proj}\Big(\mathrm{mod}(\mathcal{C}^{op}),\mathrm{mod}(\mathcal{C}^{op})\Big)}$ defined as $J_{2}^{-1}\circ (-)^{\ast}\circ J_{1}$, which will be denoted as $(-)^{\ast}$.
Moreover, maps of the form $\xymatrix{\mathcal{C}(C_{1},-)\ar[r]^(.4){\left[
{{\begin{smallmatrix}
1 \\
0
\end{smallmatrix}}} \right]} & \mathcal{C}(C_{1},-)\amalg \mathcal{C}(C_{2},-)}$ are projective in $\Big(\mathrm{mod}(\mathcal{C}),\mathrm{mod}(\mathcal{C})\Big)$ and 
$$\Big(\xymatrix{\mathcal{C}(C_{1},-)\ar[r]^(.4){\left[
{{\begin{smallmatrix}
1 \\
0
\end{smallmatrix}}} \right]} & \mathcal{C}(C_{1},-)\amalg \mathcal{C}(C_{2},-)}\Big)^{\ast}=\xymatrix{\mathcal{C}(-,C_{2})\ar[r]^(.35){\left[
{{\begin{smallmatrix}
1 \\
0
\end{smallmatrix}}} \right]} & \mathcal{C}(-,C_{2})\amalg \mathcal{C}(-,C_{1}).}$$
\end{proposition}
\begin{proof}
Let us consider $\xymatrix{\mathcal{C}(C_{1},-)\ar[r]^(.35){f=\left[
{{\begin{smallmatrix}
1 \\
0
\end{smallmatrix}}} \right]} & \mathcal{C}(C_{1},-)\amalg \mathcal{C}(C_{2},-)}$ and the Yoneda isomorphism
$$Y_{C'}:\mathcal{C}(C_{1},C')\amalg \mathcal{C}(C_{2},C')\longrightarrow \mathbb{G}\Big(\mathcal{C}(C_{1},-)\amalg \mathcal{C}(C_{2},-)\Big)(C').$$
We assert that the morphism
$$J_{1}(f):=\widehat{f}:\mathcal{C}(C_{1},-)\longrightarrow 
\mathbb{G}\Big(\mathcal{C}(C_{1},-)\amalg \mathcal{C}(C_{2},-)\Big)$$ is a projective in the category $\Big(\mathrm{Mod}(\mathcal{C}),\mathbb{G}(\mathrm{Mod}(\mathcal{C}))\Big)$.
Indeed, using that $M:=\widehat{\mathbbm{Hom}}\in \mathrm{Mod}(\mathcal{C}\otimes\mathcal{C}^{op})$ and the descripcion given in \ref{projectives}, we have that the projective $g$ given in \ref{projectives}, coincides with
$\widehat{f}$.
That is, for each $C'\in \mathcal{C}$, we have the following commutative diagram
$$\xymatrix{\mathcal{C}(C_{1},C')\ar[rr]^(.4){f_{C'}=\left[
{{\begin{smallmatrix}
1_{\mathcal{C}(C_{1},C')} \\
0
\end{smallmatrix}}} \right]}\ar@{=}[d]  & & \mathcal{C}(C_{1},C')\amalg \mathcal{C}(C_{2},C')\ar[d]^{Y_{C'}}\\
\mathcal{C}(C_{1},C')\ar[rr]^(.4){g_{C'}=\widehat{f}_{C'}} & &
\mathbb{G}\Big(\mathcal{C}(C_{1},C')\amalg \mathcal{C}(C_{2},C')\Big).}$$

Since $\mathbb{Y}:=\{Y_{C'}\}_{C'\in \mathcal{C}}: \mathcal{C}(C_{1},-)\amalg \mathcal{C}(C_{2},-)\longrightarrow \mathbb{G}\Big(\mathcal{C}(C_{1},-)\amalg \mathcal{C}(C_{2},-)\Big)$ defines a morphism of $\mathcal{C}$-modules, we have the commutative diagram
$$(\ast):\xymatrix{\mathcal{C}(C_{1},-)\ar[rr]^{\left[
{{\begin{smallmatrix}
1 \\
0
\end{smallmatrix}}} \right]}\ar@{=}[d]  & & \mathcal{C}(C_{1},-)\amalg \mathcal{C}(C_{2},-)\ar[d]^{\mathbb{Y}}\\
\mathcal{C}(C_{1},-)\ar[rr]^{g=\widehat{f}} & &
\mathbb{G}\Big(\mathcal{C}(C_{1},-)\amalg \mathcal{C}(C_{2},-)\Big).}$$
Similarly, we have the commutative diagram
$$\xymatrix{\mathcal{C}(-,C_{2})\ar[rr]^{\left[
{{\begin{smallmatrix}
1 \\
0
\end{smallmatrix}}} \right]}\ar@{=}[d]  & & \mathcal{C}(-,C_{2})\amalg \mathcal{C}(-,C_{1})\ar[d]^{\overline{\mathbb{Y}}}\\
\mathcal{C}(-,C_{2})\ar[rr]^{\overline{g}} & &
\mathbb{G}\Big(\mathcal{C}(-,C_{2})\amalg \mathcal{C}(-,C_{1})\Big).}$$
Therefore we have that $J_{2}^{-1}(\overline{g})=\xymatrix{\mathcal{C}(-,C_{2})\ar[r]^(.4){\left[
{{\begin{smallmatrix}
1 \\
0
\end{smallmatrix}}} \right]} &  \mathcal{C}(-,C_{2})\amalg \mathcal{C}(-,C_{1})}$. Therefore, by \ref{dualproyec} we have the following equalities

\begin{align*}
\Big(\xymatrix{\mathcal{C}(C_{1},-)\ar[r]^(.4){\left[
{{\begin{smallmatrix}
1 \\
0
\end{smallmatrix}}} \right]} &  \mathcal{C}(C_{1},-)\amalg \mathcal{C}(C_{2},-)}\Big)^{\ast} & = \Big(\xymatrix{\mathcal{C}(C_{1},-)\ar[r]^(.35){g} & \mathbb{G}(M_{C_{1}}\amalg \mathcal{C}(C_{2},-))}\Big)^{\ast}\\
& =\xymatrix{\mathcal{C}(-,C_{2})\ar[r]^(.35){\overline{g}} & \overline{\mathbb{G}}(M_{C_{2}}\amalg \mathcal{C}(-,C_{1}))}\\
& =\xymatrix{\mathcal{C}(-,C_{2})\ar[r]^(.4){\left[
{{\begin{smallmatrix}
1 \\
0
\end{smallmatrix}}} \right]} &  \mathcal{C}(-,C_{2})\amalg \mathcal{C}(-,C_{1}).}
\end{align*}
\end{proof}

\begin{proposition}\label{morphismdual}
Let  $M:=\widehat{\mathbbm{Hom}}\in \mathrm{Mod}(\mathcal{C}\otimes\mathcal{C}^{op})$ be  and consider the induced functor $\mathbb{G}:\mathrm{Mod}(\mathcal C)\rightarrow \mathrm{Mod}(\mathcal C)$ and a map between projectives in the category $\Big(\mathrm{Mod}(\mathcal{C}),\mathrm{Mod}(\mathcal{C})\Big)$
$$\xymatrix{\mathcal{C}(C_{1},-)\ar[r]^{a_{11}}\ar[d]^{\left[
{{\begin{smallmatrix}
1 \\
0
\end{smallmatrix}}} \right]}  & \mathcal{C}(C_{1}',-)\ar[d]^{\left[
{{\begin{smallmatrix}
1 \\
0
\end{smallmatrix}}} \right]} \\
\mathcal{C}(C_{1},-)\amalg \mathcal{C}(C_{2},-)\ar[r]^{\beta} & \mathcal{C}(C_{1}',-)\amalg \mathcal{C}(C_{2}',-)}$$
where $\beta=\left[
{{\begin{smallmatrix}
 a_{11} &  a_{12}\\
 0  & a_{22}
\end{smallmatrix}}} \right].$ Then, 
applying $(-)^{\ast}$ we get the following map in the category $\Big(\mathrm{Mod}(\mathcal{C}^{op}),\mathrm{Mod}(\mathcal{C}^{op})\Big)$
$$\xymatrix{\mathcal{C}(-,C_{2}')\ar[r]^{a_{22}^{\ast}}\ar[d]^{\left[
{{\begin{smallmatrix}
1 \\
0
\end{smallmatrix}}} \right]}  & \mathcal{C}(-,C_{2})\ar[d]^{\left[
{{\begin{smallmatrix}
1 \\
0
\end{smallmatrix}}} \right]} \\
\mathcal{C}(-,C_{2}')\amalg \mathcal{C}(-,C_{1}')\ar[r]^{\overline{\beta}} & \mathcal{C}(-,C_{2})\amalg \mathcal{C}(-,C_{1})}$$
where $\overline{\beta}=\left[
{{\begin{smallmatrix}
 a_{22}^{\ast} & a_{12}^{\ast}\\
 0  & a_{11}^{\ast}
\end{smallmatrix}}} \right]$.
\end{proposition}
\begin{proof}
By \ref{descripcionmor}, we obtain the equality
$b_{12}:=\Psi^{-1}_{}(\Theta_{}(a_{12}))$ where the morphisms
 $\Theta_{}:\mathrm{Hom}_{\mathrm{Mod}(\mathcal{C})}\Big(\mathrm{Hom}_{\mathcal{C}}(C_{2},-),M_{C_{1}'}\Big)\longrightarrow M_{C_{1}'}(C_{2})=\mathcal{C}(C_{1}',C_{2})$ and\\ 
 $\Psi_{}:\mathrm{Hom}_{\mathrm{Mod}(\mathcal{C}^{op})}\Big(\mathrm{Hom}_{\mathcal{C}}(-,C_{1}'),M_{C_{2}}\Big)=\mathrm{Hom}_{\mathrm{Mod}(\mathcal{C}^{op})}\Big(\mathrm{Hom}_{\mathcal{C}}(-,C_{1}'),\mathcal{C}(-,C_{2})\Big)\longrightarrow \mathcal{C}(C_{1}',C_{2})$ are the Yoneda Isomorphisms. Then we conclude that $b_{12}=a_{12}^{\ast}$, the rest of the proof follows from \ref{dualprojmaps} and \ref{descrimaps}.
\end{proof}

\begin{proposition}\label{projectivecover}
Let $\mathcal{C}$ be an abelian  category with projective covers and let $f:A\longrightarrow B$ be a morphism in $\mathcal{C}$.
\begin{enumerate}
\item [(i)] If  $\mathrm{Coker}(f)\neq 0$. Construct the following diagram
$$\xymatrix{P_{0}\ar[r]^{\left[
{{\begin{smallmatrix}
1 \\
0
\end{smallmatrix}}} \right]}\ar[d]^{\alpha} & P_{0}\oplus Q_{0}\ar[r]^{[0,1]}\ar[d]^{\gamma} & Q_{0}\ar[r]\ar[d]^{\beta}\ar@{-->}[dl]_{\beta'} & 0\\
A\ar[r]_{f}\ar[d] & B\ar[r]_{\pi}\ar[d] & C\ar[r]\ar[d]& 0\\
0 & 0 & 0}$$
where $\alpha:P_{0}\longrightarrow A$ and $\beta:Q_{0}\longrightarrow C$ are projective covers $\beta':Q_{0}\rightarrow B$ is the induced morphism by the projectivity of $Q_{0}$ and $\gamma=(f\alpha,\beta')$. Then the morphism $(\alpha,\gamma):\Big(P_{0},\left[
{{\begin{smallmatrix}
1 \\
0
\end{smallmatrix}}} \right], P_{0}\oplus Q_{0}\Big)\longrightarrow
\Big(A,f,B \Big)$ given by the following diagram
$$\xymatrix{\underset{P_{0}\oplus Q_{0}}{\overset{P_{0}}{\Big\downarrow}}\ar[r]^{(\alpha,\gamma)} & \underset{B}{\overset{A}{\Big\downarrow}}\ar[r]^{} & 0}$$ is a projective cover of the object $f:A\longrightarrow B$ in the category $\mathrm{maps}(\mathcal{C})$.

\item [(ii)] If $\mathrm{Coker}(f)=0$. Consider the following diagram
$$\xymatrix{P_{0}\ar[r]^{1}\ar[d]_{\alpha} & P_{0}\ar[d]^{f\alpha}\\
A\ar[r]_{f}\ar[d] & B\ar[d]\\
0 & 0}$$
where $\alpha:P_{0}\longrightarrow A$ is a projective cover. Then, the morphism $(\alpha,f\alpha):(P_{0},1,P_{0})\longrightarrow (A,f,B)$ is a projective cover of $(A,f,B)$.
\end{enumerate}
\end{proposition}
\begin{proof}
Let us see that $(\alpha,\gamma)$ is minimal. Indeed, let 
$$(\ast):\quad (\theta_{1},\theta_{2}):\Big(P_{0},\left[
{{\begin{smallmatrix}
1 \\
0
\end{smallmatrix}}} \right], P_{0}\oplus Q_{0}\Big)\longrightarrow \Big(P_{0},\left[
{{\begin{smallmatrix}
1 \\
0
\end{smallmatrix}}} \right], P_{0}\oplus Q_{0}\Big)$$ such that
$(\alpha,\gamma)(\theta_{1},\theta_{2})=(\alpha,\gamma)$. Then
$\alpha\theta_{1}=\alpha$ and $\gamma\theta_{2}=\gamma$. Since $\alpha$ is projective cover, we have that it is minimal and hence
$\theta_{1}$ is an isomorphism. Now, since $(\ast)$ is a morphism in the category of maps, we have the following commutative diagram
$$\xymatrix{P_{0}\ar[r]^{\left[
{{\begin{smallmatrix}
1 \\
0
\end{smallmatrix}}} \right]}\ar[d]^{\theta_{1}} & P_{0}\oplus Q_{0}\ar[d]^{\theta_{2}} \\
P_{0}\ar[r]^{\left[
{{\begin{smallmatrix}
1 \\
0
\end{smallmatrix}}} \right]}\ar[r] & P_{0}\oplus Q_{0}}$$
If $\theta_{2}=\left[
{{\begin{smallmatrix}
a_{11} & a_{12} \\
a_{21} & a_{22}
\end{smallmatrix}}} \right]$ we have that
$\left[
{{\begin{smallmatrix}
a_{11} & a_{12} \\
a_{21} & a_{22}
\end{smallmatrix}}} \right]\left[
{{\begin{smallmatrix}
1 \\
0
\end{smallmatrix}}} \right]=\left[
{{\begin{smallmatrix}
\theta_{1} \\
0
\end{smallmatrix}}} \right]$. Therefore, we conclude that $a_{11}=\theta_{1}$ and $a_{21}=0$. Now, $(0,\beta)\theta_{2}=
\pi \gamma\theta_{2}=\pi\gamma=(0,\beta)$ and then $(0,\beta)\left[
{{\begin{smallmatrix}
\theta_{1} & a_{12} \\
0 & a_{22}
\end{smallmatrix}}} \right]=(0,\beta a_{22})=(0,\beta).$ Hence, 
$\beta=\beta a_{22}$ and since $\beta$ is minimal, we have that $a_{22}$ is an isomorphism. Since $a_{12}:Q_{0}\longrightarrow P_{0}$, we have the morphism $a_{11}^{-1}a_{12}a_{22}^{-1}:Q_{0}\longrightarrow P_{0}$. Now it is easy to show that
$\theta_{2}^{-1}=\left[
{{\begin{smallmatrix}
a_{11}^{-1},  & & a_{11}^{-1}a_{12}a_{22}^{-1} \\
0 & & a_{22}^{-1}
\end{smallmatrix}}} \right]$. Hence $(\theta_{1},\theta_{2})$ is an isomorphism, proving that $(\alpha,\gamma)$ is minimal. Now, it is easy to show that
$\Big(P_{0},\left[
{{\begin{smallmatrix}
1 \\
0
\end{smallmatrix}}} \right], P_{0}\oplus Q_{0}\Big)$ is a projective object in $\mathrm{maps}(\mathcal C)$, then we conclude that $(\alpha,\gamma)$ is a projective cover.\\
(ii) Similar to (i).
\end{proof}



We define the transpose which is needed to get the Auslander-Reiten translate.
\begin{definition}\label{transposedef}(Transpose)
Let $f:A\longrightarrow \mathbb{G}(B)$ and object in $\Big( \mathrm{mod}(\mathcal{T}), \mathbb{G}\mathrm{mod}(\mathcal{U})\Big)$. Consider a minimal projective presentation
$$\xymatrix{\underset{\mathbb{G}(M_{T}\amalg \mathrm{Hom}_{\mathcal{U}}(U,-)}{\overset{\mathrm{Hom}_{\mathcal{T}}(T,-)}{\Big\downarrow}}\ar[r]^{(\alpha,\beta)} & \underset{\mathbb{G}(M_{T'}\amalg \mathrm{Hom}_{\mathcal{U}}(U',-)}{\overset{\mathrm{Hom}_{\mathcal{T}}(T',-)}{\Big\downarrow}}\ar[r]^{} & \underset{\mathbb{G}(B)}{\overset{A}{\Big\downarrow}}\ar[r] & 0}$$
By applying  the funtor $(-)^{\ast}$ given in \ref{funtorestrellita}, we define $\overline{\mathrm{TR}}(A,f,B):=\mathrm{Coker}((\alpha,\beta)^{\ast})$.
\end{definition}

Now, we define the Auslander-Reiten translate.
\begin{definition}(Auslander-Reiten translate)
For an object $f:A\longrightarrow \mathbb{G}(B)$ in $\Big( \mathrm{mod}(\mathcal{T}), \mathbb{G}\mathrm{mod}(\mathcal{U})\Big)$  we construct the exact sequence as in 	\ref{transposedef}
$$\xymatrix{Q^{\ast}\ar[r]^{(\alpha,\beta)^{\ast}} & P^{\ast}\ar[r]^{} & \mathrm{Coker}((\alpha,\beta)^{\ast})\ar[r] & 0.}$$
Considering the duality
$\widehat{\Theta'}: \Big( \mathrm{mod}(\mathcal{U}^{op}), \overline{\mathbb{G}}\mathrm{mod}(\mathcal{T}^{op})\Big)\longrightarrow \Big( \mathrm{mod}(\mathcal{T}), \mathbb{G}\mathrm{mod}(\mathcal{U})\Big) $
given in \cite[Proposition 6.10]{LeOS}, we define the Auslander-Reiten translate  
$\mathrm{Tau}:\Big( \mathrm{mod}(\mathcal{T}), \mathbb{G}\mathrm{mod}(\mathcal{U})\Big)\longrightarrow \Big( \mathrm{mod}(\mathcal{T}), \mathbb{G}\mathrm{mod}(\mathcal{U})\Big)$
as 
$$\mathrm{Tau}(A,f,B):=\widehat{\Theta'}\Big(\mathrm{Coker}((\alpha,\beta)^{\ast})\Big).$$
\end{definition}

Now, we are able to describe that Auslander-Reiten translate in the category $\mathrm{maps}(\mathrm{mod}(\mathcal C))$.\\
Let $\mathcal{C}$ be a dualizing $K$-variety. It is well known that under this conditions $\mathcal{C}$ is  Krull-Schmidt (see  page 318 in \cite{AuslanderRep1}) and therefore we have the every $\mathcal{C}$-module in $\mathrm{mod}(\mathcal{C})$ has a minimal projective presentation (see for example \cite[Lemma 2.1]{MOSS}). Then we  can define the transpose in $\mathrm{mod}(\mathcal{C})$ which we will denote by $\mathrm{Tr}$. That is, $\mathrm{Tr}:\mathrm{mod}(\mathcal{C})\rightarrow \mathrm{mod}(\mathcal{C}^{op})$.\\

\begin{theorem}\label{ARsequncedesc}
Let $\mathcal{C}$ be a dualizing and consider the equivalence $\mathrm{mod}(\mathbf{\Lambda})\xrightarrow{\sim} \mathrm{maps}(\mathrm{mod}(\mathcal C))$ given in \ref{mapsmodlam}(ii)  and the duality $\mathbb{D}_{\mathcal{C}^{op}}:\mathrm{mod}(\mathcal{C}^{op})\longrightarrow \mathrm{mod}(\mathcal{C})$. Let $f:C_{1}\longrightarrow C_{2}$ be a morphism in $\mathrm{mod}(\mathcal{C})$ such that there exists exact sequence $\xymatrix{C_{1}\ar[r]^{f} & C_{2}\ar[r] & C_{3}\ar[r] & 0}$ with $C_{3}\neq 0$ and $C_{3}$ not projective. Then 
 $$\mathrm{Tau}(C_{1},f,C_{2})=(\mathbb{D}_{\mathcal{C}^{op}}(Y),\mathbb{D}_{\mathcal{C}^{op}}(g),\mathbb{D}_{\mathcal{C}^{op}}\mathrm{Tr}(C_{3}))$$ for some morphism $g:\mathrm{Tr}(C_{3})\longrightarrow Y$ such that there exists an exact sequence
$$\xymatrix{0\ar[r] &\mathbb{D}_{\mathcal{C}^{op}}\mathrm{Tr}(C_{1})\ar[r] & \mathbb{D}_{\mathcal{C}^{op}}(Y)\ar[r]^{\mathbb{D}_{\mathcal{C}^{op}}(g)} & \mathbb{D}_{\mathcal{C}^{op}}\mathrm{Tr}(C_{3})}$$

\end{theorem}
\begin{proof}
Since $\mathcal{C}$ is an additive category with finite coproducts  and with splitting idempotenst, we have that every finitely presented projective $\mathcal{C}$-module  $P$ is of the form $\mathrm{Hom}_{\mathcal{C}}(C,-)$ for some object $C\in \mathcal{C}$ (see \cite{AuslanderRep1} ). Then in all what follows whenever we write a projective $\mathcal{C}$-module $P$, we mean a projective module of the form $\mathrm{Hom}_{\mathcal{C}}(C,-)$ for some object $C\in \mathcal{C}$.\\
Let $f:C_{1}\longrightarrow C_{2}$ a morphism in $\mathrm{mod}(\mathcal{C})$ and $C_{3}=\mathrm{Coker}(f)$, following \ref{projectivecover}, we construct a minimal projective presentation of $(C_{1},f,C_{2})$
$$\xymatrix{P_{1}\ar[d]^{\left[
{{\begin{smallmatrix}
1 \\
0
\end{smallmatrix}}} \right]}\ar[r]^{\lambda_{1}} & P_{0}\ar[r]^{\lambda_{0}}\ar[d]^{\left[
{{\begin{smallmatrix}
1 \\
0
\end{smallmatrix}}} \right]}& C_{1}\ar[r]\ar[d]^{f} & 0\\
P_{1}\oplus Q_{1}\ar[r]^{^{\left[
{{\begin{smallmatrix}
\lambda_{1} &  a\\
0 & b
\end{smallmatrix}}} \right]}} & P_{0}\oplus Q_{0}\ar[r]^{\gamma_{0}} & C_{2}\ar[r] & 0, }$$
where $$\xymatrix{P_{1}\ar[r]^{\lambda_{1}} & P_{0}\ar[r]^{\lambda_{0}}& C_{1}\ar[r] & 0},\quad  \xymatrix{Q_{1}\ar[r]^{b} & Q_{0}\ar[r]^{c}& C_{3}\ar[r] & 0}$$
are minimal projective presentation of $C_{1}$ and $C_{3}$ respectively.\\
Applying $(-)^{\ast}$ in the category $\mathrm{maps}(\mathrm{mod}(\mathcal C))$ we get
$$\xymatrix{(P_{0}\rightarrow P_{0}\oplus Q_{0})^{\ast}\ar[r] & (P_{1}\rightarrow P_{1}\oplus Q_{1})^{\ast}\ar[r] & \mathrm{TR}(C_{1},f,C_{2})\ar[r] & 0} $$
By \ref{morphismdual}, the last exact sequence is represented by the following diagram
$$\xymatrix{Q_{0}^{\ast}\ar[r]^{b^{\ast}}\ar[d]_{\left[
{{\begin{smallmatrix}
1 \\
0
\end{smallmatrix}}} \right]} & Q_{1}^{\ast}\ar[r]\ar[d]^{\left[
{{\begin{smallmatrix}
1 \\
0
\end{smallmatrix}}} \right]} & \mathrm{Tr}(C_{3})\ar[d]^{g}\ar[r] & 0\\
Q_{0}^{\ast}\oplus P_{0}^{\ast}\ar[r]^{\left[
{{\begin{smallmatrix}
b^{\ast} & a^{\ast} \\
0 & \lambda_{1}^{\ast}
\end{smallmatrix}}} \right]} & Q_{1}^{\ast}\oplus P_{1}^{\ast}\ar[r] & Y\ar[r]  & 0.}$$
where by definition $\overline{\mathrm{TR}}(C_{1},f,C_{2})=(\mathrm{Tr}(C_{3}),g,Y)$ with $\mathrm{Tr}$ the transpose in $\mathrm{mod}(\mathcal{C})$.
Then we can complete to the diagram
$$\xymatrix{Q_{0}^{\ast}\ar[r]^{b^{\ast}}\ar[d]_{\left[
{{\begin{smallmatrix}
1 \\
0
\end{smallmatrix}}} \right]}  & Q_{1}^{\ast}\ar[r]\ar[d]^{\left[
{{\begin{smallmatrix}
1 \\
0
\end{smallmatrix}}} \right]} & \mathrm{Tr}(C_{3})\ar[d]^{g}\ar[r] & 0\\
Q_{0}^{\ast}\oplus P_{0}^{\ast}\ar[r]^{\left[
{{\begin{smallmatrix}
b^{\ast} & a^{\ast} \\
0 & \lambda_{1}^{\ast}
\end{smallmatrix}}} \right]}\ar[d]_{\left[
{{\begin{smallmatrix}
0  & 1\\
\end{smallmatrix}}} \right]} & Q_{1}^{\ast}\oplus P_{1}^{\ast}\ar[r]\ar[d]^{\left[
{{\begin{smallmatrix}
0 & 1 \\
\end{smallmatrix}}} \right]} & Y\ar[r]\ar[d]^{h} & 0\\
P_{0}^{\ast}\ar[r]^{\lambda_{1}^{\ast}}\ar[d] & P_{1}^{\ast}\ar[r]\ar[d] & \mathrm{Tr}(C_{1})\ar[r]\ar[d] & 0\\ 
0 & 0 & 0}$$
Applying the duality $\widehat{\Theta}:\Big(\mathrm{Mod}(\mathcal{C}^{op}),\overline{\mathbb{G}}(\mathrm{Mod}(\mathcal{C}^{op})\Big)\longrightarrow \Big(\mathrm{Mod}(\mathcal{C}),\overline{\mathbb{G}}(\mathrm{Mod}(\mathcal{C})\Big)$ (see \cite[Proposition 4.9]{LeOS} and \ref{descrimaps}) we get 

$$\xymatrix{ & 0\ar[d] & 0\ar[d] & 0\ar[d]\\
0\ar[r] & \mathbb{D}_{\mathcal{C}^{op}}\mathrm{Tr}(C_{1})\ar[r]\ar[d]^{\mathbb{D}_{\mathcal{C}^{op}}(h)} & \mathbb{D}_{\mathcal{C}^{op}}(P_{1}^{\ast})\ar[r]\ar[d] & \mathbb{D}_{\mathcal{C}^{op}}(P_{0}^{\ast})\ar[d]\\
0\ar[r] & \mathbb{D}_{\mathcal{C}^{op}}(Y)\ar[r]\ar[d]^{\mathbb{D}_{\mathcal{C}^{op}}(g)} & \mathbb{D}_{\mathcal{C}^{op}}(Q_{1}^{\ast}\oplus P_{1}^{\ast})\ar[r]\ar[d] &
\mathbb{D}_{\mathcal{C}^{op}}(Q_{0}^{\ast}\oplus P_{0}^{\ast})\ar[d]\\
0\ar[r] & \mathbb{D}_{\mathcal{C}^{op}}(\mathrm{Tr}(C_{3}))\ar[r] & \mathbb{D}_{\mathcal{C}^{op}}(Q_{1}^{\ast})\ar[r] & \mathbb{D}_{\mathcal{C}^{op}}(Q_{0}^{\ast})}$$
Therefore, we get that the Auslander-Reiten translation of $(C_{1},f,C_{2})$ is the map $(\mathbb{D}_{\mathcal{C}^{op}}(Y),\mathbb{D}_{\mathcal{C}^{op}}(g),\mathbb{D}_{\mathcal{C}^{op}}\mathrm{Tr}(C_{3}))$.

\end{proof}

\begin{remark}
Since the minimal projective presentation of $C_{2}$ is a direct summand of
$$\xymatrix{P_{1}\oplus Q_{1}\ar[r]^{^{\left[
{{\begin{smallmatrix}
\lambda_{1} &  a\\
0 & b
\end{smallmatrix}}} \right]}} & P_{0}\oplus Q_{0}\ar[r]^{\gamma_{0}} & C_{2}\ar[r] & 0, }$$
It can be seen that $Y\simeq \mathrm{Tr}(C_{2})\oplus Z$ for some $Z$. Therefore $\mathbb{D}_{\mathcal{C}^{op}}(Y)\simeq \mathbb{D}_{\mathcal{C}^{op}}\mathrm{Tr}(C_{2})\oplus \mathbb{D}_{\mathcal{C}^{op}}(Z)$.
\end{remark}

\section{Almost Split Sequences in the maps category}
Dualizing $K$-categories were introduced by Auslander and Reiten as a generalization of artin $K$-algebras (see \cite{Auslander2}). It is well-known that the existence of almost split sequences is quite useful in the representation theory of artin algebras. A $K$-category $\mathcal{A}$ being dualizing ensures that the category $\mathrm{mod}(\mathcal{A})$ of finitely presented functors in $\mathrm{Mod}(\mathcal{A})$ has almost split sequences (see theorem 7.1.3 in \cite{Reiten}). From a given dualizing $K$-category $\mathcal{A}$ there are some known constructions of dualizing $K$-categories such as $\mathrm{mod}(\mathcal{A})$, the functorially finite Krull-Schmidt categories of $\mathcal{A}$, residue categories $\mathcal{A}/(1_{A})$ of $\mathcal{A}$ module the ideal $(1_{A})$ of $\mathcal{A}$ generated by the identity morphism $1_{A}$ of an object $A\in \mathcal{A}$ and the category $C^{b}(\mathrm{mod}(\mathcal{A}))$ of bounded complexes over $\mathrm{mod}(\mathcal{A})$ (see \cite{Bautista}).\\
Let $\mathcal C$ be a dualizing $K$-variety, and $\mathbf{\Lambda}=\left[\begin{smallmatrix}
 \mathcal C& 0 \\ 
 \widehat{\mathbbm{Hom}}& \mathcal C
\end{smallmatrix}\right]$.
Now, we consider almost split sequences  in $\mathrm{mod}(\mathbf{\Lambda})$ that arise from almost split sequences in $\mathrm{mod}(\mathcal C)$. That is,  we consider almost split sequences in $\mathrm{mod}(\mathbf{\Lambda})\simeq \mathrm{maps}(\mathrm{mod}(\mathcal{C}))$ of the form
\begin{equation*}
0\rightarrow (N_{1},g,N_{2})\xrightarrow {(j_2, \ j_1)}(E_{1},h,E_{2})%
\xrightarrow{(\pi_1, \ \pi_2)}(M_{1},f,M_{2})\rightarrow 0\text{,}
\end{equation*}%
such that $(M_{1},f,M_{2})$ is one of the following cases $%
(M,1_{M},M),(M,0,0),(0,0,M),$  with $M$ a non projective indecomposable $\mathcal C $-module, and 
 $(N_{1},g,N_{2})$ is one of the following cases $%
(N,1_{N},N),(N,0,0),(0,0,N)$, with $N$ a non injective indecomposable $\mathcal C $-module.\\
The following which is a generalization of  \cite[Theorem 3.1(a), Theorem 3.2 (a)]{MVOM}.

\begin{proposition}\label{ARseq1}
Let $\mathcal{C}$ be a dualizing $K$-variety.
\begin{itemize}
\item[(1)] Let  $0\rightarrow \tau M\xrightarrow{j}E\xrightarrow{\pi}M\rightarrow 0$  be
an  almost   split sequence of $\mathcal C $-modules. Then the exact sequences in $\mathrm{maps}(\mathrm{mod}(\mathcal C))$:
\begin{itemize}
\item[(i)] $0\rightarrow (\tau M,0,0)\xrightarrow{(j,\ 0)}(E,\pi,M )
\xrightarrow{(\pi,\ 1_M)}(M, 1_{M},M)\rightarrow 0$,

\item[(ii)] $0\rightarrow (\tau M,1_{\tau M},\tau M)
\xrightarrow {(1_{\tau
M},\ j)}(\tau M,j,E)\xrightarrow {(0, \ \pi)}(0,0,M)\rightarrow 0$,
\end{itemize}

are almost split.

\item[(2)] Let $0\rightarrow N\xrightarrow{j}E\xrightarrow {\pi}\tau^{-1}N\rightarrow 0$
an almost split sequence of $\mathcal C$-modules.
 Then the exact sequences in $\mathrm{maps}(\mathrm{mod}(\mathcal C))$:

\begin{itemize}
\item[(i)] $0\rightarrow (N,1_N,N)\xrightarrow{(1_N,\ j)}(N,j,E)%
\xrightarrow{(0,\ \pi)}(0,0,\tau^{-1}N)\rightarrow 0$

\item[(ii)] $0\rightarrow (N,0,0)\xrightarrow {(j, \ 0)} (E,\pi, \tau^{-1}N)
\xrightarrow {(\pi, \ 1_{\tau^{-1}N})} (\tau^{-1}N,1_{\tau^{-1}N},\tau^{-1}N)\rightarrow 0$
\end{itemize}

are almost split.

\end{itemize}

\end{proposition}

\begin{proof}
(1) (i) Since $\pi :E\rightarrow M$ does not splits, the map $(\pi
,1_{M}):(E,\pi,M )\rightarrow (M,1_{M},M)$ does not split. Let $
(q_{1},q_{2}):(X_{1},f,X_{2})\rightarrow (M,1_{M},M)$ be a map that is not a splittable epimorphism. Then $q_{2}f=1_{M}q_{1}=q_{1}$.

We claim that $q_{1}$ is not a splittable epimorphism. Indeed, if $q_{1}$ is a
splittable epimorphism, then there exists a morphism $s:M\rightarrow X_{1}$,
such that $q_{1}s=1_{M}$. Thus, we have a morphism $(s,fs):(M,1_{M},M)\rightarrow (X_1,f, X_2)$  and we get that
 $(q_1,q_2)\circ (s,fs)= (1_M,1_M)=1_{(M,1_M,M)}: (M,1_M,M)\rightarrow (M,1_{M},M)$  and hence $(q_{1},q_{2}):(X_{1},X_{2},f)\rightarrow
(M,M,1_{M})$ is a splittable epimorphism which is a contradiction.\\
Since $\pi :E\rightarrow M$ is a right almost split morphism, there exists a
map $h:X_{1}\rightarrow E$ such that $\pi h=q_{1}$, and $q_{2}f=q_{1}=\pi h$.
Thus, we have a morphism $(h,q_2):(X_1,f,X_2)\rightarrow (E,\pi,M)$,  and the following commutative  diagram 
\begin{equation*}
\begin{diagram}
\node{}
 \node{(X_1,f,X_2)}\arrow{s,r}{(q_1,\ q_2)}\arrow{sw,l}{(h,\ q_2)}\\
 \node{(E,\pi,M)}\arrow{e,t}{(\pi,\  1_M)}\node{(M,1_M,M)}
 \end{diagram}
\end{equation*}
That is, we get a lifting $(h,\
q_{2}):(X_{1},f,X_{2})\rightarrow (E,\pi, M )$ of $(q_{1},q_{2})$ and we have proved that $(\pi,1_{M})$ is right almost split and thus
$\tau (M,1_{M},M)=(\tau M,0,0)$.\\
(ii). Let $(q_{1},q_{2}):(X_{1},f,X_{2})\rightarrow (0,0,M)$ be a map that is not a splittable epimorphism. Then $q_{2}:X_{2}\rightarrow M$ is not a splittable epimorphism and $q_{1}=0$. Since $\pi:E\rightarrow M$ is a right almost split epimorphism, there exists $u:X_{2}\rightarrow E$ such that $q_{2}=\pi u$. Then $\pi(u f)=(\pi u )f=q_{2}f=0 q_{1}=0$ and since $j=\mathrm{Ker}(\pi)$ we have that there exists a morphism $v:X_{1}\rightarrow \tau (M)$ such that $uf= jv$.
Therefore we wet a map $(v,u):(X_{1},f,X_{2})\rightarrow (\tau M,j,E)$ and we have that $(0,\pi)(v,u)=(0,q_{2})=(q_{1},q_{2})$. Proving that $(0,\pi)$ is right almost split.\\
\bigskip
(2) follows by duality.
\end{proof}

\begin{proposition}\label{minimalmapscer}
 Given a minimal projective presentation $P_{1}\xrightarrow{d_1}P_{0}
\xrightarrow {d_0}M\rightarrow 0$ of $M$ in $\mathrm{mod}(\mathcal{C})$,  we have that the following diagram
$$\xymatrix{P_{1}\ar[r]^{d_{1}}\ar[d]_{\left[
{{\begin{smallmatrix}
1 \\
0
\end{smallmatrix}}} \right]}& P_{0}\ar[r]^{d_{0}}\ar[d]^{1} & M\ar[r]\ar[d] & 0\\
P_{1}\amalg P_{0}\ar[r]^{[d_{1},1]} & P_{0}\ar[r] & 0\ar[r] & 0}$$
is a minimal projective presentation of $M\rightarrow 0$ in the category $\mathrm{maps}(\mathrm{mod}(\mathcal{C}))$.
\end{proposition}
\begin{proof}
Is easy to see that $1:P_{0}\rightarrow P_{0}$ is a minimal projective cover of $M\rightarrow 0$.
We have the following commutative diagram
$$\xymatrix{0\ar[r] & K\ar[r]^{u_{1}}\ar[d]_{u_{1}}& P_{0}\ar[r]^{d_{0}}\ar[d]^{1} & M\ar[r]\ar[d] & 0\\
0\ar[r] & P_{0}\ar[r]^{1} & P_{0}\ar[r] & 0\ar[r] & 0}$$
Making the construction of \ref{projectivecover}, we get the diagram
$$\xymatrix{P_{1}\ar[r]^{\left[
{{\begin{smallmatrix}
1 \\
0
\end{smallmatrix}}} \right]}\ar[d]^{\alpha_{1}} & P_{1}\oplus P_{0}\ar[r]^{[0,1]}\ar[d]^{[d_{1},1]} & P_{0}\ar[r]\ar[d]^{d_{0}} & 0\\
K\ar[r]_{u_{1}}\ar[d] & P_{0}\ar[r]_{d_{0}}\ar[d] & M\ar[r]\ar[d]& 0\\
0 & 0 & 0}$$
Therefore, pasting the projective covers we get the diagram
$$\xymatrix{P_{1}\ar[r]^{d_{1}}\ar[d]_{\left[
{{\begin{smallmatrix}
1 \\
0
\end{smallmatrix}}} \right]}& P_{0}\ar[r]^{d_{0}}\ar[d]^{1} & M\ar[r]\ar[d] & 0\\
P_{1}\amalg P_{0}\ar[r]^{[d_{1},1]} & P_{0}\ar[r] & 0\ar[r] & 0.}$$
\end{proof}

\begin{proposition}\label{firstermARseq}
Let $\mathcal{C}$ be a dualizing $R$-variety for some commutative artin ring $R$. Let $C$ be an indecomposable non projective object and 
Let $$\eta: \xymatrix{0\ar[r] & DTrC\ar[r]^{f} & B\ar[r]^{g} & C\ar[r] & 0}$$ a non split exact sequence such that every non isomorphism $C\rightarrow C$ factors through $g$. Then $\eta$ is an almost split sequence.
\end{proposition}
\begin{proof}
The proof of \cite[2.1]{ARS} in page 147, can be adapted for this setting.
\end{proof}

\begin{proposition}\label{firstterm}
Let $\mathcal{C}$ be a dualizing $K$-variety and $0\longrightarrow A\longrightarrow B\longrightarrow C\longrightarrow 0$ an almost split sequence in $\mathrm{mod}(\mathcal{C})$. Then $A\simeq \mathrm{DTr}(C)$.
\end{proposition}
\begin{proof}
See \cite[Proposition 7.1.4]{Reiten} in page 90.
\end{proof}

We know that $\mathbf{\Lambda}=\left[\begin{smallmatrix}
 \mathcal C& 0 \\ 
 \widehat{\mathbbm{Hom}}& \mathcal C
\end{smallmatrix}\right]$ is a dualizing $K$-variety (see \cite[Proposition 6.10 ]{LeOS}) if $\mathcal{C}$ is dualizing; and therefore by \ref{firstterm}, we have that the first term of an almost split sequence in
$\mathrm{maps}(\mathrm{mod}(\mathcal{C}))\simeq \mathrm{mod}(\mathbf{\Lambda})$ (see \ref{mapsmodlam}) is determined by the ending term.\\
The following which is a generalization of  \cite[Theorem 3.1(b), Theorem 3.2 (b)]{MVOM}.

\begin{proposition} \label{ARseq2}
\begin{itemize}
\item[(i)] Let  $0\rightarrow \tau M\xrightarrow{j}E\xrightarrow{\pi}M\rightarrow 0$  be
an  almost   split sequence in  $\mathrm{mod}(\mathcal{C})$. Given a minimal projective resolution $P_{1}\xrightarrow{d_1}P_{0}
\xrightarrow {d_0}M\rightarrow 0$, we obtain a commutative diagram:
$$\xymatrix{0\ar[r] & \tau(M)\ar[rr]^{j}\ar@{=}[d]&&  E\ar[rr]^{\pi}\ar[d]^{f}  & & M\ar[r]\ar[d]^{h} & 0\\
0\ar[r]  & \tau(M)\ar[rr]^{u=\mathbb{D}_{\mathcal{C}^{op}}(q)} & & 
\mathbb{D}_{\mathcal{C}^{op}}(P_{1}^{\ast})\ar[rr]^{\mathbb{D}_{\mathcal{C}^{op}}(d_{1}^{\ast})} & & \mathbb{D}_{\mathcal{C}^{op}}(P_{0}^{\ast}).}$$
Then the exact sequence
$$\xymatrix{0\ar[r] & \mathbb{D}_{\mathcal{C}^{op}}(P_{1}^{\ast})\ar[rr]^{\left[
{{\begin{smallmatrix}
1 \\
0
\end{smallmatrix}}} \right]}\ar[d]_{\mathbb{D}_{\mathcal{C}^{op}}(d_{1}^{\ast})} & & \mathbb{D}_{\mathcal{C}^{op}}(P_{1}^{\ast})\amalg M\ar[rr]^{[0,1]}\ar[d]^{[\mathbb{D}_{\mathcal{C}^{op}}(d_{1}^{\ast}),h]} & & M\ar[r]\ar[d] & 0\\
0\ar[r] & \mathbb{D}_{\mathcal{C}^{op}}(P_{0}^{\ast})\ar[rr]^{1} & &\mathbb{D}_{\mathcal{C}^{op}}(P_{0}^{\ast})\ar[rr] && 0\ar[r] & 0}$$
is an almost split sequence in $\mathrm{maps}(\mathrm{mod}(\mathcal C))$.

\item[(ii)] Let $0\rightarrow N\xrightarrow{j}E\xrightarrow {\pi}\tau^{-1}N\rightarrow 0$
an almost split sequence in $\mathrm{mod}(\mathcal{C})$. Given a minimal injective resolution  $0\rightarrow N\xrightarrow{q_0}I_{0}\xrightarrow{q_1}I_{1}$ , we obtain a commutative
diagram
$$\xymatrix{ & (\mathbb{D}_{\mathcal{C}}(I_{0}))^{\ast}\ar[rr]^{
(\mathbb{D}_{\mathcal{C}}(q_{1}))^{\ast}}\ar[d]_{v} && (\mathbb{D}_{\mathcal{C}}(I_{0}))^{\ast}\ar[rr]^{s}\ar[d]^{\overline{v}} & & \tau^{-1}(N)\ar@{=}[d]\\
0\ar[r] & N\ar[rr]^{j} && E\ar[rr]^{\pi} & & \tau^{-1}(N)\ar[r] & 0}$$
Then the exact sequence
$$\xymatrix{0\ar[r] & 0\ar[rr]\ar[d] & & (\mathbb{D}_{\mathcal{C}}(I_{0}))^{\ast}\ar[rr]^{1}\ar[d]^{\left[
{{\begin{smallmatrix}
(\mathbb{D}_{\mathcal{C}}(q_{1}))^{\ast}\\
v
\end{smallmatrix}}} \right]} & & (\mathbb{D}_{\mathcal{C}}(I_{0}))^{\ast}\ar[r]\ar[d]^{(\mathbb{D}_{\mathcal{C}}(q_{1}))^{\ast}} & 0\\
0\ar[r] & N\ar[rr]_{{\left[
{{\begin{smallmatrix}
0\\
1
\end{smallmatrix}}} \right]}} & & (\mathbb{D}_{\mathcal{C}}(I_{1}))^{\ast}\amalg N\ar[rr]_{{\left[
{{\begin{smallmatrix}
1, & 0\\
\end{smallmatrix}}} \right]}} & & (\mathbb{D}_{\mathcal{C}}(I_{1}))^{\ast}\ar[r] & 0}$$
is an almost split sequence in $\mathrm{maps}(\mathrm{mod}(\mathcal C))$.

\end{itemize}
\end{proposition}
\begin{proof} (i).
Let  $0\rightarrow \tau M\xrightarrow{j}E\xrightarrow{\pi}M\rightarrow 0$  be
an  almost   split sequence of $\mathcal C$-modules in $\mathrm{mod}(\mathcal{C})$  and consider  a minimal projective presentation $P_{1}\xrightarrow{d_1}P_{0} \xrightarrow {d_0}M\rightarrow 0$ for $M$. Then we get $\xymatrix{P_{0}^{\ast}\ar[r]^{d_{1}^{\ast}} & P_{1}^{\ast}\ar[r]^{q}& \mathrm{Tr}(M)\ar[r] & 0}$ and applying $\mathbb{D}_{\mathcal{C}^{op}}$ we get
$$\xymatrix{0\ar[r]  & \mathbb{D}_{\mathcal{C}^{op}}\mathrm{Tr}(M)\ar[rr]^{u=\mathbb{D}_{\mathcal{C}^{op}}(q)} & &
\mathbb{D}_{\mathcal{C}^{op}}(P_{1}^{\ast})\ar[rr]^{\mathbb{D}_{\mathcal{C}^{op}}(d_{1}^{\ast})}  & & \mathbb{D}_{\mathcal{C}^{op}}(P_{0}^{\ast}).}$$
Since $\mathbb{D}_{\mathcal{C}^{op}}(P_{1}^{\ast})$ is injective there exists a map
$f:E\rightarrow \mathbb{D}_{\mathcal{C}^{op}}(P_{1}^{\ast})$ such that $fj=u$. Then we have the following commutative diagram
$$(\ast):\xymatrix{0\ar[r] & \tau(M)\ar[rr]^{j}\ar@{=}[d]&&  E\ar[rr]^{\pi}\ar[d]^{f}  & & M\ar[r]\ar[d]^{h} & 0\\
0\ar[r]  & \tau(M)\ar[rr]^{u=\mathbb{D}_{\mathcal{C}^{op}}(q)} & & 
\mathbb{D}_{\mathcal{C}^{op}}(P_{1}^{\ast})\ar[rr]^{\mathbb{D}_{\mathcal{C}^{op}}(d_{1}^{\ast})} & & \mathbb{D}_{\mathcal{C}^{op}}(P_{0}^{\ast})}$$
We note that $h\neq 0$ since the upper exact sequence does not split.\\
Next, we will show that the following diagram defines an almost split sequence
$$(\ast\ast):\xymatrix{0\ar[r] & \mathbb{D}_{\mathcal{C}^{op}}(P_{1}^{\ast})\ar[rr]^{\left[
{{\begin{smallmatrix}
1 \\
0
\end{smallmatrix}}} \right]}\ar[d]_{\mathbb{D}_{\mathcal{C}^{op}}(d_{1}^{\ast})} & & \mathbb{D}_{\mathcal{C}^{op}}(P_{1}^{\ast})\amalg M\ar[rr]^{[0,1]}\ar[d]^{[\mathbb{D}_{\mathcal{C}^{op}}(d_{1}^{\ast}),h]} & & M\ar[r]\ar[d] & 0\\
0\ar[r] & \mathbb{D}_{\mathcal{C}^{op}}(P_{0}^{\ast})\ar[rr]^{1} & &\mathbb{D}_{\mathcal{C}^{op}}(P_{0}^{\ast})\ar[rr] && 0\ar[r] & 0}$$
Indeed, since $h\neq 0$ we have that exact sequence in the category $\mathrm{maps}(\mathrm{mod}(\mathcal{C}))$ does not split. By \ref{minimalmapscer}, we have  that the following diagram
$$\xymatrix{P_{1}\ar[r]^{d_{1}}\ar[d]_{\left[
{{\begin{smallmatrix}
1 \\
0
\end{smallmatrix}}} \right]}& P_{0}\ar[r]^{d_{0}}\ar[d]^{1} & M\ar[r]\ar[d] & 0\\
P_{1}\amalg P_{0}\ar[r]^{[d_{1},1]} & P_{0}\ar[r] & 0\ar[r] & 0}$$
is a minimal projective presentation of $M\rightarrow 0$.\\
Consider the diagram $\xymatrix{P_{1}\ar[r]^{d_{1}}\ar[d]_{\left[
{{\begin{smallmatrix}
1 \\
0
\end{smallmatrix}}} \right]}& P_{0}\ar[d]^{[1,0]} \\
P_{1}\amalg P_{0}\ar[r]^{\left[
{{\begin{smallmatrix}
d_{1} & 1 \\
0 & 0
\end{smallmatrix}}} \right]}& P_{0}\amalg 0.}$\\
By \ref{morphismdual}, applying $(-)^{\ast}$ we have the diagram
$$\xymatrix{0\ar[r]\ar[d] & P_{0}^{\ast}\ar[d]^{\left[
{{\begin{smallmatrix}
1 \\
0 
\end{smallmatrix}}} \right]}\\
P_{0}^{\ast}\ar[r]^{\left[
{{\begin{smallmatrix}
1 \\
d_{1}^{\ast}
\end{smallmatrix}}} \right]} & P_{0}^{\ast}\amalg P_{1}^{\ast}.}$$
We have the following exact sequence
$\xymatrix{
P_{0}^{\ast}\ar[r]^{\left[
{{\begin{smallmatrix}
1 \\
d_{1}^{\ast}
\end{smallmatrix}}} \right]} & P_{0}^{\ast}\amalg P_{1}^{\ast}
\ar[r]^(.55){[d_{1}^{\ast},-1]} & P_{1}^{\ast}\ar[r] & 0.}$
Then we have the exact sequence in $\mathrm{maps}(\mathrm{mod}(\mathcal{C}))$
$$\xymatrix{0\ar[r]\ar[d] & P_{0}^{\ast}\ar[r]^{1}\ar[d]^{\left[
{{\begin{smallmatrix}
1 \\
0
\end{smallmatrix}}} \right]}  & P_{0}^{\ast}\ar[d]^{d_{1}^{\ast}}\ar[r] & 0\\
P_{0}^{\ast}\ar[r]^{\left[
{{\begin{smallmatrix}
1 \\
d_{1}^{\ast}
\end{smallmatrix}}} \right]} & P_{0}^{\ast}\amalg P_{1}^{\ast}
\ar[r]^{[d_{1}^{\ast},-1]} & P_{1}^{\ast}\ar[r]& 0}$$
Therefore the $\overline{\mathrm{TR}}(M,0,0)=(P_{0}^{\ast},d_{1}^{\ast},P_{1}^{\ast})$.
Now by \ref{descrimaps},  applying duality $D=\overline{\Theta}$ in the maps category, we  have that $D(P_{0}^{\ast}\rightarrow P_{1}^{\ast})$ is given by the map $\mathbb{D}_{\mathcal{C}^{op}}(d_{1}^{\ast}):\mathbb{D}_{\mathcal{C}^{op}}(P_{1}^{\ast})\rightarrow \mathbb{D}_{\mathcal{C}^{op}}(P_{0}^{\ast})$.\\
In this way, we conclude that almost split sequences in the category $\mathrm{maps}(\mathrm{mod}(\mathcal{C}))$ that have ending term $M\rightarrow 0$ must have first term $\mathbb{D}_{\mathcal{C}^{op}}(d_{1}^{\ast}):\mathbb{D}_{\mathcal{C}^{op}}(P_{1}^{\ast})\rightarrow \mathbb{D}_{\mathcal{C}^{op}}(P_{0}^{\ast})$. Now, by \ref{firstermARseq} in order to show that the diagram $(\ast\ast)$ define an almost exact sequence in $\mathrm{maps}(\mathrm{mod}(\mathcal{C}))$ is enough to see that every not isomorphism $(\alpha,0):(M,0,0)\longrightarrow (M,0,0)$
factors through $([0\,\,1],0):\Big(\mathbb{D}_{\mathcal{C}^{op}}(P_{1}^{\ast})\amalg M, [\mathbb{D}_{\mathcal{C}^{op}}(d_{1}^{\ast}),h],\mathbb{D}_{\mathcal{C}^{op}}(P_{0}^{\ast})\Big)\longrightarrow (M,0,0)$.\\
Indeed, we have that $\alpha:M\rightarrow M$ is not an isomorphism. Since $0\rightarrow \tau M\xrightarrow{j}E\xrightarrow{\pi}M\rightarrow 0$ is
an  almost   split sequence of $\mathcal C$-modules, we have that there exists $\alpha':M\longrightarrow E$ such that $\alpha=\pi \alpha'$. Considering the exact diagram $(\ast)$, we have that 
$$h\circ\alpha=\mathbb{D}_{\mathcal{C}^{op}}(d_{1}^{\ast})\circ f\circ \alpha'.$$
Since $[\mathbb{D}_{\mathcal{C}^{op}}(d_{1}^{\ast}),h]\left[
{{\begin{smallmatrix}
-f\alpha'\\
\alpha
\end{smallmatrix}}} \right]=h\circ\alpha-\mathbb{D}_{\mathcal{C}^{op}}(d_{1}^{\ast})\circ f\circ \alpha'=0$, we have the following diagram is commutative
$$\xymatrix{M\ar[rr]^{\left[
{{\begin{smallmatrix}
-f\alpha'\\
\alpha
\end{smallmatrix}}} \right]}\ar[d] & &\mathbb{D}_{\mathcal{C}^{op}}(P_{1}^{\ast})\amalg M\ar[rr]^{[0,1]}\ar[d]^{[\mathbb{D}_{\mathcal{C}^{op}}(d_{1}^{\ast}),h]} & & M\ar[d]^{0}\\
0\ar[rr] & & \mathbb{D}_{\mathcal{C}^{op}}(P_{0}^{\ast})\ar[rr] & & 0.}$$ Now, since 
$[0,1]\left[
{{\begin{smallmatrix}
-f\alpha'\\
\alpha
\end{smallmatrix}}} \right]=\alpha$, we conclude that the last diagram is the same as the morphism $(\alpha,0):(M,0,0)\longrightarrow (M,0,0)$, proving the required condition. Therefore, by \ref{firstermARseq} we conclude that the diagram $(\ast\ast)$ defines an almost split sequence in the category $\mathrm{maps}(\mathrm{mod}(\mathcal{C})$.\\
$(ii)$  Similar to $(i)$.
\end{proof}

Now, we define a functor which will give us a relation between almost split sequences in $\mathrm{mod}(\mathrm{mod}(\mathcal{C}))$ and almost split sequences in $\mathrm{mod}(\mathbf{\Lambda})$.

\begin{definition}\label{mapsmodmod}
Let $\Phi:\mathrm{maps}(\mathrm{mod}(\mathcal{C}))\longrightarrow \mathrm{mod}(\mathrm{mod}(\mathcal{C})^{op})$
given by
$$\Phi(\xymatrix{A_{1}\ar[r]^{f} & A_{0})}=\mathrm{Coker}\Big(\xymatrix{(-,A_{1})\ar[r]^{(-,f)} & (-,A_{0})}\Big).$$
\end{definition}

We can see now that the functor $\Phi $ preserves almost split sequences.

The following is a generalization of \cite[Theorem 3.4]{MVOM}

\begin{theorem} \label{ARseq3}
Let 
\[
 0\rightarrow (N_{1},g,N_{2})\xrightarrow{(j_1, \ j_2)}(E_{1},h,E_{2})%
\xrightarrow{(p_1, \ p_2)}(M_{1},f,M_{2})\rightarrow 0
\]
be an almost split sequence in $\mathrm{maps}(\mathrm{mod}(\mathcal{C}))$, such that $g,f,$ are neither splittable epimorphisms nor
splittable monomorphisms. Then the exact sequence
\[
 0\rightarrow G%
\xrightarrow{\rho}H\xrightarrow{\sigma}F\rightarrow 0
\]
 obtained from the commutative diagram:
\begin{equation*}
\begin{diagram}\dgARROWLENGTH=.7em \node{0}\arrow{s} \node{0}\arrow{s}
\node{0}\arrow{s} \node{}\\
\node{(-,N_1)}\arrow{s,l}{(-,j_1)}\arrow{e,t}{(-,g)}
\node{(-,N_2)}\arrow{s,l}{(-,j_2)}\arrow{e,t}{}
\node{G}\arrow{s,l}{\rho}\arrow{e,t}{} \node{0}\\
\node{(-,E_1)}\arrow{s,l}{(-,p_1)}\arrow{e,t}{(-,h)}
\node{(-,E_2)}\arrow{s,l}{(-,p_2)}\arrow{e,t}{}
\node{H}\arrow{s,l}{\sigma}\arrow{e,t}{} \node{0}\\
\node{(-,M_1)}\arrow{s}\arrow{e,t}{(-,f)}
\node{(-,M_2)}\arrow{s}\arrow{e,t}{} \node{F}\arrow{s}\arrow{e,t}{}
\node{0}\\ \node{0} \node{0} \node{0} \node{} \end{diagram}
\end{equation*}
is an almost split sequence in $ \mathrm{mod}(\mathrm{mod}(\mathcal{C})^{op})$.
\end{theorem}

\begin{proof}
Same proof given in \cite[Theorem 3.4]{MVOM} works for this setting.
\end{proof}

\section{Functorially finite subcategories}
Let $\mathcal{C}$ be an arbitrary category.  Let $\mathcal{X}$ be a subcategory of $\mathcal{C}$.  A morphism $f:X\longrightarrow M$ in $\mathcal{C}$ with $X\in \mathcal{X}$ is a $\textbf{right}$ $\mathcal{X}$-$\textbf{approximation}$ of $M$  if $\mathrm{Hom}_{\mathcal{C}}(Z,X)\longrightarrow \mathrm{Hom}_{\mathcal{C}}(Z,M)$ is surjective for every $Z\in \mathcal{X}$. Dually, a morphism $g:M\longrightarrow X$  with $X\in \mathcal{X}$ is a $\textbf{left}$ $\mathcal{X}$-$\textbf{approximation}$ if $\mathrm{Hom}_{\mathcal{C}}(X,Z)\longrightarrow \mathrm{Hom}_{\mathcal{C}}(M,Z)$ is surjective for every $Z\in \mathcal{X}$.\\
A subcategory $\mathcal{X}$  of $\mathcal{C}$ is $\textbf{contravariantly}$ ($\textbf{covariantly}$) $\textbf{finite}$ in $\mathcal{C}$ if every object $M\in \mathcal{C}$ has a right (left) $\mathcal{X}$-approximation; and $\mathcal{C}$ is $\textbf{functorially}$ $\textbf{finite}$ if it is both
contravariantly and covariantly finite.
\subsection{Functorially finite subcategories in $\mathbf{mod}(\mathbf{\Lambda})$ and  a result of Smal\o}
In this subsection we prove a result that generalizes the given by S. O. {Smal\o}   in \cite[Theorem 2.1]{Smalo} and we will see some implications that it give us respect the category $\mathrm{mod}(\mathcal{C})$ for a dualizing variety $\mathcal{C}$.

\begin{theorem}\label{covacomacat}	
Let $\mathcal{A}$ and $\mathcal{B}$ be abelian categories and $G:\mathcal{B}\longrightarrow \mathcal{A}$ a covariant functor.
Consider the comma category $(G(\mathcal{B}),\mathcal{A})$ and $\mathcal{X}\subseteq \mathcal{A}$ and
$\mathcal{Y}\subseteq \mathcal{B}$ subcategories containing the zero object. We denote by $\mathcal{D}_{\mathcal{X}}^{\mathcal{Y}}$ the full subcategory of $(G(\mathcal{B}),\mathcal{A})$ whose objects are: the morphisms $g:G(B)\longrightarrow A$  with $B\in \mathcal{Y}$ and $A\in \mathcal{X}$.
Then $\mathcal{D}_{\mathcal{X}}^{\mathcal{Y}}$ is covariant finite in  $(G(\mathcal{B}),\mathcal{A})$ if and only if $\mathcal{X}\subseteq \mathcal{A}$ and
$\mathcal{Y}\subseteq \mathcal{B}$ are covariant  finite subcategories.
\end{theorem}
\begin{proof}
$(\Longleftarrow)$. 
Let $g:G(B)\longrightarrow A$ an object in $(G(\mathcal{B}),\mathcal{A})$.
Since $\mathcal{Y}$ is covariant finite, there exists an $\mathcal{Y}$-left approximation $\alpha_{B}:B\longrightarrow Y_{B}$. Then we have the following pushout diagram in $\mathcal{A}$
$$(\ast):\xymatrix{G(B)\ar[r]^{G(\alpha_{B})}\ar[d]_{g} & G(Y_{B})\ar[d]^{g'}\\
A\ar[r]_{\delta} & C}$$
Since $\mathcal{X}$ is covariant finite in $\mathcal{A}$ we have a
 $\mathcal{X}$-left approximation $\beta_{C}:C\longrightarrow X_{C}$. Then we have the following commutative diagram
 $$\xymatrix{G(B)\ar[r]^{G(\alpha_{B})}\ar[d]_{g} & G(Y_{B})\ar[d]^{\beta_{C}g'}\\
A\ar[r]_{\beta_{C}\delta} & X_{C}}$$
Then, we have the object $(Y_{B},\beta_{C}g',X_{C})\in \mathcal{D}_{\mathcal{X}}^{\mathcal{Y}}$. We assert that
$(\alpha_{B},\beta_{C}\delta):(B,g,A)\longrightarrow (Y_{B},\beta_{C}g',X_{C})$ is a $\mathcal{D}_{\mathcal{X}}^{\mathcal{Y}}$-left approximation of $(B,g,A)$.\\
Indeed, let $(\lambda,\varphi):(B,g,A)\longrightarrow (B',f,A')$ a morphism in  $(G(\mathcal{B}),\mathcal{A})$ with $B'\in \mathcal{Y}$ and $A'\in \mathcal{X}$.\\ 
Since $\lambda:B\longrightarrow B'$ is a morphism in $\mathcal{B}$ with $B'\in \mathcal{B}$, there exists a morphism $\lambda':Y_{B}\longrightarrow B'$ such that $\lambda=\lambda'\circ \alpha_{B}$. Then we get the following commutative diagram
$$\xymatrix{G(B)\ar[r]^{G(\alpha_{B})}\ar[d]_{g} & G(Y_{B})\ar[d]^{fG(\lambda')}\\
A\ar[r]^{\varphi} &A'}$$
Since the diagram $(\ast)$ is pushout, there exists a morphism
$\sigma:C\longrightarrow A'$ such that $\sigma g'=fG(\lambda')$ and $\sigma\delta=\varphi$. Since $A'\in \mathcal{X}$ and $\beta_{C}:C\longrightarrow X_{C}$ is a $\mathcal{X}$-left approximation of $C$, there exists a morphism $\sigma':X_{C}\longrightarrow A'$ such that $\sigma=\sigma'\beta_{C}$. We have the diagram
$$\xymatrix{G(B)\ar[r]^{G(\alpha_{B})}\ar[d]_{g} & G(Y_{B})\ar[d]^{\beta_{C}g'}\ar[r]^{G(\lambda')} & G(B')\ar[d]^{f}\\
A\ar[r]_{\beta_{C}\delta} & X_{C}\ar[r]_{\sigma'} & A'}$$
which is commutative since
$fG(\lambda')=\sigma g'=\sigma'\beta_{C}g'$. Moreover we have that $\sigma'\beta_{C}\delta=\sigma\delta=\varphi$ and since 
$\lambda=\lambda'\circ \alpha_{B}$ we have that $G(\lambda)=G(\lambda')G(\alpha_{B})$. Therefore, the morphism $(\lambda,\varphi)$ factor through the morphism $(\alpha_{B},\beta_{C}\delta)$. Proving that $\mathcal{D}_{\mathcal{X}}^{\mathcal{Y}}$ is covariant finite in  $(G(\mathcal{B}),\mathcal{A})$.\\
$(\Longrightarrow)$. Let us suppose that $\mathcal{D}_{\mathcal{X}}^{\mathcal{Y}}$ is covariant finite in  $(G(\mathcal{B}),\mathcal{A})$. Let $B$ be and object in $\mathcal{B}$. Consider the object $g:G(B)\longrightarrow 0$ in $(G(\mathcal{B}),\mathcal{A})$. Since $\mathcal{D}_{\mathcal{X}}^{\mathcal{Y}}$ is covariant finite, there exists and object $\beta:G(Y)\longrightarrow X$ with $Y\in \mathcal{Y}$ and $X\in \mathcal{X}$ and a morphism $(\lambda,\theta):(B,0,0)\longrightarrow (Y,\beta,X)$ which is a left $\mathcal{D}_{\mathcal{X}}^{\mathcal{Y}}$-approximation. We assert that $\lambda:B\rightarrow  Y$ is a left $\mathcal{Y}$-approximation. Indeed, let $\gamma:B\rightarrow Y'$ a morphism in $\mathcal{B}$ with $Y'\in \mathcal{Y}$. Then we have the following commutative diagram
$$\xymatrix{G(B)\ar[r]^{G(\gamma)}\ar[d]^{0} & G(Y')\ar[d]^{0}\\
0\ar[r] & 0.}$$
Since $(\lambda,\theta):(B,0,0)\longrightarrow (Y,\beta,X)$ is a left $\mathcal{D}_{\mathcal{X}}^{\mathcal{Y}}$-approximation, there exists $(\psi,\tau):(Y,\beta,X)\longrightarrow (Y',0,0)$ such that $(\gamma,0)=(\psi,\tau)\circ (\lambda,\theta)$. Therefore, we get that
$\gamma=\psi\lambda$, proving that  $\lambda:B\rightarrow  Y$ is a left $\mathcal{Y}$-approximation. Thus, $\mathcal{Y}$ is covariantly finite in $\mathcal{B}$. Similarly, $\mathcal{X}$ is covariantly finite in $\mathcal{A}$.
\end{proof}

\begin{theorem}\label{biyeccovafini}
Let $\mathcal{A}$ and $\mathcal{B}$ be abelian categories and $F:\mathcal{A}\longrightarrow \mathcal{B}$, $G:\mathcal{B}\longrightarrow \mathcal{A}$ a covariant functors such that $G$ is left adjoint to $F$. Consider the comma category $(\mathcal{B},F(\mathcal{A}))$ and $\mathcal{X}\subseteq \mathcal{A}$, 
$\mathcal{Y}\subseteq \mathcal{B}$ subcategories containing the zero object. We denote by $\mathcal{C}_{\mathcal{X}}^{\mathcal{Y}}$ the full subcategory of $(\mathcal{B},F(\mathcal{A}))$ whose objects are: the morphisms $f:B \longrightarrow F(A)$  with $B\in \mathcal{Y}$ and $A\in \mathcal{X}$.
Then $\mathcal{C}_{\mathcal{X}}^{\mathcal{Y}}$ is funtorially finite in  $(\mathcal{B},F(\mathcal{A}))$ if and only if $\mathcal{X}\subseteq \mathcal{A}$ and
$\mathcal{Y}\subseteq \mathcal{B}$ are functorially finite.
\end{theorem}
\begin{proof}
We have isomorphism
$$\varphi_{B,A}:\mathrm{Hom}_{\mathcal{A}}(G(B),A)\longrightarrow \mathrm{Hom}_{\mathcal{B}}(B,F(A))$$
This defines an isomorphism between the comma categories
$$\Phi:\Big(G(\mathcal{B}),\mathcal{A}\Big)\longrightarrow \Big(\mathcal{B},F(\mathcal{A})\Big)$$
Denote by $\mathcal{D}_{\mathcal{X}}^{\mathcal{Y}}$ the full subcategory of $(G(\mathcal{B}),\mathcal{A})$ whose objects are: the morphisms $g:G(B) \longrightarrow A $  with $B\in \mathcal{Y}$ and $A\in \mathcal{X}$ and $\mathcal{C}_{\mathcal{X}}^{\mathcal{Y}}$ the full subcategory of $(\mathcal{B},F(\mathcal{A}))$ whose objects are: the morphisms $f:B\longrightarrow F(A)$  with $B\in \mathcal{Y}$ and $A\in \mathcal{X}$.
Then we have an isomorphism
$$\Phi:\mathcal{D}_{\mathcal{X}}^{\mathcal{Y}}\longrightarrow \mathcal{C}_{\mathcal{X}}^{\mathcal{Y}}$$
By \ref{covacomacat}, and its dual we get that $\mathcal{D}_{\mathcal{X}}^{\mathcal{Y}}$ is covariant finite  in 
$\Big(G(\mathcal{B}),\mathcal{A}\Big)$ and  $\mathcal{C}_{\mathcal{X}}^{\mathcal{Y}}$ is contravariantly finite in $\Big(\mathcal{B},F(\mathcal{A})\Big)$. Then, via the isomorphism $\Phi$ we get that $\mathcal{C}_{\mathcal{X}}^{\mathcal{Y}}$ is funtorially finite in  $(\mathcal{B},F(\mathcal{A}))$.
\end{proof}

The following result that generalizes the given by S. O. {Smal\o}   in \cite[Theorem 2.1]{Smalo}.

\begin{corollary}\label{Smalogeneral}
Let $\mathcal{U}$ and $\mathcal{T}$ additive categories and $M\in \mathrm{Mod}(\mathcal{U}\otimes \mathcal{T}^{op})$ and consider $\mathbf{\Lambda}=\left[ \begin{smallmatrix}
 \mathcal T& 0 \\ 
 M& \mathcal U
\end{smallmatrix}\right]$. If $\mathcal{X}\subseteq \mathcal{U}$ and  
$\mathcal{Y}\subseteq \mathcal{T}$  are  subcategories, denote by  $(\mathcal Y,  \mathbb{G} \mathcal X)$ the full subcategory of $\Big(\mathrm{Mod}(\mathcal T), \mathbb{G}\mathrm{Mod}(\mathcal U)\Big)$ whose objects
are morphisms of $\mathcal T$-modules  $f:B\longrightarrow \mathbb{G}(A)$ with  $B\in \mathcal{Y}$ and $A\in \mathcal{X}$. Then $(\mathcal Y,  \mathbb{G} \mathcal X)$ is a covariantly (contravariantly, functorially) finite subcategory of $\mathrm{Mod}\Big(\left[ \begin{smallmatrix}
 \mathcal T& 0 \\ 
 M& \mathcal U
\end{smallmatrix}\right]\Big)$ if and only if  $\mathcal{Y}\subseteq \mathcal{T}$ and $\mathcal{X}\subseteq \mathcal{U}$ are covariantly (contravariantly, functorially) finite.
\end{corollary}
\begin{proof}
By section 5 in \cite{LeOS}, we have left adjoint $\mathbb{F}:\mathrm{Mod}(\mathcal{T})\longrightarrow \mathrm{Mod}(\mathcal{U})$ to $\mathbb{G}$ and by \cite[Theorem 3.14]{LeOS}, there is equivalence  $\Big(\mathrm{Mod}(\mathcal T), \mathbb{G}\mathrm{Mod}(\mathcal U)\Big)\simeq \mathrm{Mod}\Big(\left[ \begin{smallmatrix}
 \mathcal T& 0 \\ 
 M& \mathcal U
\end{smallmatrix}\right]\Big)$. The result follows from \ref{covacomacat} its dual and \ref{biyeccovafini}.
\end{proof}

\begin{corollary}
Let $\mathcal{U}$ and $\mathcal{T}$ dualizing varieties and $M\in \mathrm{Mod}(\mathcal{U}\otimes \mathcal{T}^{op})$  such that $M_{T}\in \mathrm{mod}(\mathcal{U})$ and $M_{U}\in \mathrm{mod}(\mathcal{T})$ for all $U\in \mathcal{U}$ and $T\in \mathcal{T}$. Consider $\mathbf{\Lambda}=\left[ \begin{smallmatrix}
 \mathcal T& 0 \\ 
 M& \mathcal U
\end{smallmatrix}\right]$. If $\mathcal{X}\subseteq \mathcal{U}$ and  
$\mathcal{Y}\subseteq \mathcal{T}$  are  subcategories, denote by  $(\mathcal Y,  \mathbb{G} \mathcal X)$ the full subcategory of $\Big(\mathrm{mod}(\mathcal T), \mathbb{G}\mathrm{mod}(\mathcal U)\Big)$ whose objects
are morphisms of $\mathcal T$-modules  $f:B\longrightarrow \mathbb{G}(A)$ with  $B\in \mathcal{Y}$ and $A\in \mathcal{X}$. Then $(\mathcal Y,  \mathbb{G} \mathcal X)$ is a covariantly (contravariantly, functorially) finite subcategory of $\mathrm{mod}\Big(\left[ \begin{smallmatrix}
 \mathcal T& 0 \\ 
 M& \mathcal U
\end{smallmatrix}\right]\Big)$ if and only if  $\mathcal{Y}\subseteq \mathcal{T}$ and $\mathcal{X}\subseteq \mathcal{U}$ are covariantly (contravariantly, functorially) finite.
\end{corollary}
\begin{proof}
By \cite[Proposition 6.3]{LeOS} we have an equivalence $\Big(\mathrm{mod}(\mathcal T), \mathbb{G}\mathrm{mod}(\mathcal U)\Big)\simeq \mathrm{mod}\Big(\left[ \begin{smallmatrix}
 \mathcal T& 0 \\ 
 M& \mathcal U
\end{smallmatrix}\right]\Big)$. We also have an adjoint pair  $(\mathbb{F}^{\ast},\mathbb{G}^{\ast})$ by \cite[Proposition 6.2]{LeOS}. Therefore the result follows from \ref{covacomacat} its dual and \ref{biyeccovafini}. 
\end{proof}
The following is the generalization of a result of {Smal\o} \cite[Corollary 2.2]{Smalo}.
\begin{corollary}
Let $\mathcal{U}$ and $\mathcal{T}$ dualizing varieties and $M\in \mathrm{Mod}(\mathcal{U}\otimes \mathcal{T}^{op})$  such that $M_{T}\in \mathrm{mod}(\mathcal{U})$ and $M_{U}\in \mathrm{mod}(\mathcal{T})$ for all $U\in \mathcal{U}$ and $T\in \mathcal{T}$. Consider $\mathbf{\Lambda}=\left[ \begin{smallmatrix}
 \mathcal T& 0 \\ 
 M& \mathcal U
\end{smallmatrix}\right]$. If $\mathcal{X}\subseteq \mathcal{U}$ and  
$\mathcal{Y}\subseteq \mathcal{T}$  are  functorially finite which are closed under extensions. Then $(\mathcal Y,  \mathbb{G} \mathcal X)$ is functorially finite subcategory of $\mathrm{mod}\Big(\left[ \begin{smallmatrix}
 \mathcal T& 0 \\ 
 M& \mathcal U
\end{smallmatrix}\right]\Big)$  which is closed under extensions and moreover $(\mathcal Y,  \mathbb{G} \mathcal X)$ has almost split sequences.
\end{corollary}
\begin{proof}
By \cite[Proposition 6.10]{LeOS}, it follows that  $\mathbf{\Lambda}$ is a dualizing $R$-variety and therefore $\mathrm{mod}(\mathbf{\Lambda})$ is Krull-Schmidt. Since  $\mathcal{X}\subseteq \mathrm{mod}(\mathcal{U})$ and  
$\mathcal{Y}\subseteq \mathrm{mod}(\mathcal{T})$   are closed under extensions, this also holds for  $(\mathcal  Y, \mathbb{G}\mathcal X)$.  Thus, $(\mathcal  Y, \mathbb{G}\mathcal X)$  is a Krull-Schmidt subcategory of $(\mathrm{mod}(\mathcal T), \mathbb{G}\mathrm{mod}(\mathcal U))\simeq \mathrm{mod}(\mathbf{\Lambda})$, and the rest of the proof  follows from \cite[Corollary 3.5]{Liu}. 
\end{proof}

\begin{corollary}\label{functor0}
Let $\mathcal{C}$ be a dualizing $K$-variety and  $\mathcal X\subset \mathrm{mod}(\mathcal C)$ be a subcategory.  Denote by
$\mathrm{maps}(\mathcal X)$ the full subcategory of  $\mathrm{maps}(\mathrm{mod}(\mathcal C))$ whose objects
are morphisms of $\mathcal C$-modules  $f:B\longrightarrow A$ with  $A, B\in \mathcal X$. Then $\mathcal X$ is  contravariantly (covariantly, functorially) finite
in  $\mathrm{mod}(\mathcal C)$ if and only if  $\mathrm{maps}(\mathcal X)$ is contravariantly (covariantly, functorially) finite in   $\mathrm{maps}(\mathrm{mod}(\mathcal C))$.
\end{corollary}

\begin{corollary}\label{functorially2}
Let $R$ be a commutative ring and  $\Lambda$  be  an artin $R$-algebra. Consider  $\mathcal{C}=\mathrm{mod}(\Lambda)$ and  $\mathcal X\subset \mathrm{mod}(\mathcal C)$ be a subcategory. Then $\mathcal X$ is  contravariantly (covariantly, functorially) finite
in  $\mathrm{mod}(\mathcal C)$ if and only if  $\mathrm{maps}(\mathcal X)$ is contravariantly (covariantly, functorially) finite in   $\mathrm{maps}(\mathrm{mod}(\mathcal C))$.
\end{corollary}

\subsection{Functorially  finite  subcategories in  $\mathbf{mod}(\mathbf{mod}(\mathcal{C}))$}
Let $\mathcal{C}$ be a dualizing variety. By  Corollary \ref{functor0} there is a close relation between contravariantly, covariantly and functorially  subcategories  in $\mathrm{mod}(\mathrm{mod}\Lambda))$ and  contravariantly, covariantly and functorially  subcategories  in
$\mathrm{maps}(\mathrm{mod}(\mathcal{C}))$.\\
Now, consider the matrix category  $\mathbf{\Lambda}:=\left[
\begin{smallmatrix}
\mathcal{C}  & 0 \\
\widehat{\mathbbm{Hom}}  & \mathcal{C}
\end{smallmatrix}
\right] $. There is an equivalence of categories
$$\Psi: \mathrm{mod}\Big(\left[
\begin{smallmatrix}
\mathcal{C}  & 0 \\
\widehat{\mathbbm{Hom}}  & \mathcal{C}
\end{smallmatrix}
\right] \Big) \rightarrow  \mathrm{maps}(\mathrm{mod}(\mathcal{C})
).$$ 
Recall, we have the functor (see \ref{mapsmodmod})
$$\Phi:\mathrm{maps}(\mathrm{mod}(\mathcal{C}))\longrightarrow \mathrm{mod}(\mathrm{mod}(\mathcal{C})^{op})$$
given by $\Phi(\xymatrix{A_{1}\ar[r]^{f} & A_{0})}=\mathrm{Coker}\Big(\xymatrix{(-,A_{1})\ar[r]^{(-,f)} & (-,A_{0})}\Big).$\\
Now, we list some  ways to get  functorially  finite subcategories  in $\mathrm{maps}(\mathrm{mod}\Lambda
)$.  Obviously $\Psi(\mathcal X)$  is functorially finite in $\mathrm{maps}( \mathrm{mod}(\mathcal{C}))$ if $\mathcal X$ is  a functorially finite subcategory in $\mathrm{mod}\Big(\left[
\begin{smallmatrix}
\mathcal{C}  & 0 \\
\widehat{\mathbbm{Hom}}  & \mathcal{C}
\end{smallmatrix}
\right] \Big)$.
In this part we will see that some properties like: contravariantly,
covariantly, functorially finite subcategories of $\mathrm{maps}(\mathrm{mod}(
\mathcal{C})),$ are preserved by the functor $\Phi$.\\
The following result is a generalization of a result in \cite{MVOM}.

\begin{theorem}
Let $\mathscr C\subset \mathrm{maps}(\mathrm{mod}(\mathcal{C}))$ be a
subcategory. Then the following statements hold:

\begin{itemize}
\item[(a)] If $\mathscr{C}$ is contravariantly finite in $\mathrm{maps}(\mathrm{mod}(\mathcal{C}))$, then $\Phi(\mathscr{C})$ is a contravariantly finite subcategory of $\mathrm{mod}(\mathrm{mod}(\mathcal{C})^{op})$.

\item[(b)] If $\mathscr{C}$ is covariantly finite in $\mathrm{maps}(\mathrm{mod}(\mathcal{C}))$, then $\Phi(\mathscr{C})$ is a covariantly finite subcategory of $\mathrm{mod}(\mathrm{mod}(\mathcal{C})^{op})$.

\item[(c)] If $\mathscr{C}$ is functorially finite in $\mathrm{maps}(\mathrm{mod}(\mathcal{C}))$, then $\Phi(\mathscr{C})$ is a functorially finite subcategory of $\mathrm{mod}(\mathrm{mod}(\mathcal{C})^{op})$.
\end{itemize}
\end{theorem}

\begin{proof}
Same proof as in \cite[Theorem 3.8]{MVOM}
\end{proof}

\begin{remark}
We can define the functor $\Phi':\mathrm{maps}(\mathrm{mod}(\mathcal{C})
)\rightarrow \mathrm{mod}(\mathrm{mod}(\mathcal{C}))$ as:
\begin{equation*}
\Phi':(A_{1}\xrightarrow{f}A_{0})=\mathrm{Coker}((A_{0},-)\xrightarrow{(f,-)}(A_{1},-)).
\end{equation*}
The same properties: contravariantly,
covariantly, functorialy finite subcategories of $\mathrm{maps}(\mathrm{mod}(\mathcal{C})),$ are preserved by the functor $\Phi'$.
\end{remark}

On the other hand, if $\mathcal X$  is a functorially finite subcategory of $\mathrm{mod} (\mathcal{C})$, we get that $\mathrm{maps}(\mathcal X)$ is a functorially finite  subcategory of $\mathrm{maps}(\mathrm{mod}(\mathcal{C}))$, by Corollary \ref{functor0}. Thus  we have a way to get contravariantly (covariantly, functorially) finite subcategories in $\mathrm{maps}(\mathrm{mod} (\mathcal{C}))$  from  the ones  of $\mathrm{mod} (\mathcal{C})$, which of course are in bijective correspondence with ones in $\mathrm{mod}\Big(\left[
\begin{smallmatrix}
\mathcal{C}  & 0 \\
\widehat{\mathbbm{Hom}}  & \mathcal{C}
\end{smallmatrix}
\right]\Big)$ by the equivalence $\Psi$. Thus, we have induced maps: 
$$\xymatrix{{\left \{\text{functorially finite subcategories } \mathscr{C}\subset \mathrm{maps}( \mathrm{mod}(\mathcal{C}))
\right \}}\ar[d]^{\underset{\Phi(\mathscr{C})}{\overset{\mathscr{C}}{\big\downarrow}}}\\
{\left \{\text{functorially finite subcategories } \mathscr{E}\subset \mathrm{mod}( \mathrm{mod}(\mathcal{C})^{op})
\right \}}\ar[d]^{\underset{\mathrm{maps}(\mathscr{E})}{\overset{\mathscr{E}}{\big\downarrow}}}\\
{\left \{\text{functorially finite subcategories } \mathscr{F}\subset \mathrm{maps}( \mathrm{mod}( \mathrm{mod}(\mathcal{C})^{op}))
\right \}}}$$
Finally,  we have the following examples of functorially finite subcategories of the
category $\mathrm{maps}(\mathrm{mod}(\mathcal{C}))$.\\
We denote by $\mathfrak {emaps}(\mathrm{mod}(\mathcal{C}))$  the full subcategory of $\mathrm{maps}(\mathrm{mod}(\mathcal{C}))$ consisting of all  maps $(M_{1},f,M_{2})$, such that $f$ is an epimorphism and by $\mathfrak {mmaps}(\mathrm{mod}(\mathcal{C}))$  the full subcategory of $\mathrm{maps}(\mathrm{mod}(\mathcal{C}) )$ consisting of all  maps $(M_{1},f,M_{2})$, such that $f$ is an monomorphism.

\begin{proposition}\label{monoepifun}
The categories $\mathfrak {emaps}(\mathrm{mod}(\mathcal{C}))$ and $\mathfrak {mmaps}(\mathrm{mod}(\mathcal{C}))$ are functorially finite in $\mathrm{maps}(\mathrm{mod}(\mathcal{C})).$
\end{proposition}
\begin{proof}
The proof given in  \cite[Theorem 3.12]{MVOM} works for this setting.
\end{proof}

\footnotesize

\vskip3mm \noindent Alicia Le\'on Galeana:\\ Facultad de Ciencias, Universidad  Aut\'onoma del Estado de M\'exico\\
T\'oluca, M\'exico.\\
{\tt alicialg@hotmail.com}

\vskip3mm \noindent Martin Ort\'iz Morales:\\ Facultad de Ciencias, Universidad  Aut\'onoma del Estado de M\'exico\\
T\'oluca, M\'exico.\\
{\tt mortizmo@uaemex.mx}

\vskip3mm \noindent Valente Santiago Vargas:\\ Departamento de Matem\'aticas, Facultad de Ciencias, Universidad Nacional Aut\'onoma de M\'exico\\
Circuito Exterior, Ciudad Universitaria,
C.P. 04510, Ciudad de M\'exico, MEXICO.\\ {\tt valente.santiago.v@gmail.com}


\begin{thebibliography}{ABPRS}


\bibitem{Lidia1} L. Angeleri H\"ugel. S. Koenig, Q. Liu. {\it{Recollements and tilting objects.}} J. Pure. Appl. Algebra 215 no.4, 420-438 (2011).

\bibitem{Lidia2} L. Angeleri H\"ugel. S. Koenig, Q. Liu. {\it{On the uniqueness of stratifications of derived modules categories.}} J. Algebra, 120-137 (2012).

\bibitem{Lidia3} L. Angeleri H\"ugel. S. Koenig, Q. Liu. {\it{Jordan H\"older theorems for derived module categories of piecewise hereditary algebras.}} J. Algebra 352, 361-381 (2012).


\bibitem{Aus} M. Auslander. {\it{Representation dimension of artin algebras.}} Queen Mary College Mathematics Notes, 1971.
Republished in Selected works of Maurice Auslander: Part 1. Amer. Math. Soc., Providence (1999).


\bibitem{AuslanderRep1} M. Auslander. {\it{Representation Theory of Artin Algebras I.}} Comm. Algebra  1 (3) 177-268 (1974).

\bibitem{Auslander2} M.  Auslander and I. Reiten. {\it{Stable equivalence of dualizing $R$-varietes.}} Adv. in Math. Vol. 12, No.3, 306-366 (1974).


\bibitem{AusPlatRei} M. Auslander, M. I. Platzeck, I. Reiten. {\it{Coxeter functors without diagrams.}} Trans. Amer. Math. Soc. 250, 1-46 (1979).


\bibitem{ARS} M. Auslander, I. Reiten, S. Smal\o. {\it{Representation theory of artin algebras.}} Studies in Advanced Mathematics 36, Cambridge University Press (1995).

\bibitem{Bautista} R. Bautista, M.J. Souto Salorio, R. Zuazua. {\it{Almost split sequences for complexes of fixed size.}} J. Algebra
287 (1), 140-168, (2005).

\bibitem{Belinson} A. Beilinson, J. Bernstein, P. Deligne. {\it{Faisceaux Pervers, in: Analysis and Topology on Singular Spaces I.}} Luminy, Ast\'erisque 100, Soc. Math. France, 5-171 (1982).


\bibitem{Chen} Q. Chen, M. Zheng. {\it{Recollements of abelian categories and special types of comma categories.}} J. Algebra.  321 (9), 2474-2485 (2009).


\bibitem{DlabRingel} V. Dlab, C.M. Ringel. {\it{Representations of graphs and algebras.}} Memoirs of the A.M.S. No. 173 (1976).

\bibitem{Robert} R. M. Fossum, P. A. Griffith, I. Reiten. {\it{Trivial Extensions of Abelian Categories.}} Lecture Notes in Mathematics No. 456, Springer-Verlag, Berlin, Heidelberg, New York. (1975)

\bibitem{Franjou} V. Franjou, T. Pirashvili. {\it{ Comparison of abelian categories recollements.}} Documenta Math. 9, 41-56 (2004).

\bibitem{Gordon} R. Gordon, E. L.  Green. {\it{Modules with cores and amalgamations of indecomposable modules.}} Memoirs of the A.M.S. No. 187 (1978).

\bibitem{Green1} E. L. Green. {\it{The representation theory of tensor algebras.}} J. Algebra, 34, 136-171 (1975).

\bibitem{Green2} E. L. Green. {\it{On the representation theory of rings in matrix form.}} Pacific Journal of Mathematics. Vol. 100, No. 1, 123-138 (1982).

\bibitem{Ingmar} \O. Ingmar {\it{Quivers and admissible relations of tensor products and trivial extensions.}} Master Thesis. Norwegian University of Science and Technology.

\bibitem{Kelly} G. M. Kelly. {\it{On the radical of a category.}} J. Aust. Math. Soc. 4,  299-307 (1964).

\bibitem{Krause} H. Krause. {\it{Krull-Schmidt categories and projective covers.}} Expo. Math. 33, 535?549 (2015).

\bibitem{LeOS}  A. Le\'on-Galeana, M. Ort\'iz-Morales, V. Santiago, {\it{Triangular Matrix Categories I: Dualizing Varieties and generalized one-point extension.}} Preprin arXiv: 1903.2607226

\bibitem{Liu}  S. Liu, P. Ng, C. Paquette {\it{ Almost Split Sequences and Approximations.}}  Algebr Represent Theor 16, No. 6, 1809-1827 (2013).

\bibitem{Macpherson} R. MacPherson, K. Vilonen. {\it{Elementary construction of perverse sheaves.}} Invent. Math. 84, 403-436 (1986).

\bibitem{MVOM} R. Mart\'inez-Villa, M. Ort\'iz-Morales. {\it{ Tilting Theory and Functor Categories III: The Maps Category.}} Inter. Journal of Algebra. Vol. 5 (11), 529-561 (2011).

\bibitem{MOSS} O. Mendoza, M. Ort\'iz, E. C. S\'aenz, V. Santiago.
{\it{A generalization of the theory of standardly stratified algebras I: Standardly stratified ringoids.}}  Arxiv: 1803.01217.

\bibitem{Mitchell} B. Mitchell. {\it{Rings with Several Objects.}} Adv. in Math, Vol 8, 1-161 (1972).

\bibitem{Parshall} B. Parshall, L.L. Scott. {\it{Derived categories, quasi-hereditary algebras, and algebraic groups.}}  Proc. of the Ottawa-Moosone Workshop in algebra (1987), Math. Lect. Note Series, Carleton University and Universite d'Ottawa, 1-111 (1988).


\bibitem{Popescu} N. Popescu. {\it{Abelian categories with applications to rings and modules.}} London Mathematical Society Monographs No. 3.  Academic Press,  London, New York, N.Y (1973) MR 0340375.


\bibitem{Psaro1} C. Psaroudakis. {\it{Homological Theory of Recollements of Abelian Categories.}} J. Algebra Vol. 398,  63-110 (2014).

\bibitem{Psaro2} C. Psaroudakis, J. Vitoria. {\it{Recollements of module categories.}} J. Appl. Categor. Struc. 22, 579-593 (2014).

\bibitem{Psaro3}  C. Psaroudakis. {\it{A representation-theoretic approach to recollements of abelian categories.}} Contemp. Math. of Amer. Math. Soc. Vol. 716, 67-154 (2018).

\bibitem{Reiten} I. Reiten. {\it{The use of almost split sequences in the representation theory of artin algebras,}} Lecture Notes in Math. Vol. 944, Springer, Berling, 1982, pp. 29-104.


\bibitem{RingelTame}  C. M. Ringel. {\it{Tame Algebras and Integral Quadratic Forms,}} Lecture Notes in Mathematics, Springer-Verlag, Berlin Heidelberg (1984).

\bibitem{Yasuaki}  Y. Ogawa. {\it{Recollements for dualizing $k$-varietes and Auslander's formulas.}} Appl. Categor. Struc. (2018).
doi.org/10.1007/s10485-018-9546-y

\bibitem{Smalo} S. O. Smal\o.  {\it{Functorial Finite Subcategories Over Triangular Matrix Rings.}} Proceedings of the American Mathematical Society Vol.111.  No. 3 (1991).

\bibitem{Zhang} Pu. Zhang. {\it{Monomorphism categories, cotilting theory and Gorenstein-projective modules.}} J. Algebra, 339, 181-202, (2011).

\bibitem{Bin} B. Zhu. {\it{Triangular matrix algebras over quasi-hereditary algebras.}} Tsukuba J. Math. Vol. 25  No.1, 1-11, (2001).

\end{thebibliography}
\end{document}